\documentclass[10pt,oneside,english]{amsart}
\usepackage{mathpazo}
\usepackage[T1]{fontenc}
\usepackage[latin9]{inputenc}
\usepackage{geometry}
\usepackage{babel}
\usepackage{enumitem}
\usepackage{amsbsy}
\usepackage{amstext}
\usepackage{amsthm}
\usepackage{amssymb}
\usepackage{stmaryrd}
\usepackage{graphicx}
\usepackage[unicode=true,pdfusetitle,
 bookmarks=true,bookmarksnumbered=false,bookmarksopen=false,
 breaklinks=false,pdfborder={0 0 1},backref=false,colorlinks=false]
 {hyperref}

\makeatletter
\numberwithin{equation}{section}
\numberwithin{figure}{section}
\theoremstyle{plain}
\newtheorem{thm}{\protect\theoremname}[section]
\theoremstyle{plain}
\newtheorem{rem}[thm]{\protect\remarkname}
\theoremstyle{definition}
\newtheorem{defn}[thm]{\protect\definitionname}
\theoremstyle{plain}
\newtheorem{fact}[thm]{\protect\factname}
\theoremstyle{plain}
\newtheorem{lem}[thm]{\protect\lemmaname}
\theoremstyle{definition}
\newtheorem{problem}[thm]{\protect\problemname}
\theoremstyle{definition}
\newtheorem{example}[thm]{\protect\examplename}
\theoremstyle{remark}
\newtheorem{claim}[thm]{\protect\claimname}
\theoremstyle{plain}
\newtheorem{prop}[thm]{\protect\propositionname}
\theoremstyle{plain}
\newtheorem{cor}[thm]{\protect\corollaryname}
\theoremstyle{remark}
\newtheorem*{claim*}{\protect\claimname}

\makeatother

\providecommand{\claimname}{Claim}
\providecommand{\corollaryname}{Corollary}
\providecommand{\definitionname}{Definition}
\providecommand{\examplename}{Example}
\providecommand{\factname}{Fact}
\providecommand{\lemmaname}{Lemma}
\providecommand{\problemname}{Problem}
\providecommand{\propositionname}{Proposition}
\providecommand{\remarkname}{Remark}
\providecommand{\theoremname}{Theorem}

\begin{document}
\global\long\def\JUSTIFY#1{\mbox{\fbox{\tiny#1}}\quad}%

\global\long\def\SUPPORT{\mathrm{supp\:}}%

\global\long\def\SUPPORTPOSITIVE{\mathrm{supp}}%

\global\long\def\ESSINF{\mathrm{essinf\:}}%
\global\long\def\ESSSUP{\mathrm{esssup\:}}%

\global\long\def\DIST{d}%

\global\long\def\cutDIST{\delta_{\square}}%

\global\long\def\EXP{\mathbf{E}}%

\global\long\def\INT{\mathrm{INT}}%

\global\long\def\ACC{\mathbf{ACC}_{\mathrm{w}*}}%
\global\long\def\LIM{\mathbf{\mathbf{LIM}_{\mathrm{w}*}}}%

\global\long\def\WEAKCONV{\overset{\mathrm{w}^{*}}{\;\longrightarrow\;}}%

\global\long\def\NOTWEAKCONV{\overset{\mathrm{w^{*}}}{\;\not\longrightarrow\;}}%

\global\long\def\CUTNORMCONV{\overset{\|\cdot\|_{\square}}{\;\longrightarrow\;}}%

\global\long\def\CUTDISTCONV{\overset{\cutDIST}{\;\longrightarrow\;}}%

\global\long\def\LONECONV{\overset{\|\cdot\|_{1}}{\;\longrightarrow\;}}%

\global\long\def\NOTCUTNORMCONV{\overset{\|\cdot\|_{\square}}{\;\not\longrightarrow\;}}%

\global\long\def\NOTDISTCONV{\overset{\cutDIST}{\;\not\longrightarrow\;}}%

\global\long\def\NOTLONECONV{\overset{\|\cdot\|_{1}}{\;\not\longrightarrow\;}}%

\global\long\def\SO{\stackrel{S}{\preceq}}%

\global\long\def\SSO{\stackrel{S}{\prec}}%

\global\long\def\GRAPHONSPACE{\mathcal{W}_{0}}%

\global\long\def\UNLABELLEDGRAPHONSPACE{\widetilde{\mathcal{W}}_{0}}%

\title{Relating the cut distance and the weak{*} topology for graphons}
\author{Martin Doležal, Jan Grebík, Jan Hladký, Israel Rocha, Václav Rozho\v{n}}
\address{\emph{Doležal:} Institute of Mathematics, Czech Academy of Sciences.
Žitná~25, 110~00, Praha, Czechia. With institutional support RVO:67985840\emph{}\\
\emph{Grebík:} Mathematics Institute, University of Warwick, Coventry,
CV4~7AL, UK.\emph{ This work was done while affiliated with:} Institute
of\emph{ }Mathematics, Czech Academy of Sciences. Žitná~25, 110~00,
Praha, Czechia. With institutional support RVO:67985840\emph{}\\
\emph{Hladký:} Institute of Mathematics, Czech Academy of Sciences.
Žitná~25, 110~00, Praha, Czechia\emph{. Part of this work was done
while affiliated with:} Institut für Geometrie, TU Dresden, 01062
Dresden, Germany\emph{}\\
\emph{Rocha}: Institute of Computer Science, Czech Academy of Sciences.
Pod Vodárenskou v\v{e}ží 2, 182~07, Prague, Czechia. With institutional
support RVO:67985807.\\
\emph{Rozho\v{n}: }ETH Zürich, Switzerland.\emph{ This work was done
while affiliated with:} Institute of Computer Science, Czech Academy
of Sciences. Pod Vodárenskou v\v{e}ží 2, 182~07, Prague, Czechia.
With institutional support RVO:67985807.}
\thanks{Research of \emph{Doležal} was supported by the GA\v{C}R project No.
17-27844S and RVO: 67985840.\emph{ Grebík} was supported by the GA\v{C}R
project 17-33849L and RVO: 67985840.\emph{ Hladký} was supported by
the Alexander von Humboldt Foundation\emph{ }and by the GA\v{C}R project
No. 18-01472Y and RVO: 67985840\emph{. Rocha }and \emph{Rozho\v{n}}
were supported by the Czech Science Foundation, grant number GJ16-07822Y}
\email{dolezal@math.cas.cz, greboshrabos@seznam.cz, honzahladky@gmail.com,
israelrocha@gmail.com, rozhonv@ethz.ch }
\begin{abstract}
The theory of graphons is ultimately connected with the so-called
cut norm. In this paper, we approach the cut norm topology via the
weak{*} topology (when considering a predual of $L^{1}$-functions).
We prove that a sequence $W_{1},W_{2},W_{3},\ldots$ of graphons converges
in the cut distance if and only if we have equality of the sets of
weak{*} accumulation points and of weak{*} limit points of all sequences
of graphons $W_{1}',W_{2}',W_{3}',\ldots$ that are weakly isomorphic
to $W_{1},W_{2},W_{3},\ldots$. We further give a short descriptive
set theoretic argument that each sequence of graphons contains a subsequence
with the property above. This in particular provides an alternative
proof of the theorem of Lovász and Szegedy about compactness of the
space of graphons. We connect these results to <<multiway cut>>
characterization of cut distance convergence from {[}Ann. of Math.
(2) 176 (2012), no. 1, 151-219{]}.

These results are more naturally phrased in the Vietoris hyperspace
$K(\GRAPHONSPACE)$ over graphons with the weak{*} topology. We show
that graphons with the cut distance topology are homeomorphic to a
closed subset of $K(\GRAPHONSPACE)$, and deduce several consequences
of this fact.

From these concepts a new order on the space of graphons emerges.
This order allows to compare how structured two graphons are. We establish
basic properties of this <<structuredness order>>.
\end{abstract}

\maketitle

\section{Introduction\label{sec:Intro}}

Graphons emerged from the work of Borgs, Chayes, Lovász, Sós, Szegedy,
and Vesztergombi~\cite{Lovasz2006,Borgs2008c} on limits of sequences
of finite graphs. We write $\GRAPHONSPACE$ for the space of all \emph{graphons},
i.e., all symmetric measurable functions from $\Omega^{2}$ to $[0,1]$,
after identifying graphons that are equal almost everywhere. Here
as well as in the rest of the paper, $\Omega$ is an arbitrary separable
atomless probability space with probability measure $\nu$. While
it is meaningful to investigate the space $\GRAPHONSPACE$ with respect
to several metrics and topologies, the two that relate the most to
graph theory are the metrics $\DIST_{\square}$ and $\delta_{\square}$
based on the so-called \emph{cut norm} defined on the space $L^{1}(\Omega^{2})$
by\emph{
\[
\left\Vert Y\right\Vert _{\square}=\sup_{S,T\subseteq\Omega}\left|\int_{S\times T}Y\right|\quad\text{for each \ensuremath{Y\in L^{1}(\Omega^{2})}\:.}
\]
}Given $U,W\in\GRAPHONSPACE$ we set 
\begin{align}
\DIST_{\square}\left(U,W\right): & =\left\Vert U-W\right\Vert _{\square}=\sup_{S,T\subseteq\Omega}\left|\int_{S\times T}U-\int_{S\times T}W\right|\;,\text{ and}\nonumber \\
\cutDIST(U,W) & :=\inf_{\varphi}\DIST_{\square}\left(U,W^{\varphi}\right)\;,\label{eq:defcutdist}
\end{align}
where $\varphi$ ranges over all measure preserving bijections of
$\Omega$ and the graphon $W^{\varphi}$ is defined by 
\begin{equation}
W^{\varphi}(x,y)=W(\varphi(x),\varphi(y))\;.\label{eq:version}
\end{equation}
We call $d_{\square}$ the \emph{cut norm distance} and $\delta_{\square}$
the \emph{cut distance}. We call graphons of the form $W^{\varphi}$
\emph{versions} of $W$. Passing to a version is an infinitesimal
counterpart to considering another adjacency matrix of a graph, in
which the vertices are reordered.

The key property of the space $\GRAPHONSPACE$ is its compactness
with respect to the cut distance $\cutDIST$. The result was first
proven by Lovász and Szegedy~\cite{Lovasz2006} using the regularity
lemma,\footnote{see also~\cite{Lovasz2007} and~\cite{MR3425986} for variants of
this approach} and then by Elek and Szegedy~\cite{ElekSzegedy} using ultrafilter
techniques, by Austin~\cite{MR2426176} and Diaconis and Janson~\cite{MR2463439}
using the theory of exchangeable random graphs, and finally by Doležal
and Hladký~\cite{DH:WeakStar} by optimizing a suitable parameter
over the set of weak{*} limits.
\begin{thm}
\label{thm:compactness}For every sequence $\Gamma_{1},\Gamma_{2},\Gamma_{3},\ldots$
of graphons there is a subsequence $\Gamma_{n_{1}},\Gamma_{n_{2}},\Gamma_{n_{3}},\ldots$
and a graphon $\Gamma$ such that $\cutDIST(\Gamma_{n_{i}},\Gamma)\rightarrow0$.
\end{thm}

\subsection{Overview of the results\label{subsec:Overview-of-the}}

We view graphons as functions in the Banach space $L^{\infty}(\Omega^{2})$,
with a predual Banach space $L^{1}(\Omega^{2})$ to which we associate
the concept of weak{*} convergence. That means that a sequence of
graphons $\Gamma_{1},\Gamma_{2},\Gamma_{3},\ldots$ \emph{converges
weak{*} }to a graphon $W$ if we have
\begin{equation}
\sup_{Q\subseteq\Omega^{2}}\;\lim_{n\rightarrow\infty}\int_{Q}\Gamma_{n}-\int_{Q}W=0\;.\label{eq:weakstarfull}
\end{equation}
Since the sigma-algebra of all measurable sets $Q\subseteq\Omega^{2}$
is generated by sets of the form $S\times T$, where $S$ and $T$
are measurable subsets of $\Omega$, we can equivalently rewrite~(\ref{eq:weakstarfull})
in a way which is more convenient for us,
\begin{equation}
\sup_{S,T\subseteq\Omega}\;\lim_{n\rightarrow\infty}\int_{S\times T}\Gamma_{n}-\int_{S\times T}W=0\;.\label{eq:panle1}
\end{equation}
The weak{*} topology is weaker than the topology generated by $\DIST_{\square}$,
of which the former can be viewed as a certain uniformization. Indeed,
recall that a sequence of graphons $\Gamma_{1},\Gamma_{2},\Gamma_{3},\ldots$
converges to $W$ in the cut norm if
\begin{equation}
\lim_{n\rightarrow\infty}\;\sup_{S,T\subseteq\Omega}\left\{ \int_{S\times T}\Gamma_{n}-\int_{S\times T}W\right\} =0\;,\label{eq:panle2}
\end{equation}
that is, (\ref{eq:panle1}) and (\ref{eq:panle2}) differ only in
the order of the limit and the supremum.

In Section~\ref{sec:WeakStarANDCutDist} and Section~\ref{sec:structurdness},
which we consider the main contribution of the paper, we show that
the interplay between the cut distance and weak{*} topology creates
a rich theory.

In particular, we can make use of the weak{*} convergence to prove
Theorem~\ref{thm:compactness}. To this end, we look at the set $\ACC\left(\Gamma_{1},\Gamma_{2},\Gamma_{3},\ldots\right)$
of all weak{*} accumulation points of sequences 
\[
\left\{ \Gamma'_{1},\Gamma'_{2},\Gamma'_{3},\ldots:\text{\ensuremath{\Gamma_{n}'} is a version of \ensuremath{\Gamma_{n}}}\right\} \;.
\]
Similarly, denote by $\LIM(\Gamma_{1},\Gamma_{2},\Gamma_{3},\ldots)$
the set of all graphons $W$ for which there exist versions $\Gamma'_{1},\Gamma'_{2},\Gamma'_{3},\ldots$
of $\Gamma_{1},\Gamma_{2},\Gamma_{3},\ldots$ such that $W$ is a
weak{*} limit of the sequence $\Gamma'_{1},\Gamma'_{2},\Gamma'_{3},\ldots$.
Note that equivalently, we could have required $\Gamma'_{1},\Gamma'_{2},\Gamma'_{3},\ldots$
to be weakly isomorphic to $\Gamma_{1},\Gamma_{2},\Gamma_{3},\ldots$,
rather than being versions of $\Gamma_{1},\Gamma_{2},\Gamma_{3},\ldots$.
The set $\ACC\left(\Gamma_{1},\Gamma_{2},\Gamma_{3},\ldots\right)$
is non-empty by the Banach\textendash Alaoglu Theorem. In the set
$\ACC\left(\Gamma_{1},\Gamma_{2},\Gamma_{3},\ldots\right)$ we cleverly
select one graphon $\Gamma$. The selection is done so that in addition
to being a weak{*} accumulation point, $\Gamma$ is also a cut distance
accumulation point (the latter being clearly a stronger property).
In~\cite{DH:WeakStar}, Doležal and Hladký carried out a similar
program\footnote{which in actuality is more complicated for reasons sketched in Section~\ref{subsec:DifferentiatingACCLIM}}
where they showed that for the <<clever selection>> we can take
$\Gamma$ as the maximizer of an arbitrary graphon parameter of the
form 
\begin{equation}
\INT_{f}(W):=\int_{x}\int_{y}f\left(W(x,y)\right)\label{eq:FDolHl}
\end{equation}
where $f:[0,1]\rightarrow\mathbb{R}$ is a fixed but arbitrary continuous
strictly convex function.\footnote{\label{fn:INT}Most of~\cite{DH:WeakStar} deals with \emph{minimizing}
$\INT_{f}(W)$ for a fixed continuous strictly \emph{concave} function.
This is obviously equivalent. Also, note that in~\cite{DGHRR:Parameters}
it is shown that the assumption of continuity is not needed.}

Our main result, Theorem~\ref{thm:hyperspaceANDcutDISTANCEsimpler},
says that a sequence of graphons $\Gamma_{1},\Gamma_{2},\Gamma_{3},\ldots$
is cut distance convergent if and only if $\ACC\left(\Gamma_{1},\Gamma_{2},\Gamma_{3},\ldots\right)=\LIM\left(\Gamma_{1},\Gamma_{2},\Gamma_{3},\ldots\right)$.
This is complemented by Theorem~\ref{thm:subsequenceLIMACC} which
says that from any sequence $\Gamma_{1},\Gamma_{2},\Gamma_{3},\ldots$
of graphons, we can choose a subsequence $\Gamma_{n_{1}},\Gamma_{n_{2}},\Gamma_{n_{3}},\ldots$
such that $\LIM\left(\Gamma_{n_{1}},\Gamma_{n_{2}},\Gamma_{n_{3}},\ldots\right)=\ACC\left(\Gamma_{n_{1}},\Gamma_{n_{2}},\Gamma_{n_{3}},\ldots\right)$.
In particular, this yields Theorem~\ref{thm:compactness}. In Section~\ref{sec:Multiway}
we show that Theorem~\ref{thm:hyperspaceANDcutDISTANCEsimpler} is
actually equivalent to one of the main results of~\cite{MR2925382}
on so-called <<multiway cuts>>.

It turns out that these results can be naturally phrased in terms
of the so-called Vietoris hyperspace $K(\GRAPHONSPACE)$ over graphons
with the weak{*} topology (see Section~\ref{subsec:PreliminaryHyperspace}
for definitions). To each graphon $W:\Omega^{2}\rightarrow\left[0,1\right]$,
we associate its \emph{envelope} $\left\langle W\right\rangle =\ACC(W,W,W,\ldots)$
which is a subset of $L^{\infty}\left(\Omega^{2}\right)$. We show
that cut distance convergence of graphons is equivalent to convergence
of the corresponding envelopes in $K(\GRAPHONSPACE)$, and that envelopes
form a closed set in $K(\GRAPHONSPACE)$. As $K(\GRAPHONSPACE)$ is
known to be compact, this connection in particular provides an alternative
proof of Theorem~\ref{thm:compactness}. However, the transference
between the space of graphons and $K(\GRAPHONSPACE)$ has other applications.

From these proofs a new partial order on the space of graphons naturally
emerges. We say that $U$ is \emph{more structured} than $W$ if $\left\langle U\right\rangle \supsetneq\left\langle W\right\rangle $,
and write $U\succ W$. One illustrative example is to take $U$ to
be a complete balanced bipartite graphon and $W\equiv\frac{1}{2}$.
Obviously, $\left\langle W\right\rangle =\left\{ W\right\} $. By
considering versions of $U$ which create finer and finer chessboards,
we have $\left\langle U\right\rangle \ni W$. Thus, $U\succeq W$.
We establish basic properties of this <<structuredness order>>.
We investigate these properties in Section~\ref{sec:BasicPropertiesOfStructurdness}.
In particular, in Proposition~\ref{prop:minimalAndMaximalElements}
we characterize minimal and maximal elements. In Section~\ref{subsec:ValuesAndDegrees},
we introduce the \emph{range frequencies }of a graphon \emph{$W$}
as a pushforward probability measure on $\left[0,1\right]$ defined
by
\[
\boldsymbol{\Phi}_{W}\left(A\right):=\nu^{\otimes2}\left(W^{-1}(A)\right)\,,
\]
for a measurable set $A\subseteq\left[0,1\right]$, and introduce
a certain \emph{flatness order} on these range frequencies. Roughly
speaking, one probability measure on $\left[0,1\right]$ is flatter
than another, if the former can be obtained by a certain sort of averaging
of the latter. As we show in Proposition~\ref{prop:flatter}, this
order is compatible with the structuredness order on the corresponding
graphons. In~\cite{DGHRR:Parameters} we use range frequencies to
reprove in a very quick way the result of Doležal and Hladký (even
for discontinuous functions, as mentioned in Footnote~\ref{fn:INT}).

The main motivation for introducing and studying the structuredness
order is its connection to Theorem~\ref{thm:compactness}. Indeed,
the <<clever selection>> in our proof of Theorem~\ref{thm:compactness}
is to take a maximal element (inside $\ACC\left(\Gamma_{1},\Gamma_{2},\Gamma_{3},\ldots\right)$)
with respect to the structuredness order. However, the abstract theory
of the weak{*} approach to cut norm convergence which we introduce
in this paper has also applications in classical problems in extremal
graph theory. More specifically, one of the main results which we
prove in~\cite{DGHRR:Parameters} says that if a graph is <<step
Sidorenko>> then it is <<weakly norming>>, thus answering an open
question from~\cite{KMPW:StepSidorenko}. Our proof in~\cite{DGHRR:Parameters}
relies heavily on the theory of the structuredness order introduced
in the current paper, and we do not see an alternative proof which
would avoid it.\footnote{See Remark~3.29 of~\cite{DGHRR:Parameters}.}

\section{Preliminaries}

\subsection{General notation}

We write $\stackrel{\varepsilon}{\approx}$ for equality up to $\varepsilon$.
For example, $1\stackrel{0.2}{\approx}1.1\stackrel{0.2}{\approx}1.3$.
We write $P_{k}$ for a path on $k$ vertices and $C_{k}$ for a cycle
on $k$ vertices. 

If $A$ and $B$ are measure spaces then we say that a map $f:A\rightarrow B$
is an \emph{almost-bijection} if there exist measure zero sets $A_{0}\subseteq A$
and $B_{0}\subseteq B$ so that $f_{\restriction A\setminus A_{0}}$
is a bijection between $A\setminus A_{0}$ and $B\setminus B_{0}$.
Note that in~(\ref{eq:defcutdist}), we could have worked with measure
preserving almost-bijections $\varphi$ instead.

\subsection{Graphon basics}

Our notation is mostly standard, following~\cite{Lovasz2012}. We
write $\GRAPHONSPACE$ for the space of all graphons, that is, symmetric
measurable functions from $\Omega^{2}$ to $[0,1]$, modulo differences
on null-sets.

Graphons $U$ and $W$ are called \emph{weakly isomorphic} if $\cutDIST(U,W)=0$.
Note that in this case there is not necessarily a measure preserving
bijection $\varphi$ for which $\DIST_{\square}\left(U,W^{\varphi}\right)=0$;
see~\cite[Figure 7.1]{Lovasz2012}. In other words, being versions
and being weakly isomorphic are two slightly different notions. Let
us denote the compact space of graphons after the weak isomorphism
factorization as $\UNLABELLEDGRAPHONSPACE$. For every $W\in\GRAPHONSPACE$
we denote its equivalence class $\left\llbracket W\right\rrbracket \in\UNLABELLEDGRAPHONSPACE$.

The probability measure underlying $\Omega$ is $\nu$. We write $\nu^{\otimes k}$
for the product measure on $\Omega^{k}$.
\begin{rem}
\label{rem:UnitIntervalGroundSpace}Every separable atomless probability
space is isomorphic to the unit interval with the Lebesgue measure.
While most of our arguments are abstract and work with an arbitrary
separable atomless probability space $\Omega$, there are some other,
where we will assume that graphons are defined on the square of the
unit interval, and then will make use of the usual order on $[0,1]$.
\end{rem}

If $W\colon\Omega^{2}\rightarrow[0,1]$ is a graphon and $\varphi,\psi$
are two measure preserving bijections of $\Omega$ then we use the
short notation $W^{\psi\varphi}$ for the graphon $W^{\psi\circ\varphi}$,
i.e. $W^{\psi\varphi}(x,y)=W(\psi(\varphi(x)),\psi(\varphi(y)))=W^{\psi}(\varphi(x),\varphi(y))=\left(W^{\psi}\right)^{\varphi}\left(x,y\right)$
for $(x,y)\in\Omega^{2}$. Thus we have a right action of the group
of all measure preserving isomorphisms of $\left[0,1\right]$ on $\GRAPHONSPACE$.

The graphons that take values $0$ or $1$ almost everywhere are called
\emph{0-1 valued graphons}. 

We call the quantity $\int_{x}\int_{y}W(x,y)$ the \emph{edge density
of $W$}. Recall also that for $x\in\Omega$, we have the \emph{degree
of $x$} in $W$ defined as $\deg_{W}(x)=\int_{y}W\left(x,y\right)$.
Recall that measurability of $W$ gives that $\deg_{W}(x)$ exists
for almost each $x\in\Omega$. We say that $W$ is \emph{$p$-regular
}if for almost every $x\in\Omega$, $\deg_{W}(x)=p$.

\subsubsection{The stepping operator}

Suppose that \emph{$W:\Omega^{2}\rightarrow[0,1]$} is a graphon.
We say that $W$ is a\emph{ step graphon} if $W$ is constant on each
$\Omega_{i}\times\Omega_{j}$, for a suitable a finite partition $\mathcal{P}$
of $\Omega$, $\mathcal{P}=\left\{ \Omega_{1},\Omega_{2},\ldots,\Omega_{k}\right\} $.
We recall the definition of the stepping operator.
\begin{defn}
Suppose that $\Gamma:\Omega^{2}\rightarrow[0,1]$ is a graphon. For
a finite partition $\mathcal{P}$ of $\Omega$, $\mathcal{P}=\left\{ \Omega_{1},\Omega_{2},\ldots,\Omega_{k}\right\} $,
we define a graphon $\Gamma^{\Join\mathcal{P}}$ by setting it on
the rectangle $\Omega_{i}\times\Omega_{j}$ to be the constant $\frac{1}{\nu^{\otimes2}(\Omega_{i}\times\Omega_{j})}\int_{\Omega_{i}}\int_{\Omega_{j}}\Gamma(x,y)$.
We allow graphons to have not well-defined values on null sets which
handles the cases $\nu(\Omega_{i})=0$ or $\nu(\Omega_{j})=0$.
\end{defn}

In~\cite{Lovasz2012}, a stepping is denoted by $\Gamma_{\mathcal{P}}$
rather than $\Gamma^{\Join\mathcal{P}}$.

Finally, we say that a graphon $U$ \emph{refines }a graphon $W$,
if $W$ is a step graphon for a suitable partition $\mathcal{P}$
of $\Omega$, $\mathcal{P}=\left\{ \Omega_{1},\Omega_{2},\ldots,\Omega_{k}\right\} $,
and $U^{\Join\mathcal{P}}=W$.

\subsection{Topologies on $\protect\GRAPHONSPACE$\label{subsec:TopologiesOnW}}

There are several natural topologies on $\GRAPHONSPACE$. The $\left\Vert \cdot\right\Vert _{1}$
topology inherited from the normed space $L^{1}(\Omega^{2})$, the
topology given by the $\left\Vert \cdot\right\Vert _{\square}$ norm,
and the weak{*} topology (when $\mathcal{W}_{0}$ is viewed as a subset
of the dual space of $L^{1}(\Omega^{2})$ i.e. a subset of $L^{\infty}(\Omega^{2})$).
Note that $\GRAPHONSPACE$ is closed in $L^{1}$. We write $d_{1}\left(\cdot,\cdot\right)$
for the distance derived from the $\left\Vert \cdot\right\Vert _{1}$
norm. Recall also that by the Banach\textendash Alaoglu Theorem, $\GRAPHONSPACE$
equipped with the weak{*} topology is compact and that the weak{*}
topology on $\GRAPHONSPACE$ is metrizable. We shall denote by $d_{\mathrm{w}^{*}}(\cdot,\cdot)$
any metric compatible with this topology. For example, we can take
some countable dense measure subalgebra $\left\{ A_{n}\right\} _{n\in\mathbb{N}}$
of all measurable subsets of $\Omega$, and define 
\begin{equation}
d_{\mathrm{w}^{*}}\left(U,W\right):=\sum_{n,k\in\mathbb{N}}2^{-(n+k)}\left|\int_{A_{n}\times A_{k}}(U-W)\;\mathrm{d}\nu\right|.\label{eq:exampleweakstarmetric}
\end{equation}

The following fact summarizes the relation of the above topologies.
\begin{fact}
\label{fact:differentnorms}The following identity maps are continuous:
$(\GRAPHONSPACE,d_{1})\to(\GRAPHONSPACE,d_{\Box})\to(\GRAPHONSPACE,d_{\mathrm{w}^{*}})$.
\end{fact}

\begin{proof}
The continuity of the first map is an easy consequence of the definitions
and the second is explained in Subsection \ref{subsec:Overview-of-the}. 
\end{proof}

\subsection{Auxiliary results about $L^{1}$-spaces}

We prove two auxiliary lemmas about $L^{1}$-spaces. Lemma~\ref{lem:lonenonconvergence}
is an easy result about functions that do not converge in $L^{1}$.
\begin{lem}
\label{lem:lonenonconvergence}Suppose that $\Lambda$ is a probability
measure space with measure $\lambda$. If we have functions $g,g_{1},g_{2},g_{3},\ldots:\Lambda\rightarrow[0,1]$
for which $g_{n}\NOTLONECONV g$, then there exists an interval $J\subseteq[0,1]$
and a number $c>0$ such that for the interval $J^{+}:=\left\{ x+d:x\in J,d\in[-c,c]\right\} $
we have $\lambda\left(g^{-1}\left(J\right)\setminus g_{n}^{-1}\left(J^{+}\right)\right)\not\longrightarrow0$.
\end{lem}

\begin{proof}
By passing to a subsequence, we may assume that there exists a constant
$\varepsilon>0$ so that for each $n$, $\left\Vert g_{n}-g\right\Vert _{1}>\varepsilon$.
Take $k:=\left\lceil 4/\varepsilon\right\rceil $ and a partition
of $[0,1]$ into $k$ intervals $J_{1},J_{2},\ldots,J_{k}$ of lengths
at most $\frac{\varepsilon}{4}$ (ordered from left to right; it is
not important if they are open, closed, or semiopen). For $j\in[k]$,
$L_{j}:=g^{-1}(J_{j})$. For each $n$, there exists a number $j(n)\in[k]$
so that we have $\int_{L_{j(n)}}|g_{n}-g|>\varepsilon\cdot\lambda(L_{j(n)})$.
Observe that the strict inequality forces that $\lambda(L_{j(n)})>0$.
In particular, we then have that 
\begin{equation}
|g_{n}(x)-g(x)|>\frac{\varepsilon}{2}\label{eq:novinar}
\end{equation}
for a set of points $x\in L_{j(n)}$ of measure at least $\frac{\varepsilon}{2}\cdot\lambda\left(L_{j(n)}\right)$.
Let $j$ be a number that repeats infinitely often in the sequence
$j(1),j(2),j(3),\ldots$. By passing to a subsequence once again,
we can assume that $j=j(1)=j(2)=\ldots$. We set $J:=J_{j}$, $c:=\frac{\varepsilon}{4}$,
and $J^{+}$ as in the statement of the lemma. Observe that whenever
$x\in L_{j}$ satisfies~(\ref{eq:novinar}), then $g_{n}(x)\notin J^{+}$.
Therefore, we conclude that for each $n$, $\lambda\left(g^{-1}\left(J\right)\setminus g_{n}^{-1}\left(J^{+}\right)\right)\ge\frac{\varepsilon}{2}\cdot\lambda\left(L_{j}\right)$.
This concludes the proof.
\end{proof}
We can now state the second lemma of this section.
\begin{lem}
\label{lem:approx}For every graphon $\Gamma:\Omega^{2}\rightarrow[0,1]$
and every $\varepsilon>0$ there exists a finite partition $\mathcal{P}$
of $\Omega$ such that $\left\Vert \Gamma-\Gamma^{\Join\mathcal{P}}\right\Vert _{1}<\varepsilon$.
\end{lem}

For the proof of Lemma~\ref{lem:approx}, the following fact will
be useful.
\begin{fact}
\label{fact:replaceapprox}Suppose that $f\in L^{1}(\Lambda)$ is
an arbitrary function on a finite measure space $\Lambda$ with measure
$\lambda$. Set $a:=\frac{1}{\lambda(\Lambda)}\cdot\int_{\Lambda}f$.
Then for each $b\in\mathbb{R}$ we have that $\left\Vert f-a\right\Vert _{1}\le2\left\Vert f-b\right\Vert _{1}$.
\end{fact}

\begin{proof}
We have
\begin{align*}
\left\Vert f-a\right\Vert _{1} & =\int_{\Lambda}\left|f(x)-a\right|\le\int_{\Lambda}\left|f(x)-b\right|+\int_{\Lambda}\left|a-b\right|=\left\Vert f-b\right\Vert _{1}+\lambda(\Lambda)\cdot|a-b|\\
 & =\left\Vert f-b\right\Vert _{1}+\left|\int_{\Lambda}(f(x)-b)\right|\le2\left\Vert f-b\right\Vert _{1}\;.
\end{align*}
\end{proof}
\begin{proof}[Proof of Lemma~\ref{lem:approx}]
Since sets of the form $A\times B$, where $A,B\subseteq\Omega$
are measurable sets, generate the product sigma-algebra on $\Omega^{2}$,
there exists a finite partition $\mathcal{P}$ of $\Omega$ and a
function $S:\Omega^{2}\rightarrow\mathbb{R}$ such that $S$ is constant
on each rectangle of $\mathcal{P}\times\mathcal{P}$, and such that
$\left\Vert \Gamma-S\right\Vert _{1}<\frac{\varepsilon}{2}$. Now,
for each rectangle $(A,B)\in\mathcal{P}\times\mathcal{P}$, we apply
Fact~\ref{fact:replaceapprox} on the restricted function $\Gamma_{\restriction A\times B}$
and the constant $S_{\restriction A\times B}$. Summing up the contributions
coming from these applications of Fact~\ref{fact:replaceapprox},
we get that $\left\Vert \Gamma-\Gamma^{\Join\mathcal{P}}\right\Vert _{1}\le2\left\Vert \Gamma-S\right\Vert _{1}<\varepsilon$.
\end{proof}
We call \emph{$\Gamma^{\Join\mathcal{P}}$ }with properties as in
Lemma~\ref{lem:approx} \emph{averaged $L^{1}$-approximation of
$\Gamma$ by a step-graphon for precision $\varepsilon$}.

\subsection{Hyperspace $K(\protect\GRAPHONSPACE)$\label{subsec:PreliminaryHyperspace}}

Let $X$ be a metrizable compact space. We denote as $K(X)$ the space
of all compact subsets of $X$ with the topology generated by sets
of the form $\{L\in K(X):L\subseteq U\}$ and $\{L\in K(X):L\cap U\not=\emptyset\}$
where $U\subseteq X$ ranges over all open sets of $X$. Then $K(X)$
is called the \emph{hyperspace of $X$ with the Vietoris topology.}
\begin{fact}[(4.22) and (4.26) in~\cite{MR1321597}]
\label{fact:VietorisCompact}Let $X$ be a metrizable compact space
with compatible metric $\rho$. Then $K(X)$ is metrizable compact
(and hence separable). Furthermore, the Hausdorff metric on $K(X)$,
\begin{equation}
d_{\rho}^{\mathrm{Hf}}(L,M)=\max\left\{ \max_{x\in L}\left\{ \rho(x,M)\right\} ,\max_{y\in M}\left\{ \rho(y,L)\right\} \right\} \label{eq:HausdorffMetric}
\end{equation}
is compatible with the Vietoris topology on $K(X)$.
\end{fact}

\begin{rem}
\label{rem:interestedhyperspace}We will be interested in the situation
where $X=\GRAPHONSPACE$ is endowed with the weak{*} topology. By
the discussion in Section~\ref{subsec:TopologiesOnW}, $X$ is indeed
metrizable compact.
\end{rem}

\section{Weak{*} convergence and the cut distance\label{sec:WeakStarANDCutDist}}

\subsection{Becoming familiar with $\protect\ACC(\Gamma_{1},\Gamma_{2},\Gamma_{3},\ldots)$
and $\protect\LIM(\Gamma_{1},\Gamma_{2},\Gamma_{3},\ldots)$\label{subsec:BecomingFamiliar}}

Let us observe some basic properties of the sets $\LIM(\Gamma_{1},\Gamma_{2},\Gamma_{3},\ldots)$
and $\ACC(\Gamma_{1},\Gamma_{2},\Gamma_{3},\ldots)$. We have $\LIM(\Gamma_{1},\Gamma_{2},\Gamma_{3},\ldots)\subseteq\ACC(\Gamma_{1},\Gamma_{2},\Gamma_{3},\ldots)$.
The set $\LIM(\Gamma_{1},\Gamma_{2},\Gamma_{3},\ldots)$ can be empty
(for example when $\Gamma_{1}\equiv0,\Gamma_{2}\equiv1,\Gamma_{3}\equiv0,\Gamma_{4}\equiv1,\ldots$)
but $\ACC(\Gamma_{1},\Gamma_{2},\Gamma_{3},\ldots)$ is non-empty
by the Banach\textendash Alaoglu Theorem. Actually, we can describe
some elements of $\ACC(\Gamma_{1},\Gamma_{2},\Gamma_{3},\ldots)$
fairly easily. Let $T\subseteq[0,1]$ be the set of the accumulation
points of the edge densities of the graphons $\Gamma_{1},\Gamma_{2},\Gamma_{3},\ldots$,
i.e., $T$ is the set of the accumulation points of the sequence $\left(\int_{x}\int_{y}\Gamma_{n}(x,y)\right)_{n}$.
Now, a constant $c\in[0,1]$ (viewed as a constant graphon) lies in
$\ACC(\Gamma_{1},\Gamma_{2},\Gamma_{3},\ldots)$ if and only if $c\in T$.
The direction that $c\in\ACC(\Gamma_{1},\Gamma_{2},\Gamma_{3},\ldots)$
implies $c\in T$ is obvious. Now, suppose that $c\in T$. That is,
for some subsequence $\Gamma_{n_{1}},\Gamma_{n_{2}},\Gamma_{n_{3}},\ldots$
the densities converge to $c$. Partition each $\Gamma_{n_{i}}$ into
$i$ sets of measure $\frac{1}{i}$ and consider a version $\widehat{\Gamma_{n_{i}}}$
of $\Gamma_{n_{i}}$ obtained by a measure preserving bijection permuting
these sets randomly. Then almost surely, $\widehat{\Gamma_{n_{1}}},\widehat{\Gamma_{n_{2}}},\widehat{\Gamma_{n_{3}}},\ldots$
weak{*} converge to $c$. This is included here just to get familiar
with $\ACC(\Gamma_{1},\Gamma_{2},\Gamma_{3},\ldots)$ and a proof
is not needed at this point. However, the statement follows from Lemma~\ref{lem:basicpropofenvelopes}\ref{enu:averagingINSIDE}.

The first non-trivial fact we will prove about the set $\LIM(\Gamma_{1},\Gamma_{2},\Gamma_{3},\ldots)$
is that it is closed.
\begin{lem}
\label{lem:LIMclosed}Let $\Gamma_{1},\Gamma_{2},\Gamma_{3},\ldots\in\GRAPHONSPACE$
be a sequence of graphons. Then the following holds for the set $\LIM(\Gamma_{1},\Gamma_{2},\Gamma_{3},\ldots)$.
\begin{enumerate}[label=(\alph*)]
\item \label{enu:weakstar}$\LIM(\Gamma_{1},\Gamma_{2},\Gamma_{3},\ldots)$
is weak{*} closed in $L^{\infty}\left(\Omega^{2}\right)$. 
\item \label{enu:weakstarcompact} $\LIM(\Gamma_{1},\Gamma_{2},\Gamma_{3},\ldots)$
is weak{*} compact in $L^{\infty}\left(\Omega^{2}\right)$.
\item \label{enu:closedL1}$\LIM(\Gamma_{1},\Gamma_{2},\Gamma_{3},\ldots)$
is closed in $L^{1}\left(\Omega^{2}\right)$.
\end{enumerate}
\end{lem}

\begin{proof}[Proof of Part~\ref{enu:weakstar}]
 Suppose that $L_{1},L_{2},L_{3},\ldots$ are elements of $\LIM(\Gamma_{1},\Gamma_{2},\Gamma_{3},\ldots)$
such that $L_{k}\WEAKCONV L$ for $k\rightarrow\infty$. For every
$k$ let $\Gamma_{1}^{k},\Gamma_{2}^{k},\Gamma_{3}^{k},\ldots$ be
a sequence of versions of $\Gamma_{1},\Gamma_{2},\Gamma_{3},\ldots$
converging to $L_{k}$. We find an increasing sequence $i_{1},i_{2},i_{3},\ldots$
of integers such that for every $k$ and for every $n\geq i_{k}$
we have $d_{\mathrm{w}^{*}}(\Gamma_{n}^{k+1},L_{k+1})<\tfrac{1}{k}$.
Then the following sequence of versions of $\Gamma_{1},\Gamma_{2},\Gamma_{3},\ldots$
weak{*} converges to $L$: 
\[
\Gamma_{1}^{1},\Gamma_{2}^{1},\ldots,\Gamma_{i_{1}-1}^{1},\Gamma_{i_{1}}^{2},\Gamma_{i_{1}+1}^{2},\ldots,\Gamma_{i_{2}-1}^{2},\Gamma_{i_{2}}^{3},\Gamma_{i_{2}+1}^{3},\ldots,\Gamma_{i_{3}-1}^{3},\ldots.
\]

\emph{Proof of Part~\ref{enu:weakstarcompact}: }Recall that the
closed unit ball is compact in the weak{*} topology. Since $\LIM(\Gamma_{1},\Gamma_{2},\Gamma_{3},\ldots)$
lies in this ball, it is weak{*} compact.

\emph{Proof of Part~\ref{enu:closedL1}:} The unit ball $B$ of $L^{\infty}\left(\Omega^{2}\right)$
is closed in $L^{1}\left(\Omega^{2}\right)$. $\LIM(\Gamma_{1},\Gamma_{2},\Gamma_{3},\ldots)$
is a weak{*} closed subset of $B$, and so it is also closed in $B$
in the topology inherited from $L^{1}\left(\Omega^{2}\right)$ (by
Fact~\ref{fact:differentnorms}). So, $\LIM(\Gamma_{1},\Gamma_{2},\Gamma_{3},\ldots)$
is closed in $L^{1}\left(\Omega^{2}\right)$.
\end{proof}
\begin{rem}
\label{rem:ACCnotClosed}In \cite{DH:WeakStar}, an analogous closeness
property of $\LIM(\Gamma_{1},\Gamma_{2},\Gamma_{3},\ldots)$ than
Lemma~\ref{lem:LIMclosed}\emph{\ref{enu:closedL1}} was established
and used, namely that the set $\left\{ \INT_{f}(W):W\in\LIM(\Gamma_{1},\Gamma_{2},\Gamma_{3},\ldots)\right\} $
attains it supremum (here, $f$ is a fixed continuous strictly convex
function). Section~7.4 of~\cite{DH:WeakStar} contains an example,
due to Jon Noel, which shows that the set $\left\{ \INT_{f}(W):W\in\ACC(\Gamma_{1},\Gamma_{2},\Gamma_{3},\ldots)\right\} $
need not even achieve its supremum.
\end{rem}

\subsection{Differentiating between $\protect\ACC(\Gamma_{1},\Gamma_{2},\Gamma_{3},\ldots)$
and $\protect\LIM(\Gamma_{1},\Gamma_{2},\Gamma_{3},\ldots)$ in~\cite{DH:WeakStar}
and in the present paper\label{subsec:DifferentiatingACCLIM}}

In the proof of Theorem~\ref{thm:compactness} given in~\cite{DH:WeakStar},
which is in some sense a precursor of the current work, quite some
work is put into zigzagging between $\LIM(\Gamma_{1},\Gamma_{2},\Gamma_{3},\ldots)$
and $\ACC(\Gamma_{1},\Gamma_{2},\Gamma_{3},\ldots)$. Let us explain
this in more detail. Let us fix a continuous strictly convex function
$f$. The idea for finding the graphon $\Gamma$ in Theorem~\ref{thm:compactness}
in \cite{DH:WeakStar} is as follows. Denoting by $X$ either \emph{(i)}
$\LIM(\Gamma_{1},\Gamma_{2},\Gamma_{3},\ldots)$ or \emph{(ii)} $\ACC(\Gamma_{1},\Gamma_{2},\Gamma_{3},\ldots)$,
we take $\Gamma\in X$ that maximizes $\INT_{f}(\Gamma)$. Using the
definition of $X$, there exist versions $\Gamma_{n_{1}}',\Gamma_{n_{2}}',\Gamma_{n_{3}}',\ldots$
of $\Gamma_{n_{1}},\Gamma_{n_{2}},\Gamma_{n_{3}},\ldots$ that converge
to $\Gamma$ weak{*}.\footnote{Note that in variant (i), we actually have $n_{1}=1,n_{2}=2,n_{3}=3,\ldots$.}
The aim is to prove that $\Gamma_{n_{1}}',\Gamma_{n_{2}}',\Gamma_{n_{3}}',\ldots$
actually converge to $\Gamma$ also in the cut norm \textemdash{}
that would obviously prove Theorem~\ref{thm:compactness}. Now, the
key step in~\cite{DH:WeakStar} is to prove that if $\Gamma_{n_{1}}',\Gamma_{n_{2}}',\Gamma_{n_{3}}',\ldots$
do not converge to $\Gamma$ in the cut norm, then there exist versions
$\Gamma_{n_{k_{1}}}'',\Gamma_{n_{k_{2}}}'',\Gamma_{n_{k_{3}}}'',\ldots$
of a suitable subsequence of $\Gamma_{n_{1}}',\Gamma_{n_{2}}',\Gamma_{n_{3}}',\ldots$
that weak{*} converge to a graphon $\Gamma'$ with $\INT_{f}(\Gamma')>\INT_{f}(\Gamma)$.
Since $\Gamma_{n_{k_{1}}}'',\Gamma_{n_{k_{2}}}'',\Gamma_{n_{k_{3}}}'',\ldots$
witness that $\Gamma'\in X$, this is a contradiction. Now, let us
explain why we need favorable properties of both (i) and (ii) for
the proof. Firstly, note that in the sentence <<Since $\Gamma_{n_{k_{1}}}'',\Gamma_{n_{k_{2}}}'',\Gamma_{n_{k_{3}}}'',\ldots$
witness that $\Gamma'\in X$>> we are referring to a subsequence,
so this is a correct justification only in case $X=\ACC(\Gamma_{1},\Gamma_{2},\Gamma_{3},\ldots)$.
On the other hand, in the sentence <<we take $\Gamma\in X$ that
maximizes $\INT_{f}(\Gamma)$>> we need the maximum to be achieved.
Such a closeness property is enjoyed by $\LIM(\Gamma_{1},\Gamma_{2},\Gamma_{3},\ldots)$
as we saw in Lemma~\ref{lem:LIMclosed}\ref{enu:closedL1}, but not
by $\ACC(\Gamma_{1},\Gamma_{2},\Gamma_{3},\ldots)$ as we saw in Remark~\ref{rem:ACCnotClosed}.

So, while differences between $\LIM(\Gamma_{1},\Gamma_{2},\Gamma_{3},\ldots)$
and $\ACC(\Gamma_{1},\Gamma_{2},\Gamma_{3},\ldots)$ were viewed in~\cite{DH:WeakStar}
as a nuisance that required a subtle and technical treatment, in this
section we shall show that these differences capture the essence of
the cut norm convergence. Namely, we shall prove in Theorem~\ref{thm:subsequenceLIMACC}
that each sequence $\Gamma_{1},\Gamma_{2},\Gamma_{3},\ldots$ of graphons
contains a subsequence $\Gamma_{n_{1}},\Gamma_{n_{2}},\Gamma_{n_{3}},\ldots$
such that 
\begin{equation}
\LIM\left(\Gamma_{n_{1}},\Gamma_{n_{2}},\Gamma_{n_{3}},\ldots\right)=\ACC\left(\Gamma_{n_{1}},\Gamma_{n_{2}},\Gamma_{n_{3}},\ldots\right)\;,\label{eq:MaMa}
\end{equation}
and in Theorem~\ref{thm:hyperspaceANDcutDISTANCEsimpler} we shall
prove that~(\ref{eq:MaMa}) is equivalent to cut distance convergence
of $\Gamma_{n_{1}},\Gamma_{n_{2}},\Gamma_{n_{3}},\ldots$. Of course,
a proof of Theorem~\ref{thm:compactness} then follows immediately.

\subsection{Main results: subsequences with $\protect\LIM=\protect\ACC$}

As we observed earlier we have $\LIM(\Gamma_{1},\Gamma_{2},\Gamma_{3},\ldots)\subseteq\ACC(\Gamma_{1},\Gamma_{2},\Gamma_{3},\ldots)$
and equality usually does not hold. The next theorem however says
that we can always achieve equality after passing to a subsequence.
\begin{thm}
\label{thm:subsequenceLIMACC}Let $\mathcal{S}=\left(\Gamma_{1},\Gamma_{2},\Gamma_{3},\ldots\right)$
be a sequence of graphons. Then there exists a subsequence $\Gamma_{n_{1}},\Gamma_{n_{2}},\Gamma_{n_{3}},\ldots$
such that $\LIM\left(\Gamma_{n_{1}},\Gamma_{n_{2}},\Gamma_{n_{3}},\ldots\right)=\ACC\left(\Gamma_{n_{1}},\Gamma_{n_{2}},\Gamma_{n_{3}},\ldots\right)$.
\end{thm}

The proof of Theorem~\ref{thm:subsequenceLIMACC} proceeds by transfinite
induction. Crucially, we rely on a well-known fact from descriptive
set theory, below referred to~\cite{MR1321597}, which says that
a strictly increasing transfinite sequence of closed sets in a second
countable topological space must be of at most countable length. We
shall apply this to the space $\left(L^{\infty}(\Omega^{2}),\mathrm{w}^{*}\right)$
which is second countable because it is metrizable and separable.
\begin{proof}
For two sequences\footnote{By a \emph{sequence}, we mean a system indexed by a countable initial
segment of ordinals.} of graphons $\mathcal{U}$ and $\mathcal{T}$ we write $\mathcal{U\le^{\star}T}$
if deleting finitely many terms from $\mathcal{U}$ gives us a subsequence
of $\mathcal{T}$. Note that the relation $\le^{\star}$ is transitive.
Note that if $\mathcal{U\le^{\star}T}$ then 
\begin{equation}
\LIM\left(\mathcal{T}\right)\subseteq\LIM\left(\mathcal{U}\right)\;.\label{eq:LimSequences}
\end{equation}

In the following, we construct a countable ordinal $\alpha_{0}$ and
a transfinite sequence $\left(\mathcal{S}_{\alpha}\right)_{\alpha\leq\alpha_{0}}$
of subsequences of $\mathcal{S}$ such that for every pair of ordinals
$\gamma<\delta$ it holds that 
\begin{equation}
\mathcal{S}_{\delta}\leq^{\star}\mathcal{S}_{\gamma}\;,\label{eq:orderedfish}
\end{equation}
 and also that $\LIM\left(\mathcal{S}_{\gamma}\right)$ is a proper
subset of $\LIM\left(\mathcal{S}_{\delta}\right)$.

In the first step, we put $\mathcal{S}_{0}=\mathcal{S}$. Now suppose
that for some countable ordinal $\alpha$, we have already constructed
$\mathcal{S}_{\beta}$ for every $\beta<\alpha$. Either $\alpha=\beta+1$
for some ordinal $\beta$ or $\alpha$ is a limit ordinal. Suppose
first that $\alpha=\beta+1$ for some ordinal $\beta$. We distinguish
two cases. If $\LIM(\mathcal{S}_{\beta})=\ACC(\mathcal{S}_{\beta})$
then we define $\alpha_{0}=\beta$ and the construction is finished.
Otherwise there is some graphon $W\in\ACC(\mathcal{S}_{\beta})\setminus\LIM(\mathcal{S}_{\beta})$.
Then we proceed the construction by finding a subsequence $\mathcal{S}_{\beta+1}$
of $\mathcal{S}_{\beta}$ such that some versions of the graphons
from $S_{\beta+1}$ converge to $W$. This way we have $\mathcal{S}_{\beta+1}$$\leq^{\star}$$\mathcal{S}_{\beta}$
and $W\in\LIM(\mathcal{S}_{\beta+1})\setminus\LIM(\mathcal{S}_{\beta})$.
Now suppose that $\alpha$ is a countable limit ordinal. We find an
increasing sequence $\beta_{1},\beta_{2},\beta_{3},\ldots$ of ordinals
such that $\beta_{i}\rightarrow\alpha$ for $i\rightarrow\infty$
(this is possible as $\alpha$ has countable cofinality). Now we use
the diagonal method to define a sequence $S_{\alpha}$ such that $S_{\alpha}$$\le^{\star}$$S_{\beta_{i}}$
for every $i$. Combined with (\ref{eq:orderedfish}) and with $\beta_{i}\rightarrow\alpha$,
we get that $\mathcal{S}_{\alpha}$$\le^{\star}$$\mathcal{S}_{\beta}$
for every $\beta<\alpha$. Plugging in~(\ref{eq:LimSequences}),
we conclude $\bigcup_{\beta<\alpha}\LIM(\mathcal{S}_{\beta})\subseteq\LIM(\mathcal{S}_{\alpha})$.

The obtained transfinite sequence $\left(\LIM(\mathcal{S}_{\alpha})\right)_{\alpha\le\alpha_{0}}$
is a strictly increasing sequence of subsets of unit ball in $L^{\infty}(\Omega^{2})$.
By Lemma~\ref{lem:LIMclosed}\ref{enu:weakstar}, all of these subsets
are weak{*} closed. By~\cite[Theorem 6.9]{MR1321597}, the sequence
is at most countable, i.e. the previous construction stopped at some
countable ordinal $\alpha_{0}$. This means that $\LIM\left(\mathcal{S}_{\alpha_{0}}\right)=\ACC\left(\mathcal{S}_{\alpha_{0}}\right)$.
\end{proof}
\begin{rem}
\label{rem:LIMACCobservable}Theorem~\ref{thm:subsequenceLIMACC}
substantially extends the key Lemma~13 from~\cite{DH:WeakStar}
which states that any sequence of graphons $\Gamma_{1},\Gamma_{2},\Gamma_{3},\ldots$
contains a subsequence $\Gamma_{n_{1}},\Gamma_{n_{2}},\Gamma_{n_{3}},\ldots$
such that
\[
\sup\left\{ \INT_{f}\left(\Gamma\right)\;:\;\Gamma\in\LIM\left(\Gamma_{n_{1}},\Gamma_{n_{2}},\Gamma_{n_{3}},\ldots\right)\right\} =\sup\left\{ \INT_{f}\left(\Gamma\right)\;:\;\Gamma\in\ACC\left(\Gamma_{n_{1}},\Gamma_{n_{2}},\Gamma_{n_{3}},\ldots\right)\right\} \;,
\]
for a continuous strictly convex function $f:\left[0,1\right]\rightarrow\mathbb{R}$.
Lemma~13 in~\cite{DH:WeakStar} is proved by induction (over natural
numbers) without any appeal to descriptive set theory.
\end{rem}

As promised, we shall now state that the property asserted in Theorem~\ref{thm:subsequenceLIMACC}
is necessary and sufficient for cut distance convergence. 
\begin{thm}
\label{thm:hyperspaceANDcutDISTANCEsimpler}Let $\Gamma_{1},\Gamma_{2},\Gamma_{3},\ldots\in\GRAPHONSPACE$.
The following are equivalent:
\begin{enumerate}[label=(\alph*)]
\item \label{enu:basicMainCauchy}The sequence $\Gamma_{1},\Gamma_{2},\Gamma_{3},\ldots$
is Cauchy with respect to the cut distance $\delta_{\square}$,
\item \label{enu:basicThmLIMACC}$\LIM\left(\Gamma_{1},\Gamma_{2},\Gamma_{3},\ldots\right)=\ACC\left(\Gamma_{1},\Gamma_{2},\Gamma_{3},\ldots\right)$.
\end{enumerate}
Furthermore, in case~\ref{enu:basicMainCauchy} and \ref{enu:basicThmLIMACC}
hold, we can take a maximal element $W$ in $\LIM\left(\Gamma_{1},\Gamma_{2},\Gamma_{3},\ldots\right)$
with respect to the structuredness order (defined in Section~\ref{sec:structurdness}
below) and then $\Gamma_{1},\Gamma_{2},\Gamma_{3},\ldots\CUTDISTCONV W$.
\end{thm}

We provide a proof of Theorem~\ref{thm:hyperspaceANDcutDISTANCEsimpler}
in Section~\ref{sec:ProofOfEquivalenceSimpler}, after building key
tools in Section~\ref{sec:structurdness}. In Section~\ref{subsec:RelatingToVietoris}
we state and prove Theorem~\ref{thm:hyperspaceANDcutDISTANCE} which
extends Theorem~\ref{thm:hyperspaceANDcutDISTANCEsimpler} and relates
cut distance convergence to convergence in the hyperspace $K(\GRAPHONSPACE)$.

\section{Envelopes and the structuredness order\label{sec:structurdness}}

Suppose that $W\in\GRAPHONSPACE$ is a graphon. We call the set $\left\langle W\right\rangle :=\LIM(W,W,W,\ldots)$
the \emph{envelope} of $W$. Envelopes allow us to introduce structuredness
order on graphons. Intuitively, less-structured graphons have smaller
envelopes. Extreme examples of this are constant graphons $W\equiv c$
(for some $c\in[0,1]$), which are obviously the only graphons for
which $\left\langle W\right\rangle =\left\{ W\right\} $. This leads
us to say that a graphon $U$ is \emph{at most as structured as a
graphon $W$} if $\left\langle U\right\rangle \subseteq\left\langle W\right\rangle $.
We write $U\preceq W$ in this case. We write $U\prec W$ if $U\preceq W$
but it does not hold that $W\preceq U$. Observe that $\preceq$ is
a quasiorder on the space of graphons and if $U\preceq W$ then also
$U^{\varphi}\preceq W$ for every measure preserving bijection $\varphi$.
As we shall see in Lemma~\ref{lem:ORDER}, it is actually an order
on the space of graphons modulo weak isomorphism. To prove these results
we shall need several auxiliary results.

\begin{lem}[Lemma 7 in \cite{DH:WeakStar}]
\label{lem:averagingInsideLIM}Suppose that $\Gamma_{1},\Gamma_{2},\Gamma_{3},\ldots:\Omega^{2}\rightarrow[0,1]$
is a sequence of graphons. Suppose that $W\in\LIM(\Gamma_{1},\Gamma_{2},\Gamma_{3},\ldots)$
and that we have a partition $\mathcal{P}$ of $\Omega$ into finitely
many sets. Then $W^{\Join\mathcal{P}}\in\LIM(\Gamma_{1},\Gamma_{2},\Gamma_{3},\ldots)$.
\end{lem}

\begin{lem}
\label{lem:basicpropofenvelopes}Suppose that $W\in\GRAPHONSPACE$.
Then 
\end{lem}

\begin{enumerate}[label=(\alph*)]
\item \label{enu:closureINenve}If $Q\subseteq\langle W\rangle$ then the
weak{*} closure of $Q$ is also contained in $\langle W\rangle$,
\item \label{enu:averagingINSIDE}$W^{\Join\mathcal{P}}\in\langle W\rangle$
for every finite partition $\mathcal{P}$ of $\Omega$,
\item \label{enu:UprecWenve}$U\in\langle W\rangle$ if and only if $U\preceq W$,
\item \label{enu:envelopeswhendistzero}if $\delta_{\Box}(W,U)=0$ then
$\langle W\rangle=\langle U\rangle$.
\end{enumerate}
\begin{proof}
Item~\ref{enu:closureINenve} follows from Lemma~\ref{lem:LIMclosed}\ref{enu:weakstar}.
Item~\ref{enu:averagingINSIDE} is a special case of Lemma~\ref{lem:averagingInsideLIM}.

Let us now turn to Item~\ref{enu:UprecWenve}. If $U\preceq W$ then
$U\in\langle W\rangle$ follows from the definition of $\preceq$
and the fact that $U\in\langle U\rangle$. To prove the opposite implication
observe that if $\left(\varphi_{n}:\Omega\rightarrow\Omega\right)_{n}$
is a sequence of measure preserving bijections, $W^{\varphi_{n}}\WEAKCONV U$
and $\psi:\Omega\rightarrow\Omega$ is a measure preserving bijection
then $W^{\varphi_{n}\psi}\WEAKCONV U^{\psi}$. Then we have that every
version of $U$ is in $\left\langle W\right\rangle $. Because $\left\langle U\right\rangle $
is exactly the weak{*} closure of the set of all versions of $U$
we obtain that $U\preceq W$.

Let us now prove Item~\ref{enu:envelopeswhendistzero}. If $\delta_{\Box}(W,U)=0$
then we have a sequence $W^{\varphi_{n}}\CUTNORMCONV U$ which by
Fact~\ref{fact:differentnorms} implies that $W^{\varphi_{n}}\WEAKCONV U$.
Therefore $U\preceq W$. A symmetric argument gives $W\preceq U$
and we may conclude that $\langle W\rangle=\langle U\rangle$.
\end{proof}

\subsection{Beyond Lemma~\ref{lem:basicpropofenvelopes}\ref{enu:averagingINSIDE}\label{subsec:BeyondFiniteSubalgebras}}

In this short section we borrow results stated and proven further
below to show, in Theorem~\ref{thm:Subalgebras}, that Lemma~\ref{lem:basicpropofenvelopes}\ref{enu:averagingINSIDE}
can be strengthened. We do not have applications of Theorem~\ref{thm:Subalgebras}.
On the other hand, together with the complement of Theorem~\ref{thm:Subalgebras},
which we include as Problem~\ref{prob:Subalgebras}, this would lead
to a correspondence between the structuredness order and order on
sub-sigma-algebras of the Borel sigma-algebra. Let us now give details.

Let $\mathcal{B}$ be the sigma-algebra of measurable sets on $\Omega$.
Suppose that $W:\Omega^{2}\rightarrow[0,1]$ and $\mathcal{\mathcal{P}}$
is a finite partition of $\Omega$ into measurable sets. Let $\mathcal{P}^{*}$
be the algebra generated by $\mathcal{\mathcal{P}}$; clearly $\mathcal{P}^{*}$
is finite. We can express $W^{\Join\mathcal{P}}$ as a conditional
expectation of $W$ with respect to a sub-sigma-algebra, $W^{\Join P}=\EXP\left[W|\mathcal{P}^{*}\times\mathcal{P}^{*}\right]$.
Then Lemma~\ref{lem:basicpropofenvelopes}\ref{enu:averagingINSIDE}
tells us that this particular conditional expectation lies in $\langle W\rangle$.
Here, the fact that we are taking a conditional expectation with respect
to a finite sub-sigma-algebra is not needed.
\begin{thm}
\label{thm:Subalgebras}Suppose $\mathcal{B}$ is the sigma-algebra
of measurable sets on $\Omega$, $\mathcal{A}$ is a sub-sigma-algebra
of $\mathcal{B}$, and that $W:\Omega^{2}\rightarrow[0,1]$ is a graphon.
Then $\EXP\left[W|\mathcal{A}\times\mathcal{A}\right]\in\langle W\rangle$.
\end{thm}

\begin{proof}
Recall that we assume that $\Omega$ is a separable atomless measure
space. This implies that there is an increasing sequence $\left\{ \mathcal{A}_{n}\right\} _{n\in\mathbb{N}}$
of finite subalgebras of $\mathcal{A}$ such that $\bigcup_{n\in\mathbb{N}}\mathcal{A}_{n}$
generates $\mathcal{A}$, i.e., $\mathcal{A}$ is the smallest sub-sigma-algebra
of $\mathcal{B}$ that contains $\mathcal{A}_{n}$ for every $n\in\mathbb{N}$.
Clearly, there are finite partitions $\mathcal{P}_{n}$ of $\Omega$
for every $n\in\mathbb{N}$ such that $\mathcal{P}_{n}^{*}=\mathcal{A}_{n}$,
where we use the notation from the beginning of this subsection. It
follows that $W^{\Join P_{n}}\in\left\langle W^{\Join P_{n+1}}\right\rangle $
and by Corollary~\ref{cor:chains} together with Lemma~\ref{lem:biggerForConvergent}
we have that $W^{\Join P_{n}}\CUTDISTCONV U$ where $U\in\left\langle W\right\rangle $.
It follows from the Martingale convergence theorem \cite[Theorem 35.5]{Billingsley95}
that $W^{\Join P_{n}}\LONECONV\EXP\left[W|\mathcal{A}\times\mathcal{A}\right]$
and as a consequence we get that $\delta_{\Box}\left(\EXP\left[W|\mathcal{A}\times\mathcal{A}\right],U\right)=0$.
\end{proof}
As advertised, we pose the complement of Theorem~\ref{thm:Subalgebras},
which would give one-to-one correspondence between the structuredness
order and containment of sub-sigma-algebras of $\mathcal{B}$, as
an open problem.
\begin{problem}
\label{prob:Subalgebras}Suppose $\mathcal{B}$ is the sigma-algebra
of measurable sets on $\Omega$, and that $U\preceq W$ are two graphons.
Do there exist graphons $U',W'$, and a sub-sigma-algebra $\mathcal{A}$
of $\mathcal{B}$ such that $\cutDIST(U,U')=\cutDIST(W,W')=0$ and
$U'=\EXP\left[W'|\mathcal{A}\times\mathcal{A}\right]$?
\end{problem}

\subsection{$\preceq$-maximal elements in $\protect\LIM$}

The main result of this section, Lemma~\ref{lem:LIMACCcontainsMAX},
says that if 
\[
\LIM\left(\Gamma_{1},\Gamma_{2},\Gamma_{3},\ldots\right)=\ACC\left(\Gamma_{1},\Gamma_{2},\Gamma_{3},\ldots\right)
\]
then there exists a $\preceq$-maximal element in $\LIM\left(\Gamma_{1},\Gamma_{2},\Gamma_{3},\ldots\right)$.
Most of the work for the proof of Lemma~\ref{lem:LIMACCcontainsMAX}
is done in Lemma~\ref{lem:forVasek} which is stated for step graphons
only. To infer that certain favorable properties of a sequence of
step graphons (on which Lemma~\ref{lem:forVasek} can be applied)
can be transferred even to a graphon they approximate, Lemma~\ref{lem:biggerForConvergent}
is introduced.
\begin{lem}
\label{lem:biggerForConvergent}Suppose $U_{1},U_{2},U_{3},\ldots$
is a sequence of graphons that converges weak{*} to $U$, and suppose
that $W$ is a graphon. Suppose that for each $n\in\mathbb{N}$ we
have that $U_{n}\preceq W$. Then $U\preceq W$.
\end{lem}

\begin{proof}
Follows immediately from Lemma \ref{lem:basicpropofenvelopes}\ref{enu:closureINenve}.
\end{proof}
For the key Lemma~\ref{lem:forVasek}, we shall refine the structure
of a graphon by <<moving some parts to the left>>. To this end,
it is convenient to work on $[0,1]$ (see Remark~\ref{rem:UnitIntervalGroundSpace}).
We introduce the following definitions.
\begin{defn}
By an \emph{ordered partition} $\mathcal{P}$ of a set $S$, we mean
a finite partition of $S$, $S=P_{1}\sqcup P_{2}\sqcup P_{3}\sqcup\ldots\sqcup P_{k}$,
$\mathcal{P}=\left(P_{1},P_{2},P_{3},\ldots,P_{k}\right)$ in which
the sets $P_{1},P_{2},P_{3},\ldots,P_{k}$ are linearly ordered (in
the way they are enumerated in $\mathcal{P}$).
\end{defn}

\begin{defn}
\label{def:shiftingMaps}For an ordered partition $\mathcal{J}$ of
$I=[0,1]$ into finitely many sets $C_{1},C_{2},\ldots,C_{k}$, we
define mappings $\alpha_{\mathcal{J},1},\alpha_{\mathcal{J},2},\ldots,\alpha_{\mathcal{J},k}:I\rightarrow I$,
and a mapping $\gamma_{\mathcal{J}}:I\rightarrow I$ by 
\begin{equation}
\begin{split}\alpha_{\mathcal{J},i}(x) & =\lambda\left(\bigcup_{i=1}^{j-1}C_{j}\right)+\lambda\left(C_{i}\cap\left[0,x\right]\right),\quad i=1,2,\ldots,k\;.\\
\gamma_{\mathcal{J}}(x) & =\alpha_{\mathcal{J},i}(x)\quad\text{if }x\in C_{i},\quad i=1,2,\ldots,k\;.
\end{split}
\label{eq:shiftingMaps}
\end{equation}
 Informally, $\gamma_{\mathcal{J}}$ is defined in such a way that
it maps the set $C_{1}$ to the left side of the interval $I$, the
set $C_{2}$ next to it, and so on. Finally, the set $C_{k}$ is mapped
to the right side of the interval $I$. Clearly, $\gamma_{\mathcal{J}}$
is a measure preserving almost-bijection.

Last, given a graphon $W:I^{2}\rightarrow[0,1]$, we define a graphon
$_{\mathcal{J}}W:I^{2}\rightarrow[0,1]$ by
\begin{equation}
_{\mathcal{J}}W(x,y):=W\left(\gamma_{\mathcal{J}}^{-1}(x),\gamma_{\mathcal{J}}^{-1}(y)\right)\;.\label{eq:reorderedW}
\end{equation}
When $\mathcal{J}$ has only two parts, $\mathcal{J}=\left(A,I\setminus A\right)$,
in order to simplify notation, we write
\begin{equation}
_{A}W:={}_{\mathcal{J}}W\;.\label{eq:ordered2}
\end{equation}
\end{defn}

\begin{lem}
\label{lem:forVasek}Let $\Gamma_{1},\Gamma_{2},\Gamma_{3},\ldots\in\GRAPHONSPACE$
be a sequence of graphons on $[0,1]$. Suppose that $U,V\in\LIM\left(\Gamma_{1},\Gamma_{2},\Gamma_{3},\ldots\right)$
are two step graphons. Then there exists a step graphon $W\in\ACC\left(\Gamma_{1},\Gamma_{2},\dots\right)$
that refines $U$ and such that $U,V\preceq W$.
\end{lem}

\begin{proof}
\begin{figure}
\includegraphics[width=0.9\textwidth]{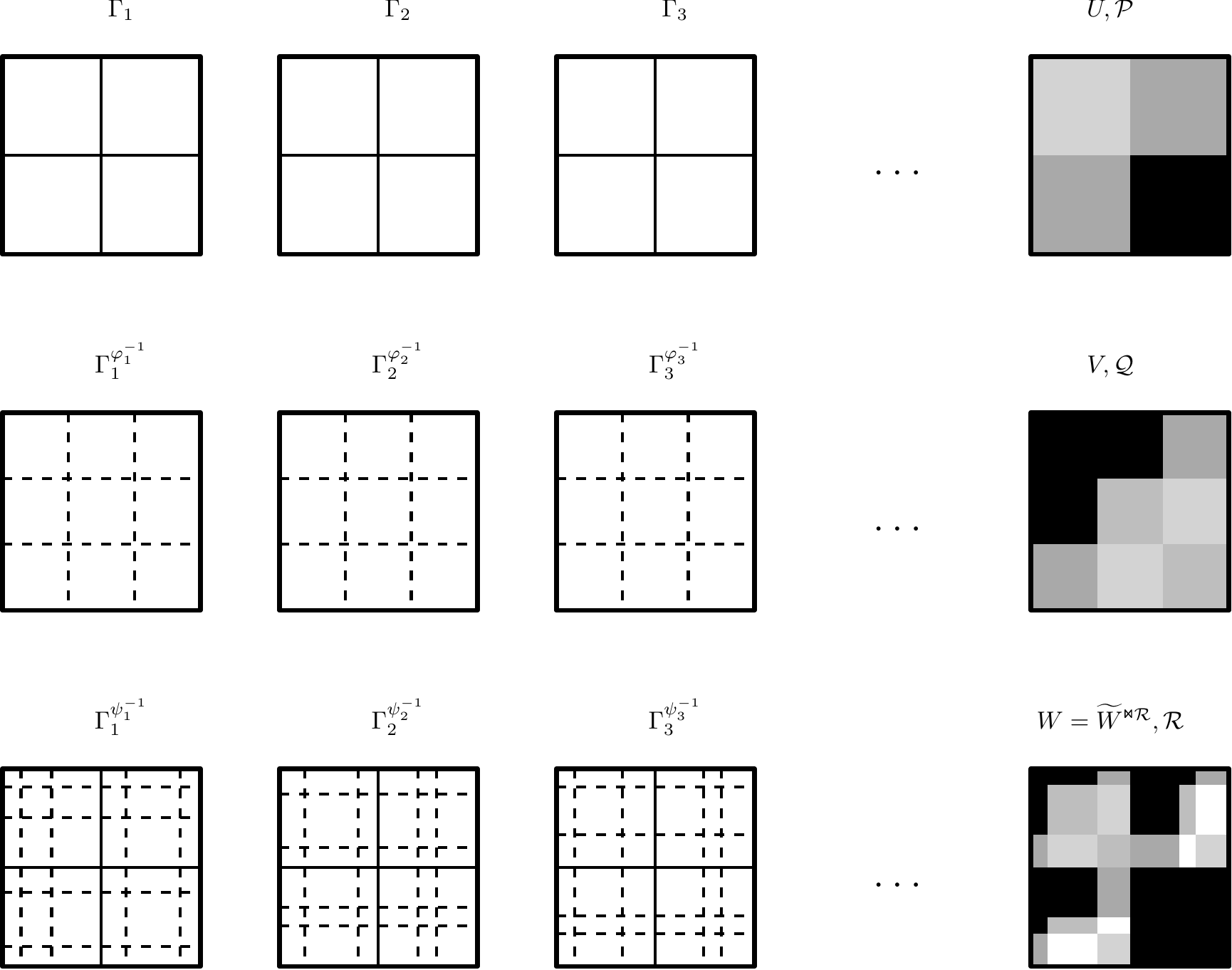}\caption{Two graphons $U,V$ are step graphons with partitions $\mathcal{P}$
and $\mathcal{Q}$. A subsequence of $\Gamma_{1}^{\psi_{1}^{-1}},\Gamma_{2}^{\psi_{2}^{-1}},\dots$
converges to $\widetilde{W}$ and the corresponding partitions converge
to the partition $\mathcal{R}$ that refines both $\mathcal{P}$ and
$\mathcal{Q}$. The graphon $W=\widetilde{W}^{\protect\Join\mathcal{R}}$
is the desired step graphon that is structured more than both $U$
and $V$. }
\end{figure}
We at first assume that the ordered partition $\mathcal{P}=\left(P_{1},\dots,P_{m}\right)$
of $U$ and the ordered partition $\mathcal{Q}=\left(Q_{1},\dots,Q_{n}\right)$
of $V$ are composed of intervals, since each partition $\mathcal{J}$
can be reordered to intervals by the measure preserving almost-bijection
$\gamma_{\mathcal{J}}$. We further assume without loss of generality
that the sequence of measure preserving bijections certifying that
$U\in\LIM\left(\Gamma_{1},\Gamma_{2},\dots\right)$ contains only
identities, i.e., that $\Gamma_{1},\Gamma_{2},\dots\WEAKCONV U$.
Let $\varphi_{1},\varphi_{2},\dots$ be measure preserving bijections
such that $\Gamma_{1}^{\varphi_{1}^{-1}},\Gamma_{2}^{\varphi_{2}^{-1}},\dots\WEAKCONV V$. 

We now describe a sequence of measure preserving bijections $\psi_{1},\psi_{2},\dots$
such that the weak star limit of $\Gamma_{1}^{\psi_{1}^{-1}},\Gamma_{2}^{\psi_{2}^{-1}}\dots$
gives the desired graphon $W$. For the $\ell$-th graphon $\Gamma_{\ell}$
we define its partition $\mathcal{\mathcal{H}}^{(\ell)}=\left(H_{1,1}^{(\ell)},H_{1,2}^{(\ell)},\dots,H_{1,n}^{(\ell)},H_{2,1}^{(\ell)}\dots,H_{m,n}^{(\ell)}\right)$
where $H_{i,j}^{(\ell)}=P_{i}\cap\varphi_{\ell}^{-1}(Q_{j})$ and
set $\psi_{\ell}=\gamma_{\mathcal{\mathcal{H}}^{\left(\ell\right)}}$
using Definition~\ref{def:shiftingMaps}. The intuition behind $\psi_{\ell}$
is that it refines each block $P_{i}$ of the partition $\mathcal{P}$
with the partition $\mathcal{Q}$; it can be, indeed, seen that for
each $i$ we have $\psi_{\ell}\left(P_{i}\right)=\psi_{\ell}\left(\bigcup_{j}H_{i,j}^{(\ell)}\right)=P_{i}$,
where the equalities hold up to a null set. 

We pass to a subsequence $mn$ times to get that both endpoints of
each of the intervals $\psi_{\ell}(H_{i,j}^{(\ell)})$ converge to
some fixed numbers from $[0,1]$, thus giving us a limit partition
$\mathcal{R}=(R_{1,1},\dots,R_{m,n})$ into intervals (some of them
may be degenerate intervals of length $0$). Note that as we know
that for each $i$ we have $\psi_{\ell}\left(\bigcup_{j}H_{i,j}^{(\ell)}\right)=P_{i}$,
it is also true that $\bigcup_{j}R_{i,j}=P_{i}$. We also have for
all $j$ that $\nu\left(\bigcup_{i}H_{i,j}^{(\ell)}\right)=\nu\left(Q_{j}\right)$,
hence $\nu\left(\bigcup_{i}R_{i,j}\right)=\nu\left(Q_{j}\right)$.
Now we use the fact that the set of accumulation points of our sequence
is non-empty due to Banach\textendash Alaoglu theorem, thus after
passing to a subsequence yet again we get a subsequence $\Gamma_{k_{1}}^{\psi_{k_{1}}^{-1}},\Gamma_{k_{2}}^{\psi_{k_{2}}^{-1}},\dots\WEAKCONV\widetilde{W}$.
Define $W$ as $\widetilde{W}^{\Join\mathcal{R}}$. We apply Lemma~\ref{lem:LIMclosed}\ref{enu:weakstar}
to $\LIM\left(\Gamma_{k_{1}},\Gamma_{k_{2}},\dots\right)$ and Lemma~\ref{lem:basicpropofenvelopes}\ref{enu:averagingINSIDE}
to $\widetilde{W}$ and $\widetilde{W}^{\Join\mathcal{R}}$ to get
that $\widetilde{W}^{\Join\mathcal{R}}\in\ACC(\Gamma_{1},\Gamma_{2},\ldots)$. 

At first we prove that $\widetilde{W}^{\Join\mathcal{R}}$ refines
$U$. Since $U$ is constant on each step $P_{i}\times P_{j}$, it
suffices to prove that for any $\varepsilon>0$ and any step $P_{i}\times P_{j}$
we have
\begin{align}
\int_{P_{i}\times P_{j}}U & =\int_{P_{i}\times P_{j}}\widetilde{W}^{\Join\mathcal{R}}.\label{eq:pikachu}
\end{align}
Take $\ell$ sufficiently large, so that 
\begin{align}
\left|\int_{P_{i}\times P_{j}}U-\int_{P_{i}\times P_{j}}\Gamma_{k_{\ell}}\right| & <\varepsilon\label{eq:raichu}
\end{align}
and 
\begin{align}
\left|\int_{P_{i}\times P_{j}}\widetilde{W}-\int_{P_{i}\times P_{j}}\Gamma_{k_{\ell}}^{\psi_{k_{\ell}}^{-1}}\right| & <\varepsilon.\label{eq:bulbasaur}
\end{align}
Putting this together with the facts that $P_{i}=\psi_{k_{\ell}}\left(P_{i}\right)$
up to a null set and $R_{i,j}\subseteq P_{i}$ for all $i,j$, we
get that 
\begin{align*}
\int_{P_{i}\times P_{j}}U\\
\JUSTIFY{Eq.\;\eqref{eq:raichu}} & \stackrel{\varepsilon}{\approx}\int_{P_{i}\times P_{j}}\Gamma_{k_{\ell}}\\
\JUSTIFY{\text{\ensuremath{P_{i}}=\ensuremath{\psi_{k_{\ell}}\left(P_{i}\right)}}} & =\int_{\psi_{k_{\ell}}(P_{i})\times\psi_{k_{\ell}}(P_{j})}\Gamma_{k_{\ell}}\\
 & =\int_{P_{i}\times P_{j}}\Gamma_{k_{\ell}}^{\psi_{k_{\ell}}^{-1}}\\
\JUSTIFY{Eq.\;\eqref{eq:bulbasaur}} & \stackrel{\varepsilon}{\approx}\int_{P_{i}\times P_{j}}\widetilde{W}\\
 & =\int_{P_{i}\times P_{j}}\widetilde{W}^{\Join\mathcal{R}}.
\end{align*}
Since this holds for every $\varepsilon>0$, we get the desired Equation~(\ref{eq:pikachu}). 

Now we prove that $\widetilde{W}^{\Join\mathcal{R}}\succeq V$. Since
$V$ is constant on each step $Q_{i}\times Q_{j}$, it suffices to
prove that for any $\varepsilon>0$ and any step $Q_{i}\times Q_{j}$
of $V$ we have

\begin{equation}
\int_{Q_{i}\times Q_{j}}V=\sum_{\substack{1\le g\le m\\
1\le h\le n
}
}\int_{R_{g,i}\times R_{h,j}}\widetilde{W}^{\Join\mathcal{R}}.\label{eq:abra}
\end{equation}
Take $\ell$ sufficiently large so that

\begin{align}
\left|\int_{Q_{i}\times Q_{j}}V-\int_{Q_{i}\times Q_{j}}\Gamma_{k_{\ell}}^{\varphi_{k_{\ell}}^{-1}}\right| & <\varepsilon,\label{eq:kadabra}\\
\left|\sum_{\substack{g,h}
}\int_{R_{g,i}\times R_{h,j}}\widetilde{W}^{\Join\mathcal{R}}-\sum_{\substack{g,h}
}\int_{R_{g,i}\times R_{h,j}}\Gamma_{k_{\ell}}^{\psi_{k_{\ell}}^{-1}}\right| & <\varepsilon\label{eq:alakazam}
\end{align}
and, moreover, the length of each interval $\psi_{k_{\ell}}\left(H_{i,j}^{(k_{\ell})}\right)$
differs from the length of interval $R_{i,j}$ by at most $\frac{\varepsilon}{4m^{2}n^{2}}$.
Now we can bound the measure of overlap of each pair of rectangles
$R_{g,i}\times R_{h,j}$ and $\psi_{k_{\ell}}\left(H_{g,i}^{(k_{\ell})}\right)\times\psi_{k_{\ell}}\left(H_{h,j}^{(k_{\ell})}\right)$.
More precisely, we have 
\begin{align}
\nu^{\otimes2}\left(\left(R_{g,i}\times R_{h,j}\right)\;\triangle\;\left(\psi_{k_{\ell}}\left(H_{g,i}^{(k_{\ell})}\right)\times\psi_{k_{\ell}}\left(H_{h,j}^{(k_{\ell})}\right)\right)\right) & <4mn\cdot\frac{\varepsilon}{4m^{2}n^{2}}=\frac{\varepsilon}{mn},\label{eq:charmander}
\end{align}
where $4mn$ comes from the facts that we bound the displacement of
all four sides of the rectangles and that their displacement depends
on the displacement of all preceding intervals. Putting all of this
together, we get that

\begin{align*}
\sum_{g,h}\int_{R_{g,i}\times R_{h,j}}\widetilde{W}^{\Join\mathcal{R}}\\
\JUSTIFY{Eq.\;\eqref{eq:alakazam}} & \stackrel{\varepsilon}{\approx}\sum_{g,h}\int_{R_{g,i}\times R_{h,j}}\Gamma_{k_{\ell}}^{\psi_{k_{\ell}}^{-1}}\\
\JUSTIFY{Eq.\;\eqref{eq:charmander}} & \stackrel{\varepsilon}{\approx}\sum_{g,h}\int_{\psi_{k_{\ell}}\left(H_{g,i}^{(k_{\ell})}\right)\times\psi_{k_{\ell}}\left(H_{h,j}^{(k_{\ell})}\right)}\Gamma_{k_{\ell}}^{\psi_{k_{\ell}}^{-1}}\\
 & =\sum_{g,h}\int_{\left(P_{g}\cap\varphi_{k_{\ell}}^{-1}\left(Q_{i}\right)\right)\times\left(P_{h}\cap\varphi_{k_{\ell}}^{-1}\left(Q_{j}\right)\right)}\Gamma_{k_{\ell}}\\
 & =\sum_{g,h}\int_{\left(\varphi_{k_{\ell}}(P_{g})\cap Q_{i}\right)\times\left(\varphi_{k_{\ell}}(P_{h})\cap Q_{j}\right)}\Gamma_{k_{\ell}}^{\varphi_{k_{\ell}}^{-1}}\\
 & =\int_{Q_{i}\times Q_{j}}\Gamma_{k_{\ell}}^{\varphi_{k_{\ell}}^{-1}}\\
\JUSTIFY{Eq.\;\eqref{eq:kadabra}} & \stackrel{\varepsilon}{\approx}\int_{Q_{i}\times Q_{j}}V\;.
\end{align*}
This yields the desired equation~(\ref{eq:abra}). 
\end{proof}
\begin{lem}
\label{lem:LIMACCcontainsMAX}Let $\Gamma_{1},\Gamma_{2},\Gamma_{3},\ldots\in\GRAPHONSPACE$
be a sequence of graphons on $[0,1]$ for which $\LIM\left(\Gamma_{1},\Gamma_{2},\Gamma_{3},\ldots\right)=\ACC\left(\Gamma_{1},\Gamma_{2},\Gamma_{3},\ldots\right)$.
Then $\LIM\left(\Gamma_{1},\Gamma_{2},\Gamma_{3},\ldots\right)$ contains
a maximum element with respect to the structuredness order.
\end{lem}

\begin{proof}
The space $\LIM\left(\Gamma_{1},\Gamma_{2},\Gamma_{3},\ldots\right)$
is separable metrizable since the space $\GRAPHONSPACE$ with the
weak{*} topology is separable metrizable and therefore we may find
a countable set $P\subseteq\LIM\left(\Gamma_{1},\Gamma_{2},\Gamma_{3},\ldots\right)$
such that its weak{*} closure is $\LIM\left(\Gamma_{1},\Gamma_{2},\Gamma_{3},\ldots\right)$.
For each $W\in P$ and $k\in\mathbb{N}$, consider a suitable graphon,
denoted by $W(k)$, that is an averaged $L^{1}$-approximation of
$W$ by a step graphon for precision $\frac{1}{k}$. Such a graphon
$W(k)$ exists by Lemma~\ref{lem:approx}. Note also that if $W(k)$
is chosen as $W^{\Join\mathcal{P}}$ for some finite partition $\mathcal{P}$
(as in Lemma~\ref{lem:approx}) then $W(k)\in\LIM\left(\Gamma_{1},\Gamma_{2},\Gamma_{3},\ldots\right)$
by Lemma~\ref{lem:averagingInsideLIM}.

Let us now consider the set $Q:=\left\{ W(k):W\in P,k\in\mathbb{N}\right\} $.
The set $Q$ is countable, contained in $\LIM\left(\Gamma_{1},\Gamma_{2},\Gamma_{3},\ldots\right)$
and its weak{*} closure is $\LIM\left(\Gamma_{1},\Gamma_{2},\Gamma_{3},\ldots\right)$.
Let $U_{1},U_{2},U_{3},\ldots$ be an enumeration of the elements
of $Q$. Let $M_{1}:=U_{1}$ and define inductively a sequence $\left\{ M_{n}\right\} _{n\in\mathbb{N}}\subseteq\LIM\left(\Gamma_{1},\Gamma_{2},\Gamma_{3},\ldots\right)$
as follows. Suppose that we have $M_{n}\in\LIM\left(\Gamma_{1},\Gamma_{2},\Gamma_{3},\ldots\right)$
and use Lemma~\ref{lem:forVasek} with $M_{n}$ in place of $U$
and $U_{n+1}$ in place of $V$ to find $M_{n+1}\in\ACC\left(\Gamma_{1},\Gamma_{2},\Gamma_{3},\ldots\right)$.
It follows from the assumption that $\ACC\left(\Gamma_{1},\Gamma_{2},\Gamma_{3},\ldots\right)=\LIM\left(\Gamma_{1},\Gamma_{2},\Gamma_{3},\ldots\right)$
that, in fact, $M_{n+1}\in\LIM\left(\Gamma_{1},\Gamma_{2},\Gamma_{3},\ldots\right)$.
Note that we have $U_{n}\preceq M_{n}$ for every $n\in\mathbb{N}$.
Moreover, it follows from the construction that there is a sequence
of finite partitions $\left\{ \mathcal{P}_{n}\right\} _{n\in\mathbb{N}}$
such that $M_{n}=M_{n+1}^{\Join\mathcal{P}_{n}}$ and $\mathcal{P}_{n+1}$
is a refinement of $\mathcal{P}_{n}$ for every $n\in\mathbb{N}$. 

By the Martingale convergence theorem \cite[Theorem 35.5]{Billingsley95}
we find $M$ such that $M_{n}\LONECONV M$ and $M^{\Join\mathcal{P}_{n}}=M_{n}$
for every $n\in\mathbb{N}$. Since $\LIM\left(\Gamma_{1},\Gamma_{2},\Gamma_{3},\ldots\right)$
is weak{*} closed by Lemma~\ref{lem:LIMclosed}\ref{enu:weakstar}
we have $M\in\LIM\left(\Gamma_{1},\Gamma_{2},\Gamma_{3},\ldots\right)$
and $M\succeq U_{n}$ for every $n\in\mathbb{N}$.

Finally, we claim that $M$ is a maximal element of $\LIM\left(\Gamma_{1},\Gamma_{2},\Gamma_{3},\ldots\right)$.
Indeed, let $\Gamma\in\LIM\left(\Gamma_{1},\Gamma_{2},\Gamma_{3},\ldots\right)$
be arbitrary. Since $Q$ is weak{*} dense in $\LIM\left(\Gamma_{1},\Gamma_{2},\Gamma_{3},\ldots\right)$,
we can find a sequence $U_{n_{1}},U_{n_{2}},U_{n_{3}},\ldots$ weak{*}
converging to $\Gamma$. Note that we have $U_{n_{i}}\preceq M$ for
each $i\in\mathbb{N}$. Lemma~\ref{lem:biggerForConvergent} gives
$\Gamma\preceq M$, as was needed.
\end{proof}

\subsection{Cut distance identifying graphon parameters}

In this section, we mention the notion of cut distance identifying
graphon parameters, which we then extensively study in a follow-up
paper~\cite{DGHRR:Parameters}. A \emph{graphon parameter} is any
function $\theta:\GRAPHONSPACE\rightarrow\mathbb{R}$ such that $\theta(W_{1})=\theta(W_{2})$
for any two graphons $W_{1}$ and $W_{2}$ with $\delta_{\square}(W_{1},W_{2})=0$.
We say that a graphon parameter $\theta(\cdot)$ is a \emph{cut distance
identifying graphon parameter} (\emph{CDIP}) if we have that $W_{1}\prec W_{2}$
implies $\theta\left(W_{1}\right)<\theta\left(W_{2}\right)$. In particular,
the key role of $\prec$-maximal element, such as in Theorem~\ref{enu:basicThmLIMACC},
can be in this setting expressed as $\theta$-maximal elements. Note
that $\INT_{f}(\cdot)$, where $f$ is a strictly convex convex, is
CDIP by \cite[Section 3.3]{DGHRR:Parameters}. Hence the proof of
Theorem~\ref{thm:compactness} in~\cite{DH:WeakStar} is a particular
instance of a CDIP-version of Theorem~\ref{enu:basicThmLIMACC},
the general version then being given in Theorem~3.3 and Theorem 3.4
in \cite{DGHRR:Parameters}. As we show in \cite{DGHRR:Parameters},
CDIP have more applications: they can be used for <<index-pumping>>
in the proof of the Frieze-Kannan regularity lemma, and they tell
us quite a bit about graph norms.

\subsection{Values and degrees with respect to the structuredness order\label{subsec:ValuesAndDegrees}}

Given a graphon $W:\Omega^{2}\rightarrow\left[0,1\right]$, we can
define a pushforward probability measure on $\left[0,1\right]$ by
\begin{equation}
\boldsymbol{\Phi}_{W}\left(A\right):=\nu^{\otimes2}\left(W^{-1}(A)\right)\,,\label{eq:pushforwardValues}
\end{equation}
for a set $A\subseteq\left[0,1\right]$. The measure $\boldsymbol{\Phi}_{W}$
gives us the distribution of the values of $W$, and we call it \emph{the
range frequencies of $W$}. Similarly, we can take the pushforward
measure of the degrees,
\begin{equation}
\boldsymbol{\Upsilon}_{W}\left(A\right):=\nu\left(\deg_{W}^{-1}\left(A\right)\right)\;,\label{eq:PushforwardDegrees}
\end{equation}
for a set $A\subseteq\left[0,1\right]$. We call $\boldsymbol{\Upsilon}_{W}$
the \emph{degree frequencies of $W$}. The measures $\boldsymbol{\Phi}_{W}$
and $\boldsymbol{\Upsilon}_{W}$ do not characterize $W$ in general
but certainly give us substantial information about $W$. Therefore,
given two graphons $U\preceq W$ it is natural to ask how $\boldsymbol{\Phi}_{U}$
compares to $\boldsymbol{\Phi}_{W}$ and how $\boldsymbol{\Upsilon}_{U}$
compares to $\boldsymbol{\Upsilon}_{W}$. To this end, we introduce
the following concept.
\begin{defn}
\label{def:flatter}Suppose that $\Lambda_{1}$ and $\Lambda_{2}$
are two finite measures on $\left[0,1\right]$. We say that $\Lambda_{1}$
is at \emph{least as flat} as $\Lambda_{2}$ if there exists a finite
measure $\Psi$ on $\left[0,1\right]^{2}$ such that $\Lambda_{1}$
is the marginal of $\Psi$ on the first coordinate, $\Lambda_{2}$
is the marginal of $\Psi$ on the second coordinate, and for each
measurable $D\subseteq\left[0,1\right]$ we have
\begin{equation}
\int_{D\times\left[0,1\right]}x\;\mathrm{d}\Psi(x,y)=\int_{D\times\left[0,1\right]}y\;\mathrm{d}\Psi(x,y)\;.\label{eq:masspreserve}
\end{equation}
In addition, we say that $\Lambda_{1}$ is \emph{strictly flatter}
than $\Lambda_{2}$ if $\Lambda_{1}\neq\Lambda_{2}$.

The condition that the marginals of $\Psi$ are $\Lambda_{1}$ and
$\Lambda_{2}$ is well-known in probability theory and referred to
as a \emph{coupling of $\Lambda_{1}$ and $\Lambda_{2}$}. However,
we are not aware of this concept being combined with the condition~(\ref{eq:masspreserve}).
\end{defn}

\begin{example}
We should understand $\Lambda_{1}$ as a certain averaging of $\Lambda_{2}$.
For example, suppose that $\Lambda_{1}$ has an atom $a\in[0,1]$,
say $\Lambda_{1}\left(\left\{ a\right\} \right)=m>0$. Then taking
$D=\left\{ a\right\} $, (\ref{eq:masspreserve}) tells us that by
averaging $y$ according to the normalized (by $\frac{1}{m}$) restriction
of the measure $\Psi$ to $\{a\}\times[0,1]$, we get $a$. As we
show in Lemma~\ref{lem:previ}, such a property extends also to non-atoms.
\end{example}

Let us prove two basic lemmas. In Lemma~\ref{lem:diagonal} we give
a useful characterization of strictly flatter pairs of measures. In
Lemma~\ref{lem:flattranisitve} we prove that the flatness relation
is actually an order. Even though we do not need these lemmas, we
believe that the theory we develop here would not be complete without
them.

We say that a finite measure $\Psi$ on $\left[0,1\right]^{2}$ is
\emph{diagonal} if we have $\Psi\left(\left\{ \left(x,y\right)\in[0,1]^{2}:x\neq y\right\} \right)=0$. 

Let us now recall the notion of disintegration of a measure. Suppose
that $\left(X,\mu\right)$ is a probability Borel measure space on
a Polish space $X$. Let $f:X\to Y$ be a Borel map onto another Polish
space $Y$ and denote as $f^{*}\mu$ the push-forward measure via
$f$. Then the Disintegration Theorem tells us that there is a system
$\{F_{y}\}_{y\in Y}$ of probability Borel measures on $X$ such that
\begin{enumerate}[label=(D\arabic*)]
\item \label{enu:dis1}$F_{y}\left(f^{-1}\left(y\right)\right)=1$ for
every $y\in Y$, and
\item \label{enu:dis2}$\int_{X}h\left(x\right)\:\mathrm{d}\mu(x)=\int_{Y}\left(\int_{X}h\left(x\right)\:\mathrm{d}F_{y}(x)\right)\:\mathrm{d}f^{*}\mu(y)$
for every Borel map $h:X\to\left[0,1\right]$.
\end{enumerate}
We will use the disintegration exclusively in the situation where
$X=\left[0,1\right]^{2}$, $f$ is the projection on the $i$-th coordinate,
$\mu=\Phi$ and $f^{*}\mu=\Lambda_{i}$ where $\Phi$ is a witness
for the fact that $\Lambda_{1}$ is at least as flat as $\Lambda_{2}$.
When we use the variable $x$ for the first coordinate and $y$ for
the second coordinate then we obtain a disintegration $\left\{ \Phi_{x}^{1}\right\} _{x\in\left[0,1\right]}$
and $\left\{ \Phi_{y}^{2}\right\} _{y\in\left[0,1\right]}$. Moreover,
to simplify notation we will always assume that each $\Phi_{z}^{i}$
lives on the interval $\left[0,1\right]$ instead of the corresponding
(horizontal or vertical) strip. For example in case of disintegration
on the second coordinate, the two conditions above then look like
this:
\begin{enumerate}[label=(D\arabic*)]
\item $\Phi_{y}^{2}\left([0,1]\right)=1$ for every $y\in[0,1]$, and
\item $\int_{[0,1]^{2}}h\left(x,y\right)\:\mathrm{d}\Phi(x,y)=\int_{[0,1]}\left(\int_{[0,1]}h\left(x,y\right)\:\mathrm{d}\Phi_{y}^{2}(x)\right)\:d\Lambda_{2}(y)$
for every Borel map $h:[0,1]^{2}\to\left[0,1\right]$.
\end{enumerate}
\begin{lem}
\label{lem:previ}Suppose that $\Lambda_{1}$ is at least as flat
as $\Lambda_{2}$, witnessed by a measure $\Phi$. Then for $\Lambda_{1}$-almost
every $x\in\left[0,1\right]$ we have
\begin{equation}
x=\int_{\left[0,1\right]}y\:\mathrm{d}\Phi_{x}^{1}\left(y\right).\label{eq:xfromy}
\end{equation}
\end{lem}

\begin{proof}
Assume not, then we may assume without loss of generality that there
is a set $D\subseteq\left[0,1\right]$ of positive $\Lambda_{1}$-measure
such that $x>\int_{\left[0,1\right]}y\:\mathrm{d}\Phi_{x}^{1}\left(y\right)$
for each $x\in D$. Then we have 
\[
\int_{D\times\left[0,1\right]}x\:\mathrm{d}\Phi(x,y)=\int_{D}x\:\mathrm{d}\Lambda_{1}(x)>\int_{D}\left(\int_{\left[0,1\right]}y\:\mathrm{d}\Phi_{x}^{1}\left(y\right)\right)\:\mathrm{d}\Lambda_{1}(x)\overset{\ref{enu:dis2}}{=}\int_{D\times\left[0,1\right]}y\:\mathrm{d}\Phi(x,y)
\]
which contradicts ~(\ref{eq:masspreserve}). 
\end{proof}
\begin{lem}
\label{lem:flatsamemeasure}Suppose that $\Lambda_{1}$ is at least
as flat as $\Lambda_{2}$. Then we have $\Lambda_{1}([0,1])=\Lambda_{2}([0,1])$.
\end{lem}

\begin{proof}
Let $\Psi$ be a witness that $\Lambda_{1}$ is at least as flat as
$\Lambda_{2}$. Then the marginal condition of Definition~\ref{def:flatter}
tells us that
\[
\Lambda_{1}([0,1])=\Psi([0,1]\times[0,1])=\Lambda_{2}([0,1])\;.
\]
\end{proof}
\begin{lem}
\label{lem:diagonal}Suppose that $\Lambda_{1}$ and $\Lambda_{2}$
are finite measures on $\left[0,1\right]$, and that $\Lambda_{1}$
is at least as flat as $\Lambda_{2}$. Then $\Lambda_{1}=\Lambda_{2}$
if and only if the only measure which witnesses that $\Lambda_{1}$
is at least as flat as $\Lambda_{2}$ is diagonal.
\end{lem}

\begin{proof}
Suppose that we have a diagonal measure $\Phi$ whose marginals are
$\Lambda_{1}$ and $\Lambda_{2}$. Then for every measurable set $D$
we have 
\[
\Lambda_{1}(D)=\Phi(D\times[0,1])=\Phi(D\times D)=\Phi([0,1]\times D)=\Lambda_{2}(D)\;,
\]
and so $\Lambda_{1}=\Lambda_{2}$.

On the other hand, suppose that we have a non-diagonal measure $\Phi$
whose marginals are $\Lambda_{1}$ and $\Lambda_{2}$. By Lemma \ref{lem:flatsamemeasure}
we may assume that both $\Lambda_{1}$ and $\Lambda_{2}$ are probability
measures. Let us fix a strictly convex function $f:[0,1]\rightarrow[0,1]$.
Let us consider the disintegration $\left\{ \Phi_{x}^{1}\right\} _{x\in\left[0,1\right]}$
of $\Phi$. Recall that by~\ref{enu:dis1}, for each $x\in[0,1]$,
$\Phi_{x}^{1}$ is a probability measure. In particular, Jensen's
inequality gives us

\[
f\left(\int y\:\mathrm{d}\Phi_{x}^{1}\left(y\right)\right)\le\int f\left(y\right)\:\mathrm{d}\Phi_{x}^{1}\left(y\right).
\]
Observe that for a positive $\Lambda_{1}$-measure of $x$'s, we have
that $\Phi_{x}^{1}$ is not a Dirac measure. For each such $x$, the
inequality above is strict. Then we have
\begin{align*}
\int f(x)\:\mathrm{d}\Lambda_{1}(x)= & \int f\left(\int y\:\mathrm{d}\Phi_{x}^{1}\left(y\right)\right)\:\mathrm{d}\Lambda_{1}(x)\\
\JUSTIFY{Jensen's\thinspace inequality\,as\,above}< & \int\left(\int f\left(y\right)\:\mathrm{d}\Phi_{x}^{1}\left(y\right)\right)\:\mathrm{d}\Lambda_{1}(x)=\int f(y)\:\mathrm{d}\Phi(x,y)=\int f\left(y\right)\:\mathrm{d}\Lambda_{2}(y).
\end{align*}
In particular, we can conclude that $\Lambda_{1}\neq\Lambda_{2}$.
\end{proof}
\begin{lem}
\label{lem:flattranisitve}Suppose that $\Lambda_{\mathrm{A}},\Lambda_{\mathrm{B}},\Lambda_{\mathrm{C}}$
are three finite measures on $\left[0,1\right]$. Suppose that$\Lambda_{\mathrm{A}}$
is at least as flat as $\Lambda_{\mathrm{B}}$ and that $\Lambda_{\mathrm{B}}$
is at least as flat as $\Lambda_{\mathrm{C}}$. Then $\Lambda_{\mathrm{A}}$
is at least as flat as $\Lambda_{\mathrm{C}}$. If, in addition at
least one of these flatness relations is strict, then $\Lambda_{\mathrm{A}}$
is strictly flatter than $\Lambda_{\mathrm{C}}$.
\end{lem}

\begin{proof}
By Lemma~\ref{lem:flatsamemeasure}, the measures $\Lambda_{\mathrm{A}}$,
$\Lambda_{\mathrm{B}}$ and $\Lambda_{\mathrm{C}}$ have the same
total measure, say $\Lambda_{\mathrm{A}}([0,1])=\Lambda_{\mathrm{B}}([0,1])=\Lambda_{\mathrm{C}}([0,1])=m$.
By multiplying these measures, and all the corresponding measures
witnessing the flatness relation by $\frac{1}{m}$, it is enough to
restrict ourselves to the case of probability measures from now on.

Let $\widetilde{\Phi}$ be a witness that $\Lambda_{\mathrm{A}}$
is at least as flat as $\Lambda_{\mathrm{B}}$, and let $\widehat{\Phi}$
be a witness that $\Lambda_{\mathrm{B}}$ is at least as flat as $\Lambda_{\mathrm{C}}$.
Let us disintegrate $\widetilde{\Phi}$ on the second coordinate,
and $\widehat{\Phi}$ on the first coordinate. This gives us families
$\left\{ \widetilde{\Phi}_{y}^{2}\right\} _{y\in[0,1]}$ and $\left\{ \widehat{\Phi}_{y}^{1}\right\} _{y\in[0,1]}$
of probability Borel measures on $[0,1]$. This way, for each measurable
set $S\subseteq[0,1]^{2}$ we have
\[
\widetilde{\Phi}(S)=\int_{y}\:\widetilde{\Phi}_{y}^{2}\left(\left\{ x\::\:(x,y)\in S\right\} \right)\;\mathrm{d}\Lambda_{\mathrm{B}}(y)\;\text{and}\;\widehat{\Phi}(S)=\int_{y}\:\widehat{\Phi}_{y}^{1}\left(\left\{ z\::\:(y,z)\in S\right\} \right)\;\mathrm{d}\Lambda_{\mathrm{B}}(y)\;.
\]
Now, for every set $S\subseteq[0,1]^{2}$ of the form $S=\bigsqcup_{i=1}^{n}A_{i}\times B_{i}$
where $A_{i}$ and $B_{i}$ are measurable subsets of $[0,1]$, we
define
\begin{equation}
\Xi(S):=\sum_{i=1}^{n}\int_{y\in[0,1]}\:\widetilde{\Phi}_{y}^{2}\left(A_{i}\right)\cdot\widehat{\Phi}_{y}^{1}\left(B_{i}\right)\;\mathrm{d}\Lambda_{\mathrm{B}}(y)\;.\label{eq:defXI}
\end{equation}
It is tedious but straightforward to verify that the value of $\Xi(S)$
does not depend on the choice of the decomposition $S=\bigsqcup_{i=1}^{n}A_{i}\times B_{i}$.
Then Carathéodory's extension theorem allows us to extend $\Xi$ to
a Borel measure on $[0,1]^{2}$, which we still denote by $\Xi$.
We claim that this measure witnesses that $\Lambda_{\mathrm{A}}$
is at least as flat as $\Lambda_{\mathrm{C}}$. Firstly, let us check
that the marginals of $\Xi$ are $\Lambda_{\mathrm{A}}$ and $\Lambda_{\mathrm{C}}$,
respectively. For $D\subseteq[0,1]$, we have that
\begin{align*}
\Xi\left(D\times[0,1]\right) & \overset{\ref{eq:defXI}}{=}\int_{y\in[0,1]}\:\widetilde{\Phi}_{y}^{2}(D)\cdot\widehat{\Phi}_{y}^{1}\left([0,1]\right)\;\mathrm{d}\Lambda_{\mathrm{B}}(y)\\
\JUSTIFY{by\,\ref{enu:dis2},\,\text{\ensuremath{\widehat{\Phi}_{y}^{1}}([0,1])=1}} & =\int_{y\in[0,1]}\:\widetilde{\Phi}_{y}^{2}(D)\;\mathrm{d}\Lambda_{\mathrm{B}}(y)=\widetilde{\Phi}\left(D\times[0,1]\right)=\Lambda_{\mathrm{A}}\left(D\right)\;.
\end{align*}
Similarly, one can verify that $\Lambda_{C}(D)=\Xi\left([0,1]\times D\right)$. 

Let $D\subseteq[0,1]$. We have the following
\begin{align*}
\int_{D\times[0,1]}x\:\mathrm{d}\Xi(x,z)= & \int_{D}x\:\mathrm{d}\Lambda_{\mathrm{A}}(x)\\
= & \int_{D\times[0,1]}x\:\mathrm{d}\tilde{\Phi}(x,y)\\
= & \int_{D\times[0,1]}y\:\mathrm{d}\tilde{\Phi}(x,y)=\int_{[0,1]}\left(\int_{[0,1]}y\cdot{\bf 1}_{D\times[0,1]}(x,y)\:\mathrm{d}\widetilde{\Phi}_{y}^{2}(x)\right)\:\mathrm{d}\Lambda_{\mathrm{B}}(y)\\
= & \int_{[0,1]}y\widetilde{\Phi}_{y}^{2}(D)\:\mathrm{d}\Lambda_{B}(y)\\
= & \int_{[0,1]}\left(\int_{[0,1]}z\:\mathrm{d}\widehat{\Phi}_{y}^{1}\left(z\right)\right)\widetilde{\Phi}_{y}^{2}(D)\:\mathrm{d}\Lambda_{\mathrm{B}}(y)\\
= & \int_{D\times[0,1]}z\:\mathrm{d}\Xi(x,z)
\end{align*}
where the last equality follows from the following claim.
\begin{claim}
Let $g:\left[0,1\right]\to\left[0,1\right]$ be a measurable function.
Then 
\[
\int_{[0,1]}\left(\int_{[0,1]}g\left(z\right)\:\mathrm{d}\widehat{\Phi}_{y}^{1}\left(z\right)\right)\widetilde{\Phi}_{y}^{2}(D)\:\mathrm{d}\Lambda_{\mathrm{B}}(y)=\int_{D\times[0,1]}g\left(z\right)\:\mathrm{d}\Xi(x,z).
\]
\end{claim}

\begin{proof}
Let $A\subseteq\left[0,1\right]$ be a measurable set and consider
its characteristic function ${\bf 1}_{A}$. We have
\begin{align*}
\int_{[0,1]}\left(\int_{[0,1]}{\bf 1}_{A}\left(z\right)\:\mathrm{d}\widehat{\Phi}_{y}^{1}\left(z\right)\right)\widetilde{\Phi}_{y}^{2}(D)\:\mathrm{d}\Lambda_{\mathrm{B}}(y)= & \int{}_{[0,1]}\widehat{\Phi}_{y}^{1}\left(A\right)\widetilde{\Phi}_{y}^{2}(D)\:\mathrm{d}\Lambda_{\mathrm{B}}(y)\\
= & \Xi\left(D\times A\right)=\int_{D\times[0,1]}{\bf 1}_{A}\left(z\right)\:\mathrm{d}\Xi(x,z).
\end{align*}
This implies that the claim holds for every step function. Assume
now that $g_{n}\to g$ uniformly and $\left\{ g_{n}\right\} _{n\in\mathbb{N}}$
are step functions. We have 
\begin{align*}
\int_{[0,1]}\left(\int_{[0,1]}g\left(z\right)\:\mathrm{d}\widehat{\Phi}_{y}^{1}\left(z\right)\right)\widetilde{\Phi}_{y}^{2}(D)\:\mathrm{d}\Lambda_{\mathrm{B}}(y) & =\int_{[0,1]}\left(\int_{[0,1]}\lim_{n}g_{n}\left(z\right)\:\mathrm{d}\widehat{\Phi}_{y}^{1}\left(z\right)\right)\widetilde{\Phi}_{y}^{2}(D)\:\mathrm{d}\Lambda_{\mathrm{B}}(y)\\
 & =\int_{[0,1]}\lim_{n}\left(\int_{[0,1]}g_{n}\left(z\right)\:\mathrm{d}\widehat{\Phi}_{y}^{1}\left(z\right)\right)\widetilde{\Phi}_{y}^{2}(D)\:\mathrm{d}\Lambda_{\mathrm{B}}(y)\\
 & =\lim_{n}\int_{[0,1]}\left(\int_{[0,1]}g_{n}\left(z\right)\:\mathrm{d}\widehat{\Phi}_{y}^{1}\left(z\right)\right)\widetilde{\Phi}_{y}^{2}(D)\:\mathrm{d}\Lambda_{\mathrm{B}}(y)\\
 & =\lim_{n}\int_{D\times[0,1]}g_{n}\left(z\right)\:\mathrm{d}\Xi(x,z)=\int_{D\times[0,1]}g\left(z\right)\:\mathrm{d}\Xi(x,z)
\end{align*}
(the second, third and fifth equalities follow by the uniform convergence)
and that finishes the proof.
\end{proof}
It remains to prove the additional part. So suppose that either $\Lambda_{\mathrm{A}}$
is strictly flatter than $\Lambda_{\mathrm{B}}$ or $\Lambda_{\mathrm{B}}$
is strictly flatter than $\Lambda_{\mathrm{C}}$. Fix a strictly convex
function $f\colon[0,1]\rightarrow[0,1]$. In the very same way as
in the proof of Lemma~\ref{lem:diagonal} it follows that

\[
\int_{[0,1]}f(x)\;\mathrm{d}\Lambda_{\mathrm{A}}(x)\le\int_{[0,1]}f(y)\;\mathrm{d}\Lambda_{\mathrm{B}}(y)\le\int_{[0,1]}f(z)\;\mathrm{d}\Lambda_{\mathrm{C}}(z)\;,
\]
and at least one of the inequalities above is strict. Therefore $\Lambda_{\mathrm{A}}\neq\Lambda_{\mathrm{C}}$.
\end{proof}
The next two propositions answer the question of relating $\boldsymbol{\Phi}_{U}$
to $\boldsymbol{\Phi}_{W}$ and $\boldsymbol{\Upsilon}_{U}$ to $\boldsymbol{\Upsilon}_{W}$
using the above concept of flatter measures. While we consider these
result interesting per se, let us note that in~\cite{DGHRR:Parameters}
we give several quick and fairly powerful applications of these results.
For example, we show that the result of Doležal and Hladký can be
extended to discontinuous functions, as advertised in Footnote~\ref{fn:INT}.
\begin{prop}
\label{prop:flatter}Suppose that we have two graphons $U\preceq W$.
Then the measure $\boldsymbol{\Phi}_{U}$ is at least as flat as the
measure $\boldsymbol{\Phi}_{W}$. Similarly, the measure $\boldsymbol{\Upsilon}_{U}$
is at least as flat as the measure $\boldsymbol{\Upsilon}_{W}$. Lastly,
if $U\prec W$ then $\boldsymbol{\Phi}_{U}$ is strictly flatter than
$\boldsymbol{\Phi}_{W}$.
\end{prop}

\begin{example}
We cannot conclude that $\boldsymbol{\Upsilon}_{U}$ is strictly flatter
than $\boldsymbol{\Upsilon}_{W}$ if $U\prec W$. To this end, it
is enough to take $U$ the constant-$p$ graphon (for some $p\in\left(0,1\right)$)
and $W$ some $p$-regular but non-constant graphon. Then $\boldsymbol{\Upsilon}_{U}$
and $\boldsymbol{\Upsilon}_{W}$ are both equal to the Dirac measure
on $p$.
\end{example}

\begin{example}
A probabilist might say that the information inherited from $\boldsymbol{\Phi}_{W}$
to $\boldsymbol{\Phi}_{U}$ is only <<annealed>>, and not <<quenched>>.
Let us explain this on an example. Suppose that $U$ is the constant-$\frac{1}{2}$
graphon and $W$ attains each of the values~$0,\frac{1}{2},1$ on
sets of measure $\frac{1}{3}$ each. Obviously, we have that $U\prec W$
and thus there is a sequence $W^{\pi_{1}},W^{\pi_{2}},W^{\pi_{3}},\ldots\WEAKCONV U\equiv\frac{1}{2}$.
Now, observe that there are many different scenarios where the values
$\frac{1}{2}$ can arise in the limit, of which we give two extreme
ones. The first possibility is that around each $(x,y)\in\Omega^{2}$,
we have alternations of values $0$'s, $\frac{1}{2}$'s, and $1$'s
in the graphons $W^{\pi_{n}}$, each with frequency $\frac{1}{3}$.
The second possibility is that for the measure $\frac{1}{3}$ of $(x,y)$'s,
the graphons $W^{\pi_{n}}$ attain values only $\frac{1}{2}$ around
$(x,y)$, and for the remaining $(x,y)$'s of measure $\frac{2}{3}$,
we have alternations of values $0$'s, and $1$'s, each with density
$\frac{1}{2}$.
\end{example}

In the proof of Proposition~\ref{prop:flatter}, we will need some
basic facts about the weak{*} convergence of measures on $[0,1]^{2}$
(this convergence is also often called weak convergence or narrow
convergence in the literature) which we recall here. We say that a
bounded sequence of finite positive measures $\Psi_{1},\Psi_{2},\Psi_{3},\ldots$
on $[0,1]^{2}$ converges in the weak{*} topology to a finite positive
measure $\Psi$ if for every continuous real function $f$ defined
on $[0,1]^{2}$ we have
\[
\lim_{n\rightarrow\infty}\int_{[0,1]^{2}}f(x,y)\:\mathrm{d}\Psi_{n}(x,y)=\int_{[0,1]^{2}}f(x,y)\:\mathrm{d}\Psi(x,y).
\]
This definition has many equivalent reformulations but we will need
only the following one: A sequence $\Psi_{1},\Psi_{2},\Psi_{3},\ldots$
converges to $\Psi$ in the weak{*} topology if and only if $\lim_{n\rightarrow\infty}\Psi_{n}(A)=\Psi(A)$
for every Borel subset $A$ of $[0,1]^{2}$ which satisfies that the
$\Psi$-measure of its boundary in $[0,1]^{2}$ is $0$.\footnote{The boundary of a set $A\subseteq[0,1]^{2}$ is defined as the set
of all points in the closure of $A$ which are not interior points
of $A$. Note that the interior points are considered only from point
of view of the topological space $[0,1]^{2}$ (not $\mathbb{R}^{2}$).
So for example, the boundary of the closed set $A=[0,\tfrac{1}{2}]\times[0,1]$
is $\{\tfrac{1}{2}\}\times[0,1]$ as all other points from $A$ are
interior points of $A$ in $[0,1]^{2}$.} Recall also that every sequence of probability measures on a compact
separable space has a weak{*} convergent subsequence.

The following lemma is the key ingredient of the proof of Lemma~\ref{lem:strictplusnonstrict}.
However, it may be of independent interest as it connects our research
with Choquet theory. (We will not use Choquet theory, and the rest
of this paragraph is meant only to hint the connection.) Recall that
from the point of view of Choquet theory, if $\Lambda_{1},\Lambda_{2}$
are finite positive measures on some compact convex subset $C$ of
a normed space then we say that \emph{$\Lambda_{1}$ is smaller than
$\Lambda_{2}$} if $\int_{C}f\:\mathrm{d}\Lambda_{1}\le\int_{C}f\:\mathrm{d}\Lambda_{2}$
for every continuous convex function $f\colon C\rightarrow\mathbb{R}$.
The next lemma states that the relation <<being flatter>> is naturally
embedded into this Choquet ordering of measures (when the compact
convex set $C$ is the unit interval $[0,1]$).
\begin{lem}
\label{lem:stone-weierstrass}Let $\Lambda_{1},\Lambda_{2}$ be finite
measures on $[0,1]$ such that $\Lambda_{1}$ is at least as flat
as $\Lambda_{2}$. Then for every continuous convex function $f\colon[0,1]\rightarrow\mathbb{R}$
it holds $\int_{[0,1]}f\:\mathrm{d}\Lambda_{1}\le\int_{[0,1]}f\:\mathrm{d}\Lambda_{2}$.
Moreover, if $\Lambda_{1}$ is strictly flatter than $\Lambda_{2}$
then there is a continuous convex function $g\colon[0,1]\rightarrow\mathbb{R}$
such that $\int_{[0,1]}g\:\mathrm{d}\Lambda_{1}<\int_{[0,1]}g\:\mathrm{d}\Lambda_{2}$.
\end{lem}

\begin{proof}
Let $\Psi$ be a witness for $\Lambda_{1}$ being at least as flat
as $\Lambda_{2}$, as in Definition~\ref{def:flatter}. Fix a continuous
convex function $f\colon[0,1]\rightarrow\mathbb{R}$ and $\varepsilon>0$.
Find a natural number $n$ such that $f(x)-f(y)<\varepsilon$ whenever
$x,y\in[0,1]$ are such that $|x-y|\le\frac{2}{n}$. Let $[0,1]=I_{1}\sqcup I_{2}\sqcup\ldots\sqcup I_{n}$
be a partition of $[0,1]$ into pairwise disjoint intervals of lengths
$\frac{1}{n}$. For every $i$ fix a point $x_{i}\in I_{i}$. As for
every $i$ we have
\[
\int_{\left(x,y\right)\in I_{i}\times\left[0,1\right]}x\;\mathrm{d}\Psi(x,y)=\int_{\left(x,y\right)\in I_{i}\times\left[0,1\right]}y\;\mathrm{d}\Psi(x,y)\:,
\]
\[
\left|x_{i}\Psi(I_{i}\times[0,1])-\int_{\left(x,y\right)\in I_{i}\times\left[0,1\right]}x\;\mathrm{d}\Psi(x,y)\right|\le\frac{1}{n}\cdot\Psi(I_{i}\times[0,1])
\]
and similarly
\[
\left|\sum_{j=1}^{n}y_{j}\Psi(I_{i}\times I_{j})-\int_{\left(x,y\right)\in I_{i}\times\left[0,1\right]}y\;\mathrm{d}\Psi(x,y)\right|\le\frac{1}{n}\cdot\Psi(I_{i}\times[0,1])\:,
\]
it follows that
\[
\left|x_{i}-\sum_{j=1}^{n}\frac{\Psi(I_{i}\times I_{j})}{\Psi(I_{i}\times[0,1])}y_{j}\right|\le\frac{2}{n}\;.
\]
So by the convexity of $f$ we have for every $i$ that
\[
f(x_{i})\stackrel{\varepsilon}{\approx}f\left(\sum_{j=1}^{n}\frac{\Psi(I_{i}\times I_{j})}{\Psi(I_{i}\times[0,1])}y_{j}\right)\le\sum_{j=1}^{n}\frac{\Psi(I_{i}\times I_{j})}{\Psi(I_{i}\times[0,1])}f(y_{j})\;.
\]
Therefore
\begin{eqnarray*}
\int_{x\in[0,1]}f(x)\:\mathrm{d}\Lambda_{1}(x) &  & =\sum_{i=1}^{n}\int_{x\in I_{i}}f(x)\:\mathrm{d}\Lambda_{1}(x)\stackrel{\varepsilon\cdot\Psi([0,1]^{2})}{\approx}\sum_{i=1}^{n}f(x_{i})\Psi(I_{i}\times[0,1])\\
 &  & \stackrel{\varepsilon\cdot\Psi([0,1]^{2})}{\approx}\sum_{i=1}^{n}f\left(\sum_{j=1}^{n}\frac{\Psi(I_{i}\times I_{j})}{\Psi(I_{i}\times[0,1])}y_{j}\right)\Psi(I_{i}\times[0,1])\\
 &  & \le\sum_{i=1}^{n}\sum_{j=1}^{n}f(y_{j})\Psi(I_{i}\times I_{j})\\
 &  & =\sum_{j=1}^{n}f(y_{j})\Psi([0,1]\times I_{j})\\
 &  & \stackrel{\varepsilon\cdot\Psi([0,1]^{2})}{\approx}\int_{y\in[0,1]}f(y)\:\mathrm{d}\Lambda_{2}(y)\;.
\end{eqnarray*}
As this is true for every $\varepsilon>0$ we conclude that $\int_{[0,1]}f\:\mathrm{d}\Lambda_{1}\le\int_{[0,1]}f\:\mathrm{d}\Lambda_{2}$.

Now suppose that $\Lambda_{1}$ is strictly flatter than $\Lambda_{2}$.
Then there is a continuous function $h\colon[0,1]\rightarrow\mathbb{R}$
such that $\int_{[0,1]}h\:\mathrm{d}\Lambda_{1}\neq\int_{[0,1]}h\:\mathrm{d}\Lambda_{2}$.
Recall that functions of the form $g_{1}-g_{2}$, where both $g_{1},g_{2}\colon[0,1]\rightarrow\mathbb{R}$
are continuous and convex, are uniformly dense in the space of all
continuous functions $h\colon[0,1]\rightarrow\mathbb{R}$. Indeed,
this follows from the Stone\textendash Weierstrass theorem as every
polynomial function $p\colon[0,1]\rightarrow\mathbb{R}$ can be written
as $p=g_{1}-g_{2}$ for $g_{1},g_{2}$ continuous convex (just define
$g_{1}(x):=p(x)+Kx^{2}$ and $g_{2}(x):=Kx^{2}$ where $K$ is a sufficiently
large constant). So there is a continuous convex function $g\colon[0,1]\rightarrow\mathbb{R}$
such that $\int_{[0,1]}g\:\mathrm{d}\Lambda_{1}\neq\int_{[0,1]}g\:\mathrm{d}\Lambda_{2}$.
But then the previous part of the proof gives us that $\int_{[0,1]}g\:\mathrm{d}\Lambda_{1}<\int_{[0,1]}g\:\mathrm{d}\Lambda_{2}$.
\end{proof}
Let us now give an intuitively clear lemma.
\begin{lem}
\label{lem:strictplusnonstrict}Let $\Theta^{A},\Delta^{A},\Theta^{B},\Delta^{B}$
be four finite measures on $[0,1]$. Suppose that $\Theta^{A}$ is
strictly flatter than $\Delta^{A}$ and that $\Theta^{B}$ is at least
as flat as $\Delta^{B}$. Then the measure $\Theta^{A}+\Delta^{A}$
is strictly flatter than $\Theta^{B}+\Delta^{B}$.
\end{lem}

\begin{proof}
Let $\Psi^{A}$ (resp.~$\Psi^{B}$) be a witness for $\Theta^{A}$
(resp.~$\Theta^{B}$) being as flat as $\Delta^{A}$ (resp.~$\Delta^{B}$),
as in Definition~\ref{def:flatter}. Then $\Psi^{A}+\Psi^{B}$ shows
that $\Theta^{A}+\Delta^{A}$ is at least as flat as $\Theta^{B}+\Delta^{B}$.
It remains to show that the relation is actually strict.

Let $g\colon[0,1]\rightarrow\mathbb{R}$ be a continuous convex function
such that $\int_{[0,1]}g\:\mathrm{d}\Theta^{A}<\int_{[0,1]}g\:\mathrm{d}\Delta^{A}$.
Such a function exists by Lemma~\ref{lem:stone-weierstrass}. Then
another application of Lemma~\ref{lem:stone-weierstrass} gives that
$\int_{[0,1]}g\:\mathrm{d}(\Theta^{A}+\Delta^{A})<\int_{[0,1]}g\:\mathrm{d}(\Theta^{B}+\Delta^{B})$,
and so the measures $\Theta^{A}+\Delta^{A}$ and $\Theta^{B}+\Delta^{B}$
are not the same.
\end{proof}
The proof of Proposition~\ref{prop:flatter} relies on the following
lemma.
\begin{lem}
\label{lem:flatmeasures}Let $B$ be a separable atomless finite measure
space with the measure $\beta$, and let $\left(f_{n}:B\rightarrow[0,1]\right)_{n}$
and $f:B\rightarrow[0,1]$ be measurable functions such that $\lim_{n\rightarrow\infty}\int_{A}f_{n}(x)\;d\beta(x)=\int_{A}f(x)\;d\beta(x)$
for every measurable $A\subseteq B$. Suppose that $\Delta_{n}$ and
$\Delta$ are the pushforward measures of $f_{n}$ and $f$, $\Delta_{n}(L):=\beta\left(f_{n}^{-1}(L)\right)$,
$\Delta(L):=\beta\left(f^{-1}(L)\right)$. Suppose that the measures
$\Delta_{n}$ weak{*} converge to a probability measure $\Delta^{*}$.
Then $\Delta$ is at least as flat as $\Delta^{*}$.
\end{lem}

\begin{proof}
For every natural number $n$ and every measurable subset $A$ of
$[0,1]^{2}$ we define 
\[
\Psi_{n}(A)=\beta\left(\left\{ x\in B\colon\left(f(x),f_{n}(x)\right)\in A\right\} \right).
\]
Clearly every $\Psi_{n}$ is a measure on $[0,1]^{2}$. Let $\Psi$
be some weak{*} accumulation point of the sequence $\Psi_{1},\Psi_{2},\Psi_{3},\ldots$.
Without loss of generality, we may assume that the sequence $\Psi_{1},\Psi_{2},\Psi_{3},\ldots$
converges to $\Psi$. Let $Z$ be the set consisting of all points
$z\in(0,1)$ for which either $\Psi(\{z\}\times[0,1])>0$ or $\Psi([0,1]\times\{z\})>0$
or $\Delta^{*}(\{z\})>0$. Then $Z$ is at most countable. So if $\mathcal{I}$
is the system of all intervals $I\subseteq[0,1]$ whose endpoints
do not belong to $Z$ then $\mathcal{I}$ is closed under taking finite
intersections and it generates the sigma-algebra of all Borel subsets
of $[0,1]$. Moreover, whenever $I,J\in\mathcal{I}$ then the boundary
of $I\times J$ is of $\Psi$-measure $0$ and of $\Delta^{*}$-measure
$0$. Denote by $\Psi^{x}$ and $\Psi^{y}$ the marginals of $\Psi$
on the first and on the second coordinate, respectively. Then for
every $I\in\mathcal{I}$ we have
\begin{eqnarray*}
\Psi^{x}(I) & = & \Psi(I\times[0,1])=\lim_{n\rightarrow\infty}\Psi_{n}(I\times[0,1])\\
 & = & \lim_{n\rightarrow\infty}\beta\left(\left\{ x\in B\colon\left(f(x),f_{n}(x)\right)\in I\times[0,1]\right\} \right)\\
 & = & \lim_{n\rightarrow\infty}\Delta(I)=\Delta(I).
\end{eqnarray*}
As this is true for every $I\in\mathcal{I}$, it clearly follows that
$\Psi^{x}=\Delta$. On the other hand, for every $I\in\mathcal{I}$
we have that
\begin{eqnarray*}
\Psi^{y}(I) & = & \Psi([0,1]\times I)=\lim_{n\rightarrow\infty}\Psi_{n}([0,1]\times I)\\
 & = & \lim_{n\rightarrow\infty}\beta\left(\left\{ x\in B\colon\left(f(x),f_{n}(x)\right)\in[0,1]\times I\right\} \right)\\
 & = & \lim_{n\rightarrow\infty}\Delta_{n}(I)\\
\JUSTIFY{\text{\ensuremath{\Delta_{n}\WEAKCONV\Delta^{*}}}} & = & \Delta^{*}(I).
\end{eqnarray*}
So, again we have $\Psi^{y}=\Delta^{*}$. To finish the proof, it
remains to show that $\Psi$ satisfies~(\ref{eq:masspreserve}).
That is, we need to show that 
\begin{equation}
\int_{\left(x,y\right)\in C\times\left[0,1\right]}x\:\mathrm{d}\Psi(x,y)=\int_{\left(x,y\right)\in C\times\left[0,1\right]}y\:\mathrm{d}\Psi(x,y)\label{eq:margins}
\end{equation}
for every Borel measurable subset $C$ of $[0,1]$. Again, it is enough
to show~(\ref{eq:margins}) only for every $C\in\mathcal{I}$. So
fix $C\in\mathcal{I}$, $\varepsilon>0$ and find some partition $\{I_{1},I_{2},\ldots,I_{m}\}$
of the interval $[0,1]$ into intervals from $\mathcal{I}$ of lengths
smaller than $\varepsilon$. Suppose first that the interval $C$
is of length smaller than $\varepsilon$. Fix some points $x_{0}\in C$
and $y_{j}\in I_{j}$ (for every $j$). Then we have
\begin{eqnarray*}
\int_{\left(x,y\right)\in C\times\left[0,1\right]}y\:\mathrm{d}\Psi(x,y) &  & =\sum_{j=1}^{m}\int_{\left(x,y\right)\in C\times I_{j}}y\:\mathrm{d}\Psi(x,y)\\
 &  & \stackrel{\varepsilon\cdot\Psi(C\times[0,1])}{\approx}\sum_{j=1}^{m}y_{j}\Psi(C\times I_{j})=\sum_{j=1}^{m}y_{j}\lim_{n\rightarrow\infty}\Psi_{n}(C\times I_{j})\\
 &  & =\sum_{j=1}^{m}y_{j}\lim_{n\rightarrow\infty}\beta\left(\left\{ x\in B\colon\left(f(x),f_{n}(x)\right)\in C\times I_{j}\right\} \right)\\
 &  & =\sum_{j=1}^{m}\lim_{n\rightarrow\infty}\int_{x\in f^{-1}(C)\cap f_{n}^{-1}(I_{j})}y_{j}\:\mathrm{d}\beta\\
 &  & \stackrel{\varepsilon\cdot\beta(f^{-1}(C))}{\approx}\lim_{n\rightarrow\infty}\int_{x\in f^{-1}(C)}f_{n}(x)\:\mathrm{d}\beta=\int_{x\in f^{-1}(C)}f(x)\:\mathrm{d}\beta\\
 &  & \stackrel{\varepsilon\cdot\beta(f^{-1}(C))}{\approx}\int_{x\in f^{-1}(C)}x_{0}\:\mathrm{d}\beta=x_{0}\cdot\beta\left(f^{-1}(C)\right)=x_{0}\cdot\Delta(C)=x_{0}\cdot\Psi^{x}(C)\\
 &  & =\int_{\left(x,y\right)\in C\times\left[0,1\right]}x_{0}\:\mathrm{d}\Psi(x,y)\stackrel{\varepsilon\cdot\Psi(C\times[0,1])}{\approx}\int_{\left(x,y\right)\in C\times\left[0,1\right]}x\:\mathrm{d}\Psi(x,y),
\end{eqnarray*}
so
\[
\int_{\left(x,y\right)\in C\times\left[0,1\right]}y\:\mathrm{d}\Psi(x,y)\stackrel{2\varepsilon\cdot(\Psi(C\times[0,1])+\beta(f^{-1}(C)))}{\approx}\int_{\left(x,y\right)\in C\times\left[0,1\right]}x\:\mathrm{d}\Psi(x,y).
\]
In general, the interval $C$ can be decomposed into subintervals
belonging to $\mathcal{I}$ such that each of them is of length smaller
than $\varepsilon$, and by additivity of integration we conclude
that
\[
\int_{\left(x,y\right)\in C\times\left[0,1\right]}y\:\mathrm{d}\Psi(x,y)\stackrel{2\varepsilon\cdot(\Psi([0,1]^{2})+\beta(B))}{\approx}\int_{\left(x,y\right)\in C\times\left[0,1\right]}x\:\mathrm{d}\Psi(x,y).
\]
 As this is true for every $\varepsilon>0$, we have verified~(\ref{eq:margins}).
\end{proof}
We are now ready to prove Proposition~\ref{prop:flatter}.
\begin{proof}[Proof of Proposition~\ref{prop:flatter}, non-strict part]
Suppose that we have two graphons $U,W:\Omega^{2}\rightarrow[0,1]$,
where $U\preceq W$. Then there exist measure preserving bijections
$\pi_{1},\pi_{2},\pi_{3},\ldots$ on $\Omega$ so that 
\begin{equation}
W^{\pi_{n}}\WEAKCONV U\;.\label{eq:peece}
\end{equation}
Now, we can apply Lemma~\ref{lem:flatmeasures} with $B:=\Omega^{2}$,
$f_{n}:=W^{\pi_{n}}$, $\Delta_{n}:=\boldsymbol{\Phi}_{W^{\pi_{n}}}=\boldsymbol{\Phi}_{W}$,
$f:=U$, and $\Delta:=\boldsymbol{\Phi}_{U}$. The lemma gives that
$\boldsymbol{\Phi}_{U}$ is at least as flat as $\boldsymbol{\Phi}_{W}$.

Observe that~(\ref{eq:peece}) implies that for the degree functions
$\deg_{U}:\Omega\rightarrow[0,1]$ and $\deg_{W^{\pi_{n}}}:\Omega\rightarrow[0,1]$
we have $\int_{A}\deg_{W^{\pi_{n}}}\WEAKCONV\int_{A}\deg_{U}$ for
every measurable $A\subseteq\Omega$. Now, we can apply Lemma~\ref{lem:flatmeasures}
with $B:=\Omega$, $f_{n}:=\deg_{W^{\pi_{n}}}$, $\Delta_{n}:=\boldsymbol{\Upsilon}_{W^{\pi_{n}}}=\boldsymbol{\Upsilon}_{W}$,
$f:=\deg_{U}$, and $\Delta:=\boldsymbol{\Upsilon}_{U}$. The lemma
gives that $\boldsymbol{\Upsilon}_{U}$ is at least as flat as $\boldsymbol{\Upsilon}_{W}$.
\end{proof}
\begin{proof}[Proof of Proposition~\ref{prop:flatter}, strictly flatter part]
Next, suppose that $U\prec W$. Then for the measure preserving bijections
$\pi_{1},\pi_{2},\pi_{3},\ldots$ as above, we have
\begin{equation}
W^{\pi_{n}}\NOTLONECONV U\;.\label{eq:NOTl1conv}
\end{equation}
Indeed, suppose that~(\ref{eq:NOTl1conv}) is not true, that is,
$W^{\pi_{n}}\LONECONV U$. Then for measure preserving bijections
$\psi_{n}:=\left(\pi_{n}\right)^{-1}$ we have $U^{\psi_{n}}\LONECONV W$,
and in particular $W\preceq U$. This is a contradiction to the fact
that $U\prec W$. Now, Lemma~\ref{lem:lonenonconvergence} implies
that there exists an interval $J\subseteq[0,1]$ and a <<strictly
bigger>> interval $J^{+}$ such that
\[
\nu^{\otimes2}\left(U^{-1}(J)\setminus\left(W^{\pi_{n}}\right)^{-1}(J^{+})\right)\not\longrightarrow0\;.
\]
By passing to a subsequence, let us assume that $\nu^{\otimes2}\left(U^{-1}(J)\setminus\left(W^{\pi_{n}}\right)^{-1}(J^{+})\right)>\varepsilon$
for each $n$ and for some $\varepsilon>0$. We shall now apply Lemma~\ref{lem:flatmeasures}
twice. To this end, we take $X:=U^{-1}(J)$. Furthermore, we write
the measures $\boldsymbol{\Phi}_{U}$ and $\boldsymbol{\Phi}_{W^{\pi_{n}}}$
as $\boldsymbol{\Phi}_{U}=\Phi_{U}^{X}+\Phi_{U}^{\Omega^{2}\setminus X}$
and $\boldsymbol{\Phi}_{W^{\pi_{n}}}=\Phi_{W^{\pi_{n}}}^{X}+\Phi_{W^{\pi_{n}}}^{\Omega^{2}\setminus X}$,
where $\Phi_{U}^{X}(L):=\nu^{\otimes2}\left(X\cap U^{-1}(L)\right)$,
$\Phi_{U}^{\Omega^{2}\setminus X}(L):=\nu^{\otimes2}\left(U^{-1}(L)\setminus X\right)$,
and the measures $\Phi_{W^{\pi_{n}}}^{X}$ and $\Phi_{W^{\pi_{n}}}^{\Omega^{2}\setminus X}$
are defined analogously. Let $\left(\Delta^{X},\Delta^{\Omega^{2}\setminus X}\right)$
be an arbitrary accumulation point of the sequence $\left(\Phi_{W^{\pi_{n}}}^{X},\Phi_{W^{\pi_{n}}}^{\Omega^{2}\setminus X}\right)_{n}$
of pairs of measures with respect to the product of weak{*} topologies
on measures on $[0,1]^{2}$. Crucially, note that $\boldsymbol{\Phi}_{W}=\Delta^{X}+\Delta^{\Omega^{2}\setminus X}$.
\begin{itemize}
\item We first apply Lemma~\ref{lem:flatmeasures} with $B:=X$, $f_{n}:=\left(W^{\pi_{n}}\right)_{\restriction X}$,
$\Delta_{n}:=\Phi_{W^{\pi_{n}}}^{X}$, $f:=U_{\restriction X}$, and
$\Delta:=\Phi_{U}^{X}$. The lemma gives that $\Phi_{U}^{X}$ is at
least as flat as $\Delta^{X}$. Recall that the support of the individual
measures $\Phi_{W^{\pi_{n}}}^{X}$ uniformly exceeds the interval
$J^{+}$. From this, we conclude that the support of their weak{*}
accumulation point $\Delta^{X}$ does not lie in $J$. On the other
hand, observe that the support of $\Phi_{U}^{X}$ lies inside $J$.
Thus $\Phi_{U}^{X}\neq\Delta^{X}$. We conclude that $\Phi_{U}^{X}$
is strictly flatter than $\Delta^{X}$.
\item Next, we apply Lemma~\ref{lem:flatmeasures} with $B:=\Omega^{2}\setminus X$,
$f_{n}:=\left(W^{\pi_{n}}\right)_{\restriction\Omega^{2}\setminus X}$,
$\Delta_{n}:=\Phi_{W^{\pi_{n}}}^{\Omega^{2}\setminus X}$, and $f:=U_{\restriction\Omega^{2}\setminus X}$,
$\Delta:=\Phi_{U}^{\Omega^{2}\setminus X}$. The lemma gives that
$\Phi_{U}^{\Omega^{2}\setminus X}$ is at least as flat as $\Delta^{\Omega^{2}\setminus X}$.
\end{itemize}
The proof now follows by Lemma~\ref{lem:strictplusnonstrict} to
$\boldsymbol{\Phi}_{W}=\Delta^{X}+\Delta^{\Omega^{2}\setminus X}$.
\end{proof}
Let us finish this section with an auxiliary result which will be
applied in Section~\ref{sec:BasicPropertiesOfStructurdness}. The
result states that probability measures supported on $\left\{ 0,1\right\} $
are maximal with respect to the flat order.
\begin{lem}
\label{lem:01measuresextremal}Suppose that $\Lambda_{1}$ and $\Lambda_{2}$
are two probability measures on $[0,1]$ such that $\Lambda_{1}$
is strictly flatter than $\Lambda_{2}$. Then $\Lambda_{1}$ is not
supported on $\left\{ 0,1\right\} $.
\end{lem}

\begin{proof}
The proof is obvious.
\end{proof}

\subsection{Relationship between envelopes, the cut distance, and range frequencies.}

Our last result states that two envelopes are equal if and only if
the corresponding graphons are weakly isomorphic. This result relies
on Theorem~\ref{thm:hyperspaceANDcutDISTANCEsimpler} which will
be proven in Section~\ref{sec:ProofOfEquivalenceSimpler}, and is
used later in the proof of Corollary~\ref{cor:homeomorphism}, which
puts into relation the cut distance and the Vietoris hyperspace $K(\GRAPHONSPACE)$.
\begin{cor}
\label{cor:summarizeENVcutdistMEASURES}Let $U,W\in\mathcal{W}_{0}$.
The following are equivalent:
\begin{itemize}
\item $\left\langle U\right\rangle =\left\langle W\right\rangle $,
\item $\delta_{\Box}\left(U,W\right)=0$.
\end{itemize}
Moreover, if $U\preceq W$ then the conditions above are equivalent
to 
\begin{itemize}
\item $\boldsymbol{\Phi}_{U}=\boldsymbol{\Phi}_{W}$.
\end{itemize}
\end{cor}

\begin{proof}
Assume that $\left\langle U\right\rangle =\left\langle W\right\rangle $.
Then the conditions in Theorem~\ref{thm:hyperspaceANDcutDISTANCEsimpler}\ref{enu:basicThmLIMACC}
are satisfied because we have $\left\langle U\right\rangle =\LIM(U,U,U\ldots)=\ACC(U,U,U,\ldots)$.
By the <<furthermore>> part of Theorem~\ref{thm:hyperspaceANDcutDISTANCEsimpler},
the constant sequence $U,U,U,\ldots$ converges in $\delta_{\Box}$
to a maximal element of $\left\langle U\right\rangle $. By the assumption
$\left\langle U\right\rangle =\left\langle W\right\rangle $, the
graphon $W$ is a maximal element of $\left\langle U\right\rangle $.
The only possibility under which a constant sequence can converge
in the cut distance to another graphon is when their cut distance
is $0$, that is $\delta_{\Box}\left(U,W\right)=0$.

The opposite implication is Lemma~\ref{lem:basicpropofenvelopes}\ref{enu:envelopeswhendistzero}.

We now turn to the <<moreover>> part. Suppose that $\left\langle U\right\rangle =\left\langle W\right\rangle $.
By Proposition~\ref{prop:flatter} we have that $\boldsymbol{\Phi}_{U}$
is at least as flat as $\boldsymbol{\Phi}_{W}$ and vice-versa. By
Lemma~\ref{lem:stone-weierstrass} this means that $\int f\:\mathrm{d}\boldsymbol{\Phi}_{W}=\int f\:\mathrm{d}\boldsymbol{\Phi}_{U}$
for every continuous convex function $f\colon[0,1]\rightarrow\mathbb{R}.$
Applying the Stone\textendash Weierstrass theorem in the same way
as in the proof of Lemma~\ref{lem:stone-weierstrass}, this is true
for every continuous (not necessarily convex) function $f\colon[0,1]\rightarrow\mathbb{R}$.
Therefore, $\boldsymbol{\Phi}_{U}=\boldsymbol{\Phi}_{W}$.

Assume finally that $U\preceq W$ and $\boldsymbol{\Phi}_{U}=\boldsymbol{\Phi}_{W}$.
Then $\boldsymbol{\Phi}_{U}$ is not strictly flatter than $\boldsymbol{\Phi}_{W}$
and therefore by Proposition~\ref{prop:flatter}, it is not the case
that $U\prec W$. This means that $\left\langle U\right\rangle =\left\langle W\right\rangle $. 
\end{proof}

\section{Proof of Theorem~\ref{thm:hyperspaceANDcutDISTANCEsimpler}\label{sec:ProofOfEquivalenceSimpler}}

The main idea of the proof of the implication \ref{enu:basicThmLIMACC}$\Rightarrow$\ref{enu:basicMainCauchy}
is the following. Let $W$ be a $\preceq$-maximal element in $\LIM\left(\Gamma_{1},\Gamma_{2},\Gamma_{3},\ldots\right)$.
Such a $W$ is guaranteed to exist by Lemma~\ref{lem:LIMACCcontainsMAX}.
Then the sequence $\Gamma_{1},\Gamma_{2},\Gamma_{3},\ldots$ converges
to $W$ in the cut distance. To make this argument precise we first
recall several definitions and results from~\cite{DH:WeakStar}.
\begin{lem}[Claim 1 and Claim 2 in \cite{DH:WeakStar}]
\label{lem:cohen}Let $W,\Gamma_{1},\Gamma_{2},...\in\GRAPHONSPACE$
be graphons defined on $\Omega=\left[0,1\right]$ and assume that
$\Gamma_{n}\WEAKCONV W$. Take some sequence $B_{1},B_{2},...\subseteq[0,1]$
of measurable sets and a subsequence $\left(n_{k}\right)_{k}$ such
that $1_{B_{n_{k}}}\WEAKCONV s$ and $_{B_{n_{k}}}\Gamma_{n_{k}}\WEAKCONV\tilde{W}$
(with the notation from~(\ref{eq:ordered2})) for some function $s:\left[0,1\right]\rightarrow\left[0,1\right]$
and some graphon $\tilde{W}$. We define
\begin{equation}
\psi(x)=\int_{0}^{x}s(y)\:\mathrm{d}y\ \text{\ \ \ \ {and}}\ \ \ \ \ \varphi(x)=\psi(1)+\int_{0}^{x}(1-s(y))\:\mathrm{d}y.\label{eq:def_of_psi_phi}
\end{equation}
Then for almost every $(x,y)\in[0,1]^{2}$ we have 
\begin{align}
W(x,y)= & \tilde{W}(\psi(x),\psi(y))s(x)s(y)+\tilde{W}(\psi(x),\varphi(y))s(x)(1-s(y))\nonumber \\
 & +\tilde{W}(\varphi(x),\psi(y))(1-s(x))s(y)+\tilde{W}(\varphi(x),\varphi(y))(1-s(x))(1-s(y)).\label{eq:WvsTildeW}
\end{align}

Moreover, if $W$ is not a cut-norm accumulation point of the sequence
$\Gamma_{1},\Gamma_{2},\Gamma_{3},\ldots$ and the sets $B_{1},B_{2},...\subseteq[0,1]$
are chosen to witness this fact, i.e., such that $\int_{B_{n}\times B_{n}}(\Gamma_{n}-W)>\varepsilon$
or $\int_{B_{n}\times B_{n}}(\Gamma_{n}-W)<-\varepsilon$ for some
$\varepsilon>0$ (which does not depend on $n$), the convex combination~(\ref{eq:WvsTildeW})
is proper\footnote{that is, at least two summands on the right-hand side are positive}
on a set of positive $\nu^{\otimes2}$ measure.
\end{lem}

\begin{proof}[Proof of Theorem~\ref{thm:hyperspaceANDcutDISTANCEsimpler}, \ref{enu:basicThmLIMACC}$\Rightarrow$\ref{enu:basicMainCauchy}]
 Suppose that $\LIM\left(\Gamma_{1},\Gamma_{2},\Gamma_{3},\ldots\right)=\ACC\left(\Gamma_{1},\Gamma_{2},\Gamma_{3},\ldots\right)$
and assume that $W\in\LIM\left(\Gamma_{1},\Gamma_{2},\Gamma_{3},\ldots\right)$
is a maximal element as given in Lemma~\ref{lem:LIMACCcontainsMAX}.
We may also assume that $\Gamma_{n}\WEAKCONV W$. We claim that this
already implies that $\Gamma_{n}\CUTNORMCONV W$. Suppose not. Then
by passing to a subsequence we may assume that there is an $\varepsilon>0$
and a sequence $B_{1},B_{2},...$ of Borel subsets of $[0,1]$ such
that $\int_{B_{n}\times B_{n}}(\Gamma_{n}-W)>\varepsilon$ (or $\int_{B_{n}\times B_{n}}(\Gamma_{n}-W)<-\varepsilon$
which can be handled similarly). We will use versions $_{B_{n}}\Gamma_{n}$
of the graphons $\Gamma_{n}$ that are defined via~(\ref{eq:ordered2}).
We take a function $s:[0,1]\rightarrow[0,1]$ and a graphon $\tilde{W}$
such that $1_{B_{n}}\WEAKCONV s$ and $_{B_{n}}\Gamma_{n}\WEAKCONV\tilde{W}$.
If $1_{B_{n}}$ or $_{B_{n}}\Gamma_{n}$ are not weak{*} convergent
(that is, $s$ or $\tilde{W}$ do not exist), then we pass to a suitable
convergent subsequence (which we still index by $1,2,\ldots$). 
\begin{claim*}
We have $\langle W\rangle=\langle\tilde{W}\rangle$.
\end{claim*}
\begin{proof}
For each $n\in\mathbb{N}$, let $\mathcal{I}_{n}=\{I_{n,0},I_{n,1},...,I_{n,2^{n}-1}\}$
be a partition of $[0,1]$ into $2^{n}$ many equimeasurable intervals,
i.e., $I_{n,k}=[\frac{k}{2^{n}},\frac{k+1}{2^{n}})$. Define a measure
preserving almost-bijection $\varphi_{n}$ by
\begin{align*}
\varphi_{n}(x) & =\begin{cases}
\frac{k}{2^{n}}+x-\psi\left(\frac{k}{2^{n}}\right) & \text{if \ensuremath{x\in\left[\psi\left(\frac{k}{2^{n}}\right),\psi\left(\frac{k+1}{2^{n}}\right)\right)} and}\\
\frac{k}{2^{n}}+\psi\left(\frac{k+1}{2^{n}}\right)-\psi\left(\frac{k}{2^{n}}\right)+x-\varphi\left(\frac{k}{2^{n}}\right) & \text{if \ensuremath{x\in\left[\varphi\left(\frac{k}{2^{n}}\right),\varphi\left(\frac{k+1}{2^{n}}\right)\right)},}
\end{cases}
\end{align*}
where $\psi$ and $\varphi$ are defined by~(\ref{eq:def_of_psi_phi}).
Define $\tilde{W}_{n}(x,y)=\tilde{W}\left(\varphi_{n}^{-1}(x),\varphi_{n}^{-1}(y)\right)$.
We claim that $\tilde{W}_{n}\WEAKCONV W$. First note that if $m\in\mathbb{N}$
and $m\ge n$ and $k,l<2^{n}$, then 
\begin{equation}
\int_{I_{n,k}\times I_{n,l}}\tilde{W}_{m}=\sum_{\substack{i=2^{m-n}\cdot k}
}^{2^{m-n}\cdot\left(k+1\right)-1}\;\sum_{\substack{j=2^{m-n}\cdot l}
}^{2^{m-n}\cdot\left(l+1\right)-1}\;\int_{I_{m,i}\times I_{m,j}}\tilde{W}_{m}.\label{eq:nointernet}
\end{equation}
For each $n\in\mathbb{N}$ we have 
\begin{align*}
\int_{I_{n,k}\times I_{n,l}}\tilde{W}_{n} & =\int_{\psi(\frac{k}{2^{n}})}^{\psi(\frac{k+1}{2^{n}})}\int_{\psi(\frac{l}{2^{n}})}^{\psi(\frac{l+1}{2^{n}})}\tilde{W}+\int_{\psi(\frac{k}{2^{n}})}^{\psi(\frac{k+1}{2^{n}})}\int_{\varphi(\frac{l}{2^{n}})}^{\varphi(\frac{l+1}{2^{n}})}\tilde{W}\\
 & \,+\int_{\varphi(\frac{k}{2^{n}})}^{\varphi(\frac{k+1}{2^{n}})}\int_{\psi(\frac{l}{2^{n}})}^{\psi(\frac{l+1}{2^{n}})}\tilde{W}+\int_{\varphi(\frac{k}{2^{n}})}^{\varphi(\frac{k+1}{2^{n}})}\int_{\varphi(\frac{l}{2^{n}})}^{\varphi(\frac{l+1}{2^{n}})}\tilde{W}\\
 & =\int_{I_{n,k}\times I_{n,l}}\tilde{W}(\psi(x),\psi(y))s(x)s(y)+\int_{I_{n,k}\times I_{n,l}}\tilde{W}(\psi(x),\varphi(y))s(x)(1-s(y))\\
 & \,+\int_{I_{n,k}\times I_{n,l}}\tilde{W}(\varphi(x),\psi(y))(1-s(x))s(y)+\int_{I_{n,k}\times I_{n,l}}\tilde{W}(\varphi(x),\varphi(y))(1-s(x))(1-s(y))\\
 & =\int_{I_{n,k}\times I_{n,l}}W.
\end{align*}
Plugging this back to~(\ref{eq:nointernet}) gives
\[
\int_{I_{n,k}\times I_{n,l}}\tilde{W}_{m}=\int_{I_{n,k}\times I_{n,l}}W\;.
\]
This implies that for every $n\in\mathbb{N}$ and $k,l<2^{n}$, we
have $\int_{I_{n,k}\times I_{n,l}}\tilde{W}_{m}\to\int_{I_{n,k}\times I_{n,l}}W$
as $m\rightarrow\infty$. Consequently $\tilde{W}_{m}\WEAKCONV W$
and $W\in\langle\tilde{W}\rangle$. On the other hand, $\tilde{W}\in\langle W\rangle$
since $\tilde{W}\in\LIM(\Gamma_{1},\Gamma_{2},\Gamma_{3},\ldots)$,
and the maximality of $W$ in $\LIM(\Gamma_{1},\Gamma_{2},\Gamma_{3},\ldots)$.
\end{proof}
Now Proposition~\ref{prop:flatter} together with Lemma~\ref{lem:stone-weierstrass}
give us that $\int f\;\mathrm{d}\boldsymbol{\Phi}_{\tilde{W}}=\int f\;\mathrm{d}\boldsymbol{\Phi}_{W}$
for every continuous convex function. So to get the desired contradiction
it suffices to prove the following claim.
\begin{claim*}
Let $f:[0,1]\to[0,1]$ be a strictly convex function. Then $\int f\;\mathrm{d}\boldsymbol{\Phi}_{\tilde{W}}>\int f\;\mathrm{d}\boldsymbol{\Phi}_{W}$.
\end{claim*}
\begin{proof}
Recall that the convex combination in~(\ref{eq:WvsTildeW}) is proper
on a set of positive $\nu^{\otimes2}$ measure. Therefore the strict
convexity of $f$ gives us
\begin{eqnarray*}
\int_{0}^{1}f(x)\;\mathrm{d}\boldsymbol{\Phi}_{W} & = & \int_{0}^{1}\int_{0}^{1}f(W(x,y))\\
 & < & \int_{0}^{1}\int_{0}^{1}f\left(\tilde{W}(\psi(x),\psi(y))s(x)s(y)\right)+\int_{0}^{1}\int_{0}^{1}f\left(\tilde{W}(\psi(x),\varphi(y))s(x)(1-s(y))\right)\\
 &  & +\int_{0}^{1}\int_{0}^{1}f\left(\tilde{W}(\varphi(x),\psi(y))(1-s(x))s(y)\right)\\
 &  & +\int_{0}^{1}\int_{0}^{1}f\left(\tilde{W}(\varphi(x),\varphi(y))(1-s(x))(1-s(y))\right)\\
 & = & \int_{0}^{\psi(1)}\int_{0}^{\psi(1)}f\left(\tilde{W}(x,y)\right)+\int_{0}^{\psi(1)}\int_{\psi(1)}^{1}f\left(\tilde{W}(x,y)\right)\\
 &  & +\int_{\psi(1)}^{1}\int_{0}^{\psi(1)}f\left(\tilde{W}(x,y)\right)+\int_{\psi(1)}^{1}\int_{\psi(1)}^{1}f\left(\tilde{W}(x,y)\right)\\
 & = & \int_{0}^{1}\int_{0}^{1}f(\tilde{W}(x,y))=\int_{0}^{1}f(x)\;\mathrm{d}\boldsymbol{\Phi}_{\tilde{W}}\;.
\end{eqnarray*}
\end{proof}
Let $\Gamma_{1},\Gamma_{2},\Gamma_{3},\ldots$ be a sequence of graphons
which is Cauchy with respect to the cut distance. By Theorem~\ref{thm:subsequenceLIMACC}
there is a subsequence $\Gamma_{n_{1}},\Gamma_{n_{2}},\Gamma_{n_{3}},\ldots$
such that $\LIM(\Gamma_{n_{1}},\Gamma_{n_{2}},\Gamma_{n_{3}},\ldots)=\ACC(\Gamma_{n_{1}},\Gamma_{n_{2}},\Gamma_{n_{3}},\ldots)$.
By the proof of implication \ref{enu:basicThmLIMACC}$\Rightarrow$\ref{enu:basicMainCauchy}
this subsequence converges to some $W$ in the cut distance. As the
original sequence $\Gamma_{1},\Gamma_{2},\Gamma_{3},\ldots$ is Cauchy
with respect to the cut distance it follows that it converges to $W$
as well. We may suppose that $\Gamma_{n}\CUTNORMCONV W$. To conclude
that $\LIM(\Gamma_{1},\Gamma_{2},\Gamma_{3},\ldots)=\ACC(\Gamma_{1},\Gamma_{2},\Gamma_{3},\ldots)$
consider any $\Gamma_{n_{k}}^{\varphi_{k}}\WEAKCONV U$. We claim
that then also $W^{\varphi_{k}}\WEAKCONV U$. To see this fix some
Borel set $A\subseteq[0,1]$ then
\[
\left|\int_{A\times A}(W^{\varphi_{k}}-U)\right|\le\left|\int_{\varphi_{k}(A)\times\varphi_{k}(A)}(W-\Gamma_{n_{k}})\right|+\left|\int_{A\times A}(\Gamma_{n_{k}}^{\varphi_{k}}-U)\right|\to0
\]
where the first term tends to $0$ by $\left\Vert \cdot\right\Vert _{\square}$-convergence
and the second by the weak{*} convergence. Then we use the same trick
again:
\[
\left|\int_{A\times A}(\Gamma_{n}^{\varphi_{n}}-U)\right|\le\left|\int_{\varphi_{n}(A)\times\varphi_{n}(A)}(\Gamma_{n}-W)\right|+\left|\int_{A\times A}(W^{\varphi_{n}}-U)\right|\to0\;.
\]
This gives us $\Gamma_{n}^{\varphi_{n}}\WEAKCONV U$ which means that
$U\in\LIM\left(\Gamma_{1},\Gamma_{2},\Gamma_{3},\ldots\right)$.
\end{proof}

\section{Relating the hyperspace $K(\protect\GRAPHONSPACE)$ and the cut distance\label{subsec:RelatingToVietoris}}

Lemma~\emph{\ref{lem:LIMclosed}\ref{enu:weakstarcompact}} says
that envelopes are members of the hyperspace $K(\GRAPHONSPACE)$.
First, we provide a proof of an extension of Theorem~\ref{thm:hyperspaceANDcutDISTANCEsimpler},
stated in Theorem~\ref{thm:hyperspaceANDcutDISTANCE} below. In addition
to the original statement, we include a characterization of cut distance
convergence in terms of the hyperspace $K(\GRAPHONSPACE)$, and also
describe the limit graphon.
\begin{thm}
\label{thm:hyperspaceANDcutDISTANCE}Let $W,\Gamma_{1},\Gamma_{2},\Gamma_{3},\ldots\in\GRAPHONSPACE$.
The following are equivalent:
\begin{enumerate}[label=(\alph*)]
\item \label{enu:thmextGammaW}$\Gamma_{n}\CUTDISTCONV W$,
\item \label{enu:thmextVietoris}$\left\langle \Gamma_{n}\right\rangle \to\left\langle W\right\rangle $
in $K(\GRAPHONSPACE)$,
\item \label{enu:thmextLIMACC}$\LIM\left(\Gamma_{1},\Gamma_{2},\Gamma_{3},\ldots\right)=\ACC\left(\Gamma_{1},\Gamma_{2},\Gamma_{3},\ldots\right)$
and $W\in\LIM\left(\Gamma_{1},\Gamma_{2},\Gamma_{3},\ldots\right)$
is the $\preceq$-maximal element of $\LIM\left(\Gamma_{1},\Gamma_{2},\Gamma_{3},\ldots\right)$.
\end{enumerate}
\end{thm}

\begin{proof}
Condition~\ref{enu:thmextLIMACC} is equivalent to Condition~\ref{enu:thmextGammaW}
by Theorem~\ref{thm:hyperspaceANDcutDISTANCEsimpler}. Before we
prove that~\ref{enu:thmextVietoris} is equivalent to~\ref{enu:thmextLIMACC}
we recall that a basic open neighborhood of $K\in K\left(\GRAPHONSPACE\right)$
is given by some finite open cover $\mathcal{U}$ of $K$, namely
for every such cover $\mathcal{U}$ we define the open neighborhood
of $K$ as $\mathcal{O_{\mathcal{U}}}=\left\{ L:L\subseteq\bigcup_{O\in\mathcal{U}}O\ \text{and}\ \forall O\in\mathcal{U}\ L\cap O\not=\emptyset\right\} $.

\ref{enu:thmextLIMACC} $\implies$\ref{enu:thmextVietoris}. First
note that $\langle W\rangle=\LIM\left(\Gamma_{1},\Gamma_{2},\Gamma_{3},\ldots\right)$.
Suppose that $\left\langle \Gamma_{n}\right\rangle \not\to\left\langle W\right\rangle $
in $K\left(\GRAPHONSPACE\right)$. There must be a basic open neighborhood
$\mathcal{\mathcal{O_{\mathcal{U}}}}$ of $\langle W\rangle$ and
a subsequence $\langle\Gamma_{k_{1}}\rangle,\langle\Gamma_{k_{2}}\rangle,...$
such that either $\langle\Gamma_{k_{j}}\rangle\not\subseteq\bigcup_{O\in\mathcal{U}}O$
for every $j\in\mathbb{N}$ or there is $O\in\mathcal{U}$ such that
$\langle\Gamma_{k_{j}}\rangle\cap O=\emptyset$ for every $j\in\mathbb{N}$.
If the first possibility happens, then clearly $\ACC\left(\Gamma_{k_{1}},\Gamma_{k_{2}},\ldots\right)\cap(\GRAPHONSPACE\setminus\bigcup_{O\in\mathcal{U}}O)\not=\emptyset$,
which is a contradiction with $\ACC\left(\Gamma_{k_{1}},\Gamma_{k_{2}},\ldots\right)=\LIM\left(\Gamma_{1},\Gamma_{2},\Gamma_{3},\ldots\right)\subseteq\bigcup_{O\in\mathcal{U}}O$.
If the second possibility occurs, for some $O\in\mathcal{U}$ then
clearly $\ACC\left(\Gamma_{k_{1}},\Gamma_{k_{2}},\ldots\right)\cap O=\emptyset$.
That is also a contradiction once $\ACC\left(\Gamma_{k_{1}},\Gamma_{k_{2}},\ldots\right)=\LIM\left(\Gamma_{1},\Gamma_{2},\Gamma_{3},\ldots\right)=\langle W\rangle$
and $\langle W\rangle\cap O\not=\emptyset$ because $\langle W\rangle\in\mathcal{\mathcal{O_{\mathcal{U}}}}$.

\ref{enu:thmextVietoris} $\implies$ \ref{enu:thmextLIMACC}. Let
$U\in\ACC\left(\Gamma_{1},\Gamma_{2},\Gamma_{3},\ldots\right)$. We
claim that $U\in\langle W\rangle$. Assume it is not true and take
any open set $O\subseteq\GRAPHONSPACE$ in the weak{*} topology such
that $\left\langle W\right\rangle \subseteq O$ and $U$ is not in
the closure of $O$. Then by the convergence in $K\left(\GRAPHONSPACE\right)$
we may find some $m\in\mathbb{N}$ such that for every $k\ge m$ we
have $\left\langle \Gamma_{k}\right\rangle \subseteq O$. This implies
that $\ACC\left(\Gamma_{1},\Gamma_{2},\Gamma_{3},\ldots\right)$ is
a subset of the closure of $O$ and that is a contradiction because
we assumed that $U$ is not in that closure. Therefore, we have just
proved that $\ACC(\Gamma_{1},\Gamma_{2},\Gamma_{3},\ldots)\subseteq\langle W\rangle$.

Let $U\in\left\langle W\right\rangle $, then for every open set $O\subseteq\GRAPHONSPACE$
in the weak{*} topology such that $U\in O$ we can find $m\in\mathbb{N}$
such that for every $k\ge m$ we have $\left\langle \Gamma_{k}\right\rangle \cap O\not=\emptyset$.
Since $\GRAPHONSPACE$ with the weak{*} topology is a metric space
we may find a sequence of versions such that $\Gamma_{n}^{\varphi_{n}}\WEAKCONV U$.
Therefore, we have $\langle W\rangle\subseteq\LIM(\Gamma_{1},\Gamma_{2},\Gamma_{3},\ldots)$,
and so $\LIM\left(\Gamma_{1},\Gamma_{2},\Gamma_{3},\ldots\right)=\ACC\left(\Gamma_{1},\Gamma_{2},\Gamma_{3},\ldots\right)$.
By the first part of the argument it follows that $\langle W\rangle=\LIM\left(\Gamma_{1},\Gamma_{2},\Gamma_{3},\ldots\right)$
and thus $W$ is the $\preceq$-maximal element of $\LIM\left(\Gamma_{1},\Gamma_{2},\Gamma_{3},\ldots\right)$. 
\end{proof}
We can now formulate Corollary~\ref{cor:homeomorphism}, which is
the final statement of this section. It allows us to transfer the
space $\left(\UNLABELLEDGRAPHONSPACE,\delta_{\Box}\right)$ into the
hyperspace $K(\GRAPHONSPACE)$. This transference will be useful later. 

Prior to giving the statement, observe that for the envelope map $\left\langle \cdot\right\rangle $
defined on $\GRAPHONSPACE$ we have $\left\langle W_{1}\right\rangle =\left\langle W_{2}\right\rangle $
for weakly isomorphic graphons $W_{1}$ and $W_{2}$. That allows
us to define $\left\langle \cdot\right\rangle $ even on the factor-space
$\UNLABELLEDGRAPHONSPACE$.
\begin{cor}
\label{cor:homeomorphism}The envelope map $\left\langle \cdot\right\rangle :\left(\UNLABELLEDGRAPHONSPACE,\delta_{\Box}\right)\to K(\GRAPHONSPACE)$
is a continuous injection. We have that $\left(\UNLABELLEDGRAPHONSPACE,\delta_{\Box}\right)$
is homeomorphic to some closed subspace $X$ of $K(\GRAPHONSPACE)$.
Moreover, the metric $\delta_{\Box}$ is equivalent\footnote{Recall that two metrics on a topological space are \emph{equivalent}
if they give the same topology.} to the pullback $\chi$ of the hyperspace metric on $K(\GRAPHONSPACE)$
defined in~(\ref{eq:HausdorffMetric}), that is,
\[
\chi\left(\left\llbracket U\right\rrbracket ,\left\llbracket W\right\rrbracket \right)=\max\left\{ \sup_{\varphi}\left\{ \inf_{\psi}\{d_{\mathrm{w}^{*}}(U^{\varphi},W^{\psi})\}\right\} ,\sup_{\varphi}\left\{ \inf_{\psi}\{d_{\mathrm{w}^{*}}(U^{\psi},W^{\varphi})\}\right\} \right\} ,\ \ \ \ \ \left\llbracket U\right\rrbracket ,\left\llbracket W\right\rrbracket \in\UNLABELLEDGRAPHONSPACE\;,
\]
where $\varphi$ and $\psi$ range through all measure preserving
bijections on $\Omega$.\footnote{Note that the definition of $\chi\left(\left\llbracket U\right\rrbracket ,\left\llbracket W\right\rrbracket \right)$
does not depend on the particular representatives $U$ and $W$.

Also note that this statement holds for any metric $d_{\mathrm{w}^{*}}$
compatible with the weak{*} topology, not only the one given in~(\ref{eq:exampleweakstarmetric}).}

Finally, $\left(\UNLABELLEDGRAPHONSPACE,\delta_{\Box}\right)$ is
compact.
\end{cor}

\begin{proof}
The map $\left\langle \cdot\right\rangle $ is well-defined and injective
by Corollary~\ref{cor:summarizeENVcutdistMEASURES} and it is continuous
by Theorem~\ref{thm:hyperspaceANDcutDISTANCE}. The set $X=\left\langle \UNLABELLEDGRAPHONSPACE\right\rangle $
is closed in the Vietoris topology. Indeed, suppose that $\left\langle \Gamma_{n}\right\rangle \to L$
in $K\left(\GRAPHONSPACE\right)$ for some $L\in K\left(\GRAPHONSPACE\right)$.
Then it follows from \cite[Exercise 4.21, Exercise 4.23]{MR1321597}
that $L=\LIM\left(\Gamma_{1},\Gamma_{2},\Gamma_{3},...\right)=\ACC\left(\Gamma_{1},\Gamma_{2},\Gamma_{3},...\right)$,
see the definition of topological limit before \cite[Exercise 4.23]{MR1321597}.
By Theorem~\ref{thm:hyperspaceANDcutDISTANCEsimpler} we know that
there is some $W\in\GRAPHONSPACE$ such that $\Gamma_{n}\CUTDISTCONV W$.
By the continuity of $\langle.\rangle$ we also have $\left\langle \Gamma_{n}\right\rangle \to\left\langle W\right\rangle $
in $K\left(\GRAPHONSPACE\right)$. 

$K(\GRAPHONSPACE)$ is compact by Fact~\ref{fact:VietorisCompact}
and Remark~\ref{rem:interestedhyperspace}. Since $X$ is a closed
subset of $K(\GRAPHONSPACE)$, it is also compact. By Theorem~\ref{thm:hyperspaceANDcutDISTANCE},
we have that the inverse map $\langle\cdot\rangle^{-1}$ is also continuous.
Therefore $\UNLABELLEDGRAPHONSPACE$ is homeomorphic to $X$, and
hence compact. Therefore, $\delta_{\Box}$ and $\chi$ give the same
compact topology on $\UNLABELLEDGRAPHONSPACE$.
\end{proof}

\section{Relation to multiway cuts\label{sec:Multiway}}

In this section, we compare our view with~\cite{MR2925382}. The
main result of~\cite{MR2925382} is a statement that several properties
of graph(on) sequences are equivalent to cut distance convergence.
These properties include so-called right-convergence and convergence
of free energies. As we shall see, however, the one that relates to
our approach involves so-called multiway cuts.

Let us give the key definition (see p.~178 in~\cite{MR2925382}).
First, for two real matrices $M$ and $N$ of the same dimensions,
their \emph{$L_{1}$-distance} is defined as $\left\Vert M-N\right\Vert _{1}:=\sum_{i,j}\left|M_{ij}-N_{ij}\right|$.
Suppose that $\mathbf{a}=\left(a_{1},\ldots,a_{q}\right)$ is a vector
with non-negative entries that sum up to~1. Given a graphon $W:\Omega^{2}\rightarrow[0,1]$,
we write $\mathcal{S}_{\mathbf{a}}(W)$ for the set of all $q\times q$
matrices $M$ that can be obtained in the following way. Partition
suitably $\Omega=\Omega_{1}\sqcup\ldots\sqcup\Omega_{q}$ so that
$\nu(\Omega_{i})=a_{i}$. Then for each $i,j\in[q]$, define $M_{ij}:=\int_{\Omega_{i}}\int_{\Omega_{j}}W$.
So, $\mathcal{S}_{\mathbf{a}}(W)\subseteq[0,1]^{q\times q}$. In particular,
given two graphons $U$ and $W$, we can talk about the \emph{Hausdorff
distance with respect to $\ell_{1}$} of the sets $\mathcal{S}_{\mathbf{a}}(U)$
and $\mathcal{S}_{\mathbf{a}}(W)$,
\[
d_{1}^{\mathrm{Hf}}\left(\mathcal{S}_{\mathbf{a}}(U),\mathcal{S}_{\mathbf{a}}(W)\right):=\max\left\{ \sup_{M\in\mathcal{S}_{\mathbf{a}}(U)}\inf_{N\in\mathcal{S}_{\mathbf{a}}(W)}\left\Vert M-N\right\Vert _{1},\sup_{M\in\mathcal{S}_{\mathbf{a}}(W)}\inf_{N\in\mathcal{S}_{\mathbf{a}}(U)}\left\Vert M-N\right\Vert _{1}\right\} \;.
\]
Then one of the main equivalences of~\cite{MR2925382}, given there
in Theorem~3.5 and stated here in a slightly tailored form\footnote{More precisely, Theorem 3.5 of~\cite{MR2925382} does not talk about
$\mathcal{S}_{\mathbf{a}}(\cdot)$, but about a slightly different
object $\mathcal{S}_{q}(\cdot).$ However, equivalence of these two
approaches easily follows from Lemma 4.5 of~\cite{MR2925382}.}, is the following.
\begin{thm}[\cite{MR2925382}]
\label{thm:Multiway}Let $\Gamma_{1},\Gamma_{2},\Gamma_{3},\ldots\in\GRAPHONSPACE$.
The following are equivalent:
\begin{enumerate}[label=(\alph*)]
\item The sequence $\Gamma_{1},\Gamma_{2},\Gamma_{3},\ldots$ is Cauchy
with respect to the cut distance $\delta_{\square}$,
\item For each $\mathbf{a}$ as above, the sequence $\mathcal{S}_{\mathbf{a}}(\Gamma_{1}),\mathcal{S}_{\mathbf{a}}(\Gamma_{2}),\mathcal{S}_{\mathbf{a}}(\Gamma_{3}),\ldots$
is Cauchy with respect to $d_{1}^{\mathrm{Hf}}$.
\end{enumerate}
\end{thm}

To close the circle, we prove that the second conditions of Theorem~\ref{thm:hyperspaceANDcutDISTANCEsimpler}
and of Theorem~\ref{thm:Multiway} are equivalent. We shall use the
Hausdorff metric $d_{\mathrm{w^{*}}}^{\mathrm{Hf}}$ on $K(\GRAPHONSPACE)$
derived as in~(\ref{eq:HausdorffMetric}) from the weak{*} metric
defined, say, by~(\ref{eq:exampleweakstarmetric}) .
\begin{prop}
\label{prop:MultiwayEquivalentLIMACC}Let $\Gamma_{1},\Gamma_{2},\Gamma_{3},\ldots\in\GRAPHONSPACE$.
The following are equivalent:
\begin{enumerate}[label=(\alph*)]
\item \label{enu:Tired1}$\LIM\left(\Gamma_{1},\Gamma_{2},\Gamma_{3},\ldots\right)=\ACC\left(\Gamma_{1},\Gamma_{2},\Gamma_{3},\ldots\right)$,
\item \label{enu:Tired2}The sequence $\left\langle \Gamma_{1}\right\rangle ,\left\langle \Gamma_{2}\right\rangle ,\left\langle \Gamma_{3}\right\rangle ,\ldots$
is Cauchy with respect to $d_{\mathrm{w^{*}}}^{\mathrm{Hf}}$,
\item \label{enu:Tired3}For each $\mathbf{a}$ as above, the sequence $\mathcal{S}_{\mathbf{a}}(\Gamma_{1}),\mathcal{S}_{\mathbf{a}}(\Gamma_{2}),\mathcal{S}_{\mathbf{a}}(\Gamma_{3}),\ldots$
is Cauchy with respect to $d_{1}^{\mathrm{Hf}}$.
\end{enumerate}
\end{prop}

\begin{proof}
First we show that~\ref{enu:Tired3} implies~\ref{enu:Tired1}.
Suppose that $\Omega=\left[0,1\right]$ and let $\mathcal{I}_{k}=\left\{ I_{k,1},\dots,I_{k,k}\right\} $
be the canonical partition of $\left[0,1\right]$ into equimeasurable
intervals, i.e., $\lambda\left(I_{k,i}\right)=\frac{1}{k}$ and $I_{k,i}$
precedes $I_{k,i+1}$. If $\mathcal{P}=\left\{ P_{1},\ldots,P_{k}\right\} $
is another equitable partition of $\left[0,1\right]$, then we denote
as $\pi_{\mathcal{P}}$ any fixed measure preserving bijection such
that $\pi_{\mathcal{P}}\left(P_{i}\right)=I_{k,i}$ for every $i\in[k]$.
Let $W\in\ACC\left(\Gamma_{1},\Gamma_{2},\Gamma_{3},\ldots\right)$.
We may suppose that there is a subsequence $\left\{ n_{j}\right\} _{j\in\mathbb{N}}$
such that $\Gamma_{n_{j}}\WEAKCONV W$. Fix $k\in\mathbb{N}$. It
is easy to see that $\left(\Gamma_{n_{j}}^{\Join\mathcal{I}_{k}}\right)\WEAKCONV W^{\Join\mathcal{I}_{k}}$.
Define $M_{n_{j}}\left(r,s\right)=\int_{I_{k,r}\times I_{k,s}}\Gamma_{n_{j}}$
and $N_{k}\left(r,s\right)=\int_{I_{k,r}\times I_{k,s}}W$ for every
$r,s\in[k]$. Then we have $M_{n_{j}}\in\mathcal{S}_{{\bf a}}\left(\Gamma_{n_{j}}\right)$
for every $j\in\mathbb{N}$ and $N_{k}\in\mathcal{S}_{{\bf a}}\left(W\right)$
where ${\bf a}=\left(\frac{1}{k},\dots,\frac{1}{k}\right)$. Moreover,
$M_{n_{j}}\LONECONV N_{k}$. The assumption~\ref{enu:Tired3} allows
to find an equitable partition $\mathcal{P}_{n}=\left\{ P_{i}\right\} _{\in[k]}$
for every $n\not\in\left\{ n_{j}:j\in\mathbb{N}\right\} $ such that
if we define $M_{n}\left(r,s\right)=\int_{P_{r}\times P_{s}}\Gamma_{n}$
for every $r,s\in[k]$, then we have $M_{n}\LONECONV N_{k}$. One
can easily verify that for $W_{n}:=\left(\left(\Gamma_{n}\right)^{\Join\mathcal{P}_{n}}\right)^{\pi_{\mathcal{P}_{n}}}$
we have $W_{n}=\left(\left(\Gamma_{n}\right)^{\pi_{\mathcal{P}_{n}}}\right)^{\Join\mathcal{I}_{k}}$
and therefore $W_{n}\in\left\langle \Gamma_{n}\right\rangle $ by
Lemma\ref{lem:averagingInsideLIM}. This implies that $W_{n}\WEAKCONV W^{\Join\mathcal{I}_{k}}$
because $W_{n}\upharpoonright I_{k,r}\times I_{k,s}=k^{-2}\cdot M_{n}\left(r,s\right)$
holds for every $n\in\mathbb{N}$ and $r,s\in[k]$. Hence, $W^{\Join\mathcal{I}_{k}}\in\LIM\left(\Gamma_{1},\Gamma_{2},\Gamma_{3},\ldots\right)$.
Finally, since $W^{\Join\mathcal{I}_{k}}\WEAKCONV W$ and $\LIM\left(\Gamma_{1},\Gamma_{2},\Gamma_{3},\ldots\right)$
is weak{*} closed by Lemma\ref{lem:LIMclosed} \ref{enu:weakstar},
we conclude that $W\in\LIM\left(\Gamma_{1},\Gamma_{2},\Gamma_{3},\ldots\right)$.

The direction from~\ref{enu:Tired1} to~\ref{enu:Tired2} follows
from Corollary~\ref{cor:homeomorphism}. It remains to show that~\ref{enu:Tired2}
implies~\ref{enu:Tired3}. Let $\mathbf{a}\in[0,1]^{q}$ be a vector
whose entries sum up to~1, $\epsilon>0$ and $\mathcal{P}=\left\{ P_{i}\right\} _{i=1}^{q}$
be a partition of $\Omega$ with the property that $\nu(P_{i})=a_{i}$.
Recall that we have fixed a collection $\left\{ A_{r}\right\} _{r\in\mathbb{N}}$
of measurable subsets of $\left[0,1\right]$ that is dense in the
measure algebra and that defines the Hausdorff metric $d_{\mathrm{w^{*}}}^{\mathrm{Hf}}$
on $K(\GRAPHONSPACE)$ by~(\ref{eq:exampleweakstarmetric}). It follows
that there is a sequence $\left\{ A_{r_{i}}\right\} _{i=1}^{q}$ such
that $\sum_{i=1}^{q}\lambda\left(P_{i}\triangle A_{r_{i}}\right)<\epsilon$.
Let $W,U\in\mathcal{W}$ be such that $d_{\mathrm{w^{*}}}^{\mathrm{Hf}}\left(\left\langle U\right\rangle ,\left\langle W\right\rangle \right)<\frac{\epsilon}{2^{2r_{q}}}$.

Let $M\in\mathcal{S}_{{\bf a}}\left(W\right)$. Then there are measure
preserving bijections $\pi,\phi$ such that $M\left(i,j\right)=\int_{P_{i}\times P_{j}}W^{\phi}=\int_{\phi^{-1}\left(P_{i}\right)\times\phi^{-1}\left(P_{j}\right)}W$
and $d_{w^{*}}\left(\left(W^{\phi}\right)^{\Join\mathcal{P}},U^{\pi}\right)<\frac{\epsilon}{2^{2r_{q}}}$.
Define $N\in\mathcal{S}_{{\bf a}}\left(U\right)$ as $N\left(i,j\right)=\int_{P_{i}\times P_{j}}U^{\pi}=\int_{\pi^{-1}\left(P_{i}\right)\times\pi^{-1}\left(P_{j}\right)}U$.
This gives

\begin{align*}
\left\Vert M-N\right\Vert _{1} & =\sum_{i,j=1}^{q}\left|\int_{P_{i}\times P_{j}}\left(W^{\phi}\right)^{\Join\mathcal{P}}-U^{\pi}\right|\le\sum_{i,j=1}^{q}\left|\int_{A_{r_{i}}\times A_{r_{j}}}\left(W^{\phi}\right)^{\Join\mathcal{P}}-U^{\pi}\right|+\epsilon\\
 & \le2^{2r_{q}}\sum_{i,j=1}^{q}2^{-\left(r_{i}+r_{j}\right)}\left|\int_{A_{r_{i}}\times A_{r_{j}}}\left(W^{\phi}\right)^{\Join\mathcal{P}}-U^{\pi}\right|+\epsilon\le2^{2r_{q}}d_{\mathrm{w^{*}}}^{\mathrm{Hf}}\left(\left\langle U\right\rangle ,\left\langle W\right\rangle \right)+\epsilon\\
 & \le2\epsilon.
\end{align*}
This shows that if $d_{\mathrm{w^{*}}}^{\mathrm{Hf}}\left(\left\langle U\right\rangle ,\left\langle W\right\rangle \right)<\frac{\epsilon}{2^{2r_{q}}}$,
then $d_{1}^{\mathrm{Hf}}\left(\mathcal{S}_{{\bf a}}\left(U\right),\mathcal{S}_{{\bf a}}\left(W\right)\right)\le2\epsilon$.
Consequently, \ref{enu:Tired2} implies~\ref{enu:Tired3}.
\end{proof}

\section{Properties of the structuredness (quasi)order\label{sec:BasicPropertiesOfStructurdness}}

Above, we obtained properties of the structuredness (quasi)order $\preceq$
that were needed for our abstract proof of Theorem~\ref{thm:hyperspaceANDcutDISTANCE}.
In this section, we establish further properties of $\preceq$. In
Lemma~\ref{lem:ORDER} we prove that $\preceq$ is actually a closed
order on $\UNLABELLEDGRAPHONSPACE$. In Corollary~\ref{cor:chains},
we prove that $\preceq$-increasing/decreasing chains of graphons
are cut distance convergent. In Corollary~\ref{cor:WhichHyperspaceElementsAreEnvelopes},
we characterize the elements of $K(\GRAPHONSPACE)$ that are envelopes
of graphons. Finally, in Proposition~\ref{cor:WhichHyperspaceElementsAreEnvelopes},
we characterize $\preceq$-minimal and $\preceq$-maximal elements.
This last-mentioned result is just starting point of investigating
the structure of the poset $\preceq$, which we leave open.

\medskip{}
Let us first prove that $\preceq$ is actually an order modulo weak
isomorphism.
\begin{lem}
\label{lem:ORDER}The relation $\preceq$ on the space $\UNLABELLEDGRAPHONSPACE$
is an order, and as a subset of $\UNLABELLEDGRAPHONSPACE\times\UNLABELLEDGRAPHONSPACE$
it is closed.
\end{lem}

\begin{proof}
Since by Corollary~\ref{cor:homeomorphism} the space $\left(\UNLABELLEDGRAPHONSPACE,\delta_{\Box}\right)$
is homeomorphic to some closed subspace of $K(\GRAPHONSPACE)$ and
the relation $\preceq$ is interpreted as $\subseteq$ on $K(\GRAPHONSPACE)$
it is enough to verify the properties for the relation $\subseteq$
on $K(\GRAPHONSPACE)$. But both properties are trivially satisfied
for the relation $\subseteq$. 
\end{proof}
Next, we turn our attention to finding upper and lower bounds with
respect to $\preceq$. Let us first give an auxiliary result, which
is then utilized in Corollaries~\ref{cor:chains} and~\ref{cor:WhichHyperspaceElementsAreEnvelopes}.
\begin{prop}
\label{prop:directed}
\begin{enumerate}[label=(\alph*)]
\item \label{enu:upper-directed}Suppose that $P\subseteq\GRAPHONSPACE$
is upper-directed in the structuredness order, i.e., for every $U_{0},U_{1}\in P$
there is $V\in P$ such that $U_{0},U_{1}\preceq V$. Then there is
a graphon $W\in\GRAPHONSPACE$ such that $W$ is the supremum of $P$
with respect to $\preceq$.
\item \label{enu:down-directed}Suppose that $P\subseteq\GRAPHONSPACE$
is down-directed in the structuredness order, i.e., for every $U_{0},U_{1}\in P$
there is $V\in P$ such that $V\preceq U_{0},U_{1}$. Then there is
a graphon $W\in\GRAPHONSPACE$ such that $W$ is the infimum of $P$
with respect to $\preceq$.
\end{enumerate}
\end{prop}

\begin{proof}
Let us focus on~\ref{enu:upper-directed}. First of all consider
the set $\langle P\rangle=\left\{ \langle U\rangle:U\in P\right\} $.
This set is upper-directed with respect to $\subseteq$ in $K(\GRAPHONSPACE)$.
Let $K$ be the weak{*} closure of $\bigcup_{U\in P}\langle U\rangle$.
Clearly, $K\in K(\GRAPHONSPACE)$. Further, $K$ is the supremum of
$\langle P\rangle$ with respect to $\subseteq$ on $K(\GRAPHONSPACE)$.
To finish the proof, we only need to show that there exists $W\in\GRAPHONSPACE$
such that $K=\langle W\rangle$. Consider a countable set $P_{0}\subseteq P$
such that $\langle P_{0}\rangle$ is dense in $\langle P\rangle$.
This can be done since $K(\GRAPHONSPACE)$ is separable metrizable
by Fact~\ref{fact:VietorisCompact}. Take some enumeration $U_{1},U_{2},...$
of $P_{0}$. Define inductively an increasing chain $\Gamma_{1},\Gamma_{2},...\in P$
such that for every $n\in\mathbb{N}$ we have that $U_{1},...,U_{n},\Gamma_{n-1}\preceq\Gamma_{n}$.
This can be done since $P$ is upper-directed. Since $\Gamma_{1}\preceq\Gamma_{2}\preceq\Gamma_{3}\preceq\ldots$,
we have $\LIM\left(\Gamma_{1},\Gamma_{2},\Gamma_{3},\ldots\right)=\ACC\left(\Gamma_{1},\Gamma_{2},\Gamma_{3},\ldots\right)$.
Indeed, whenever we take $\Gamma\in\ACC\left(\Gamma_{1},\Gamma_{2},\Gamma_{3},\ldots\right)$,
say $\Gamma_{n_{1}}^{\pi_{n_{1}}},\Gamma_{n_{2}}^{\pi_{n_{2}}},\Gamma_{n_{3}}^{\pi_{n_{3}}},\ldots\CUTNORMCONV\Gamma$,
then for each index $i$ in the interval $(n_{k},n_{k+1})$, we can
use that $\Gamma_{n_{k}}\preceq\Gamma_{i}$ to approximate $\Gamma_{n_{k}}^{\pi_{n_{k}}}$
by some version $\Gamma_{i}^{\pi_{i}}$ of $\Gamma_{i}$ (with a vanishing
error as $i\rightarrow\infty$). With the gaps $(n_{k},n_{k+1})$
filled-in this way, we have $\Gamma_{n_{1}}^{\pi_{n_{1}}},\Gamma_{1+n_{1}}^{\pi_{1+n_{1}}},\Gamma_{2+n_{1}}^{\pi_{2+n_{1}}},\ldots\CUTNORMCONV\Gamma$,
and consequently $\Gamma\in\LIM\left(\Gamma_{1},\Gamma_{2},\Gamma_{3},\ldots\right)$,
as we wanted. Moreover $K=\LIM\left(\Gamma_{1},\Gamma_{2},\Gamma_{3},\ldots\right)$
because $\langle P_{0}\rangle$ is dense in $\langle P\rangle$. By
Lemma~\ref{lem:LIMACCcontainsMAX}, we can pick a $\preceq$-maximal
element $W$ of $\LIM\left(\Gamma_{1},\Gamma_{2},\Gamma_{3},\ldots\right)$.
Now, we have that $K=\langle W\rangle$.

The proof of~\ref{enu:down-directed} is similar. The only difference
is that the desired infimum is of the form $K=\bigcap_{U\in P}\left\langle U\right\rangle $
and we inductively build a decreasing sequence. 
\end{proof}
Along very similar lines, we can prove that $\preceq$-increasing/decreasing
chains of graphons are cut distance convergent.
\begin{cor}
\label{cor:chains}
\begin{enumerate}[label=(\alph*)]
\item \label{enu:chainINCREASING-1}Suppose that $W_{1}\preceq W_{2}\preceq W_{3}\preceq\ldots$
is a sequence of graphons. Then this sequence is cut distance convergent.
\item \label{enu:chainDECREASING-1}Suppose that $W_{1}\succeq W_{2}\succeq W_{3}\succeq\ldots$
is a sequence of graphons. Then this sequence is cut distance convergent.
\end{enumerate}
\end{cor}

\begin{proof}
Suppose that $W_{1}\preceq W_{2}\preceq W_{3}\preceq\ldots$. Then
we have $\LIM\left(W_{1},W_{2},W_{3},\ldots\right)=\ACC\left(W_{1},W_{2},W_{3},\ldots\right)$.
By Lemma~\ref{lem:LIMACCcontainsMAX}, we can pick a $\preceq$-maximal
element $W$ of $\LIM\left(W_{1},W_{2},W_{3},\ldots\right)$. Now,
by Theorem~\ref{thm:hyperspaceANDcutDISTANCE} we have that $W_{n}\CUTDISTCONV W$.

The proof for a decreasing sequence is the same.
\end{proof}
Next, we characterize the elements of $K(\GRAPHONSPACE)$ that are
envelopes of graphons.
\begin{cor}
\label{cor:WhichHyperspaceElementsAreEnvelopes}Let $K\in K(\GRAPHONSPACE)$.
Then there exists $W\in\GRAPHONSPACE$ such that $K=\langle W\rangle$
if and only if $K$ is upper-directed (for every $U_{0},U_{1}\in K$
there is $V\in K$ such that $U_{0},U_{1}\preceq V$) and downwards
closed (for every $U\in K$ and $V\preceq U$ we have $V\in K$). 
\end{cor}

\begin{proof}
Suppose first that $K=\langle W\rangle$. Then for every $U_{0},U_{1}\in K$,
we have $U_{0},U_{1}\preceq W$. That is, $K$ is upper-directed.
Suppose next that $U\in K$ and $V\preceq U$. Let $\epsilon>0$ be
arbitrary. Since $V\preceq U$, there exists a version $U^{\pi}$
of $U$ such that $d_{\mathrm{w}^{*}}\left(V,U^{\pi}\right)<\frac{\epsilon}{2}$.
Since $U\in K$, we have $U\preceq W$. Thus, also $U^{\pi}\preceq W$.
Therefore, there exists a version $W^{\theta}$ of $W$ such that
$d_{\mathrm{w}^{*}}\left(U^{\pi},W^{\theta}\right)<\frac{\epsilon}{2}$.
We conclude that $d_{\mathrm{w}^{*}}\left(V,W^{\theta}\right)<\epsilon$.
Since $\epsilon$ was arbitrarily small and since $K$ is weak{*}
closed, we conclude that $V\in K$.

Let us now turn to the other implication. As $K$ is upper-directed
then by Proposition~\ref{prop:directed} we can take its supremum
$W$. Because $K$ is downwards closed we have $\langle W\rangle=K$.
\end{proof}
Last, we characterize the minimal and maximal elements of the structuredness
order; the latter part being suggested to us by László Miklós Lovász. 
\begin{prop}
\label{prop:minimalAndMaximalElements}The minimal elements of the
structuredness order are exactly constant graphons, and for every
graphon $W$ there is a minimal graphon $W_{\min}\preceq W$. The
maximal elements of the structuredness order are exactly 0-1 valued
graphons, and for every graphon $W$ there is a maximal graphon $W_{\max}\succeq W$.
\end{prop}

\begin{proof}
The first part follows directly from the fact that an envelope of
any graphon contains a constant graphon (see Lemma~\ref{lem:basicpropofenvelopes}\ref{enu:averagingINSIDE}),
and also the graphon itself. For the second part we at first prove
that only 0-1 valued graphons can be maximal. 

Suppose that $W$ is a graphon such that its value is neither $0$
nor $1$ on a set of positive measure. This implies that there is
an $\varepsilon>0$ such that $W$ has values between $\varepsilon$
and $1-\varepsilon$ on a set of positive measure. Now consider a
map $\varphi\colon[0,1]\rightarrow[0,1]$ such that $\varphi(x)=2x$
for $0\leq x\leq\frac{1}{2}$ and $\varphi(x)=2x-1$ for $\frac{1}{2}<x\leq1$.
The graphon $W^{\varphi}$ contains four copies of $W$ scaled by
a factor of one half. Let $B\subseteq[0,1]^{2}$ be the set, on which
$W^{\varphi}$ takes values between $\varepsilon$ and $1-\varepsilon$.
Let $\widetilde{W}$ be a graphon such that $\widetilde{W}=W^{\varphi}+\varepsilon$
for $x,y\in\left([0,\frac{1}{2}]^{2}\;\cup\;[\frac{1}{2},1]^{2}\right)\cap B$
, $\widetilde{W}=W^{\varphi}-\varepsilon$ for $x,y\in\left([0,\frac{1}{2}]\times[\frac{1}{2},1]\;\cup\;[\frac{1}{2},1]\times[0,\frac{1}{2}]\right)\cap B$
and $\widetilde{W}=W$ otherwise. The values of $\widetilde{W}$ are
bounded by $0$ and $1$ and $W$ and $\widetilde{W}$ are not weakly
isomorphic (compare $\INT_{f}(W)$ and $\INT_{f}(\widetilde{W})$
for any strictly convex function $f$). Moreover, we claim that $W\in\left\langle \widetilde{W}\right\rangle $.
To see this, one can construct a sequence of measure preserving almost-bijections
$\psi_{1},\psi_{2},\dots$ , defined as $\psi_{n}(x)=\frac{\lfloor2nx\rfloor}{2n}+x$
for $0\leq x\leq\frac{1}{2}$ and $\psi_{n}(x)=\frac{\lfloor2nx\rfloor-2n+1}{2n}+x$
for $\frac{1}{2}\leq x\leq1$, that interlace the two intervals $[0,\frac{1}{2}]$
and $[\frac{1}{2},1]$ and thus serve as an approximation of $\varphi$.
The fact that $\widetilde{W}{}_{1}^{\psi_{1}^{-1}},\widetilde{W}_{2}^{\psi_{2}^{-1}},\dots\WEAKCONV W$
can be seen directly from the definition of weak{*} convergence.

Next, we prove that all 0-1 graphons are maximal. Indeed, let $W$
be a 0-1 valued graphon, and suppose that there exists some graphon
$U$ such that $U\succ W$. Then for the measures $\boldsymbol{\Phi}_{U}$
and $\boldsymbol{\Phi}_{W}$ we have that $\boldsymbol{\Phi}_{W}$
is strictly flatter than $\boldsymbol{\Phi}_{U}$ by Proposition~\ref{prop:flatter}.
This contradicts Lemma~\ref{lem:01measuresextremal}.

Finally, let $W$ be an arbitrary graphon. Consider the set $\mathcal{P}$
of all graphons $P\succeq W$. Then every chain in $\mathcal{P}$
has a supremum in the structuredness order by Proposition~\ref{prop:directed}\ref{enu:upper-directed}.
Therefore, we can apply Zorn's lemma to conclude that there is a maximal
graphon $W_{\max}\succeq W$.
\end{proof}
\begin{rem}
Let us take $W\equiv\frac{1}{2}$ and a sequence $W_{n}$ of graphons
corresponding to Erd\H{o}s\textendash Rényi random graphs $\mathbb{G}(n,\frac{1}{2})$.
By Proposition~\ref{prop:minimalAndMaximalElements}, $W$ is a $\prec$-minimal
elements, while each $W_{n}$ is a $\prec$-maximal element. Yet,
it is well-known that $W_{n}\CUTDISTCONV W$ almost surely. This example
shows that even a cut distance convergent sequence can be fairly <<separated>>
from the limit point with respect to the structurdness order. (Note
also, that while $W$ is a $\prec$-minimal elements, and each $W_{n}$
is a $\prec$-maximal element, $W_{n}$ and $W$ are typically incomparable
in the structuredness order, since the density of $W_{n}$ is typically
not exactly $\frac{1}{2}$.)
\end{rem}

In the example below we show that the structuredness order does not
have meet and joins in general.
\begin{example}
We shall construct graphons $W_{1},W_{2},U_{1},U_{2}$ such that we
have $U_{1},U_{2}\preceq W_{1},W_{2}$, but there is no graphon $V$
such that $U_{1},U_{2}\preceq V\preceq W_{1},W_{2}$. For a fixed
$\varepsilon>0$ we define the four graphons (on the unit square,
denoting the Lebesgue measure on $[0,1]$ as $\nu$) as follows (the
definitions should be clear from Figure~\ref{fig:insersection_of_envelopes}).
\begin{figure}
\includegraphics[width=1\textwidth]{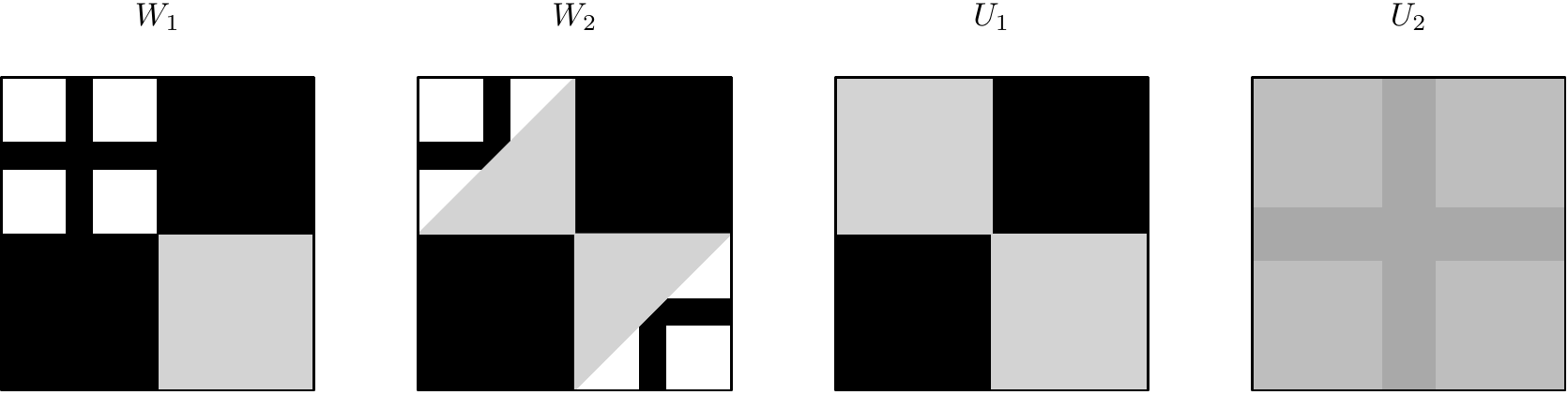}

\caption{\label{fig:insersection_of_envelopes}The four graphons $W_{1},W_{2},U_{1},U_{2}$
witnessing that an intersection of two envelopes is not necessarily
an envelope. }
\end{figure}
\begin{enumerate}
\item Define $W_{1}(x,y)=1$ if and only if $(x,y)\in[\frac{1}{2},1]\times[0,\frac{1}{2}]\;\cup\;[0,\frac{1}{2}]\times[\frac{1}{2},1]\;\cup\;[\frac{1}{4}-\frac{\varepsilon}{2},\frac{1}{4}+\frac{\varepsilon}{2}]\times[0,1]\;\cup\;[0,1]\times[\frac{1}{4}-\frac{\varepsilon}{2},\frac{1}{4}+\frac{\varepsilon}{2}]$.
Moreover, $W_{1}=4(\varepsilon-\varepsilon^{2})$ for $(x,y)\in[\frac{1}{2},1]^{2}$
and is zero otherwise. 
\item Define $W_{2}$ as $W_{1}$ but switch its values on the triangle
with vertices $[0,\frac{1}{2}],[\frac{1}{2},\frac{1}{2}],[\frac{1}{2},0]$
with the triangle with vertices $[\frac{1}{2},1],[1,1],[1,\frac{1}{2}]$.
Note that $\int_{[0,\frac{1}{2}]^{2}}W_{1}=\int_{[\frac{1}{2},1]^{2}}W_{1}=\int_{[0,\frac{1}{2}]^{2}}W_{2}=\int_{[\frac{1}{2},1]^{2}}W_{2}$. 
\item Define $U_{1}=4(\varepsilon-\varepsilon^{2})$ for $(x,y)\in[0,\frac{1}{2}]^{2}\cup[\frac{1}{2},1]^{2}$
and $1$ otherwise. Note that $U_{1}\preceq W_{1},W_{2}$ (this follows
by applying Lemma~\ref{lem:basicpropofenvelopes}\ref{enu:averagingINSIDE}
twice just for the top-left and bottom-right subgraphons of the graphons~$W_{1}$
and~$W_{2}$). 
\item Finally, define $U_{2}$ to be such that 
\[
U_{2}(x,y)=\frac{1}{4}\left(W_{1}\left(\frac{x}{2},\frac{y}{2}\right)+W_{1}\left(\frac{x+1}{2},\frac{y}{2}\right)+W_{1}\left(\frac{x}{2},\frac{y+1}{2}\right)+W_{1}\left(\frac{x+1}{2},\frac{y+1}{2}\right)\right).
\]
We have again $U_{2}\preceq W_{1},W_{2}$ (consider a sequence of
bijections interlacing the two intervals $[0,\frac{1}{2}]$ and $[\frac{1}{2},1]$,
as in the proof of Proposition~\ref{prop:minimalAndMaximalElements}). 
\end{enumerate}
We are now ready to get to the main two properties of this example.
\begin{claim*}[Claim A]
If $V$ is a graphon such that $U_{1}\preceq V$, then $V$ contains
a set $J$ such that $V$ is~1 almost everywhere on $J\times\left([0,1]\setminus J\right)$.
\end{claim*}
\begin{proof}
The argument is similar to the proof of~\cite[Lemma 3]{DoHl:Polytons}
and~\cite[Lemma 2.3]{HlHuPi:Komlos}. Observe that in $U_{1}$, the
values on the rectangle $\left[0,\frac{1}{2}\right]\times\left[\frac{1}{2},1\right]$
are~1 everywhere. Since there exists a sequence of measure preserving
transformations $\pi_{1},\pi_{2},\dots$ such that $V^{\pi_{n}}\WEAKCONV U_{1}$,
taking $J_{n}:=\pi_{n}\left(\left[0,\frac{1}{2}\right]\right)$ yields
sets of measure $\frac{1}{2}$ with the property that 
\begin{equation}
\lim_{n}\int_{J_{n}\times\left(\left[0,1\right]\setminus J_{n}\right)}V=\int_{\left[0,\frac{1}{2}\right]\times\left[\frac{1}{2},1\right]}U_{1}=\frac{1}{4}\;.\label{eq:ManyHodny}
\end{equation}
Let us pass to a subsequence so that the sequence of indicators $\left(\mathbf{1}_{J_{n}}\right)_{n}$
converges weak{*}, say to a function $g:[0,1]\rightarrow[0,1]$. Note
that then the sequence $\left(\mathbf{1}_{[0,1]\setminus J_{n}}\right)_{n}$
converges weak{*} to $1-g$. 

We claim that $V$ is~1 almost everywhere on $\left(\mathrm{supp}\,g\right)\times\left(\mathrm{supp}\,(1-g)\right)$.
Before proving this, let us explain, why this gives the desired set
$J$. To this end, note that since $\nu(J_{n})=\frac{1}{2}$, weak{*}
convergence implies $\int g=\frac{1}{2}$. In particular, the measure
of the support of $g$ is at least $\frac{1}{2}$ while at the same
time the measure of points where $g\equiv1$ is at most $\frac{1}{2}$.
So, we take $J$ to be an arbitrary set of measure exactly $\frac{1}{2}$
contained in the former set and containing the latter set. It is clear
that all the required properties are satisfied. 

So, let us prove that $V$ is~1 almost everywhere on $\left(\mathrm{supp}\,g\right)\times\left(\mathrm{supp}\,(1-g)\right)$.
It suffices to prove that for each $a>0$, the Lebesgue measure of
those pairs $(x,y)\in g^{-1}\left((a,1]\right)\times g^{-1}\left([0,1-a)\right)$
for which $V(x,y)<1-a$ is zero. So, suppose that this fails for some
$a>0$. By basic properties of measure, we can then find two sets
$X\subseteq g^{-1}\left((a,1]\right)$, $Y\subseteq g^{-1}\left([0,1-a)\right)$
of positive measure such that 
\begin{equation}
\nu^{\otimes2}\left(\left\{ (x,y)\in X\times Y:V(x,y)\ge1-a\right\} \right)<\frac{a^{3}}{100}\cdot\nu(X)\nu(Y)\;.\label{eq:XY}
\end{equation}
Now, observe that since $g$ is a weak{*} limit of $\left(\mathbf{1}_{J_{n}}\right)_{n}$
and $X\subseteq g^{-1}\left((a,1]\right)$, we have that for all sufficiently
large $n$, $\nu(X\cap J_{n})>a\cdot\nu(X)$. By a similar reasoning,
$\nu\left(Y\cap\left(\left[0,1\right]\setminus J_{n}\right)\right)>a\cdot\nu(Y)$
for all sufficiently large $n$. To say this in words, a substantial
(non-vanishing) part of the set $J_{n}$ is inside the set $X$ and
a substantial part of the set $\left[0,1\right]\setminus J_{n}$ is
inside the set $Y$, and by~(\ref{eq:XY}), $V$ is very far from
being close to~1 on $X\times Y$. This contradicts~(\ref{eq:ManyHodny}),
which is saying that $V$ must be very close to~1 on most of $J_{n}\times\left(\left[0,1\right]\setminus J_{n}\right)$.
\end{proof}
\begin{claim*}[Claim B]
There does not exist any graphon $V$ such that $U_{1},U_{2}\preceq V\preceq W_{1},W_{2}$. 
\end{claim*}
\begin{proof}
Assume that there is a graphon $V$ such that $U_{1},U_{2}\preceq V\preceq W_{1},W_{2}$.
Let the set $J$ be given from Claim~A. Without loss of generality
assume that $J=[0,\frac{1}{2}]$. Now, we turn our attention to the
graphons $W_{i}$, $i=1,2$. From the fact that there is a sequence
of measure preserving transformations $\varphi_{1},\varphi_{2},\dots$
such that $W_{1}^{\varphi_{1}},W_{1}^{\varphi_{2}},\dots\WEAKCONV V$
and, thus, $\lim_{n\rightarrow\infty}\int_{\left[0,\frac{1}{2}\right]\times\left[\frac{1}{2},1\right]}W_{1}^{\varphi_{n}}=\int_{\left[0,\frac{1}{2}\right]\times\left[\frac{1}{2},1\right]}V$,
we obtain that for any $\delta>0$ there is $n$ sufficiently large
such that either $\nu\left(\varphi_{n}\left([0,\frac{1}{2}]\right)\cap[0,\frac{1}{2}]\right)\leq\delta$
or $\nu\left(\varphi_{n}\left([0,\frac{1}{2}]\right)\cap[0,\frac{1}{2}]\right)\geq\frac{1}{2}-\delta$.
Indeed, suppose that $\frac{1}{2}-\delta>\nu\left(\varphi_{n}\left([0,\frac{1}{2}]\right)\cap[0,\frac{1}{2}]\right)>\delta$
and observe that the density of $W_{1}^{\varphi_{n}}$ is equal to
$4(\varepsilon-\varepsilon^{2})$ on a subset of $[0,\frac{1}{2}]\times[\frac{1}{2},1]$
of measure at least $\delta^{2}$. Now it suffices to recall that
$\int_{[0,\frac{1}{2}]\times[\frac{1}{2},1]}W_{1}^{\varphi_{n}}\rightarrow\nu\left([0,\frac{1}{2}]\times[\frac{1}{2},1]\right)=\frac{1}{4}$
which implies that the assumption is false for large enough $n$.
From the fact that either $\nu\left(\varphi_{n}\left([0,\frac{1}{2}]\right)\cap[0,\frac{1}{2}]\right)\leq\delta$
or $\nu\left(\varphi_{n}\left([0,\frac{1}{2}]\right)\cap[0,\frac{1}{2}]\right)\geq\frac{1}{2}-\delta$
we conclude that actually (after passing to a subsequence) some versions
of $W_{1}\cap[0,\frac{1}{2}]^{2},W_{1}\cap[0,\frac{1}{2}]^{2},\dots$
converge weak{*} to either $V\cap[0,\frac{1}{2}]^{2}$ or $V\cap[\frac{1}{2},1]^{2}$
(redefine $\varphi_{n}$ on a set of small measure such that it maps
$[0,\frac{1}{2}]$ either onto $[0,\frac{1}{2}]$ or onto $[\frac{1}{2},1]$
to get the required measure preserving bijections). Without loss of
generality assume that some versions of $W_{1}\cap[\frac{1}{2},1]^{2},W_{1}\cap[\frac{1}{2},1]^{2},\dots$
converge weak{*} to $V\cap[\frac{1}{2},1]^{2}$. Similarly, we get
that there are versions of $W_{2}\cap[0,\frac{1}{2}]^{2},W_{2}\cap[0,\frac{1}{2}]^{2},\dots$
converging weak{*} to $V\cap[0,\frac{1}{2}]^{2}$ (the other case
when the versions $W_{2}\cap[\frac{1}{2},1]^{2},W_{2}\cap[\frac{1}{2},1]^{2},\dots$
converge weak{*} to $V\cap[0,\frac{1}{2}]^{2}$ can be treated in
the same way). 

Notice that 
\begin{equation}
\int_{[0,1]\times[\frac{1}{2}-\varepsilon,\frac{1}{2}+\varepsilon]}U_{2}=(2\varepsilon)\cdot1\cdot\frac{3}{4}+o(\varepsilon)=\frac{3}{2}\varepsilon+o(\varepsilon)\;.\label{eq:U2}
\end{equation}

On the other hand, 
\begin{align}
\sup_{\nu\left(C\right)=2\varepsilon}\int_{[0,1]\times C}V & =\sup_{A\subseteq[0,\frac{1}{2}],B\subseteq[\frac{1}{2},1],\nu\left(A\cup B\right)=2\varepsilon}\left(\int_{[0,1]\times A}V+\int_{[0,1]\times B}V\right)\nonumber \\
 & =2\varepsilon\cdot\frac{1}{2}+\sup_{A\subseteq[0,\frac{1}{2}],B\subseteq[\frac{1}{2},1],\nu\left(A\cup B\right)=2\varepsilon}\left(\int_{[0,\frac{1}{2}]\times A}V+\int_{[\frac{1}{2},1]\times B}V\right)\nonumber \\
 & \leq\varepsilon+\sup_{A\subseteq[0,\frac{1}{2}],B\subseteq[\frac{1}{2},1],\nu\left(A\cup B\right)=2\varepsilon}\left(\int_{[0,\frac{1}{2}]\times A}W_{2}+\int_{[\frac{1}{2},1]\times B}W_{1}\right)\nonumber \\
 & =\varepsilon+\left(\frac{1}{4}\varepsilon+o(\varepsilon)\right)+o(\varepsilon)\nonumber \\
 & =\frac{5}{4}\varepsilon+o(\varepsilon),\label{eq:54}
\end{align}
which, for $\varepsilon$ small enough, is smaller than the appropriate
value for $U_{2}$ appearing in~(\ref{eq:U2}). Now, we can conclude
that $V\not\succeq U_{2}$. Indeed, suppose that $V\succeq U_{2}$.
Then for every $a>0$ there is a version $V^{\pi}$ of $V$ such that
$\DIST_{\mathrm{w}^{*}}\left(V^{\pi},U_{2}\right)<a$. In particular,
there exists a version $V^{\pi}$ such that $\left|\int_{[0,1]\times[\frac{1}{2}-\varepsilon,\frac{1}{2}+\varepsilon]}V^{\pi}-U_{2}\right|<\frac{1}{100}$.
Taking $C=\pi^{-1}\left([\frac{1}{2}-\varepsilon,\frac{1}{2}+\varepsilon]\right)$,
(\ref{eq:54}) contradicts~(\ref{eq:U2}).
\end{proof}
\end{example}

\section{Conclusion}

Inspired by previous work~\cite{DH:WeakStar}, we created a comprehensive
theory of approaching the cut distance convergence via the weak{*}
topology. The main results, Theorem~\ref{thm:hyperspaceANDcutDISTANCEsimpler}
and Theorem~\ref{thm:subsequenceLIMACC}, say that in each sequence
$\Gamma_{1},\Gamma_{2},\Gamma_{3},\ldots$ there exists a subsequence
$\Gamma_{n_{1}},\Gamma_{n_{2}},\Gamma_{n_{3}},\ldots$ such that $\LIM\left(\Gamma_{n_{1}},\Gamma_{n_{2}},\Gamma_{n_{3}},\ldots\right)=\ACC\left(\Gamma_{n_{1}},\Gamma_{n_{2}},\Gamma_{n_{3}},\ldots\right)$,
and that the latter property is equivalent to the cut distance convergence
of the said subsequence. It follows from the equivalence established
in Proposition~\ref{prop:MultiwayEquivalentLIMACC} that Theorem~\ref{thm:hyperspaceANDcutDISTANCEsimpler}
can be regarded as the functional-analytic formulation of Theorem~\ref{thm:Multiway},
which was originally established in~\cite{MR2925382}. So, while
each of the above mentioned theorems could have been obtained from
previously existing results, we believe that the weak{*} approach
we introduce provides an important perspective on the theory of graph
limits. For example, Corollary~\ref{cor:chains} offers families
of convergent graphon sequences, which are very natural with this
perspective but would be even difficult to define previously. It turns
out that many of these concepts extend (in a nontrivial way) to hypergraphons,
and this is currently work in progress.

Our initial motivation was to build a theory rather than to solve
any particular problem, quite contrary to the usual perception of
combinatorics (see \cite{MR1754768}). That said, we already see that
this abstract theory has its fruits on the problem-solving side: in~\cite{DGHRR:Parameters}
we use it to prove that a connected graph is weakly norming if and
only if it is step Sidorenko, and that if a graph is norming then
it is step forcing (see~\cite[Remark 3.20]{DGHRR:Parameters} for
an explanation of the role of the the weak{*} theory in these results).
We see a potential for more applications in the theory of hypergraphons.

The weak{*} approach may provide alternative proofs of various generalizations
of Theorem~\ref{thm:compactness}, such as Banach space valued graphons~\cite{KLS:MultigraphLimits},
representations of limits of sparse graphs using measures~\cite{MR3979225},
and an $L^{p}$-approach to limits of sparse graphs~\cite{BHHZ:LPgraphlimits}.
To indicate the plausibility of this approach in the setting $L^{p}$-limits
(which seems the most feasible one), we recall that the corresponding
notion of weak{*} convergence in this setting is based on tests against
all $L^{q}$-functions (where $\frac{1}{p}+\frac{1}{q}=1$). This
would of course change the definition of $\LIM(\Gamma_{1},\Gamma_{2},\Gamma_{3},\ldots)$,
and all related definitions. Working out this program in detail is
currently in progress.

\section*{Acknowledgments}

We would like to thank László Miklós Lovász, Daniel Král and Oleg
Pikhurko for their useful suggestions, and Frederik Garbe and Ond\v{r}ej
K\r{u}rka for comments on a preliminary version of this paper. We
thank a referee for his or her comments.

\bibliographystyle{plain}
\bibliography{bibl}

\end{document}